\documentclass{amsart}
 
\usepackage{url}
\usepackage{amssymb,amsmath,txfonts,mathrsfs,amsthm}
\usepackage{comment}
\usepackage[neveradjust]{paralist}
\usepackage{mathtools}
\usepackage{multirow}
\usepackage[outline]{contour}
\contourlength{1.2pt}
\usepackage{tikz}
\usetikzlibrary{intersections,arrows.meta,calc,quotes,math,decorations.pathreplacing,decorations.markings,cd,arrows,positioning,fit,matrix,
shapes.geometric,external,decorations.pathmorphing,backgrounds,circuits,circuits.ee.IEC,shapes}
\usepackage[a4paper,top=3cm,bottom=3cm,inner=3cm,outer=3cm]{geometry}
\usepackage[foot]{amsaddr}
\usepackage{graphicx}
\usepackage{stackrel}

\usepackage{xcolor}
\usepackage{framed,color}
\usepackage{float}

\definecolor{shadecolor}{rgb}{1,0.8,0.3}
\definecolor{linkcolor}{rgb}{0.5,0,0}
\definecolor{mycitecolor}{rgb}{0,0,0.7}
\definecolor{myrefcolor}{rgb}{0,0,0.7}
\definecolor{hyperrefcolor}{rgb}{0.5,0,0}

\definecolor{lblue}{rgb}{0,250,255}
\definecolor{lyellow}{rgb}{20,20,0}
\definecolor{purple}{HTML}{6e0280}
\definecolor{darkorange}{HTML}{b55a2a}
\definecolor{darkgreen}{rgb}{0,100,0}
\definecolor{lightblue}{rgb}{0.8,0.9,1}
\definecolor{darkgreen}{rgb}{0,0.5,0.2}
\definecolor{darkblue}{rgb}{0,0,100}
\definecolor{darkred}{HTML}{d40000}

\newcommand{\re}[1]{\textcolor{darkred}{#1}}
\newcommand{\bl}[1]{\textcolor{blue}{#1}}
\newcommand{\p}[1]{\textcolor{purple}{#1}}

\usepackage[
    pagebackref,
    colorlinks, 
    citecolor=blue, 
    urlcolor=linkcolor, 
    final, 
    hyperindex, 
    linkcolor = blue
]{hyperref}

\newcommand{\doi}[1]{DOI: \url{#1}}

\usepackage[draft]{fixme}
\usepackage[capitalize]{cleveref}
\crefname{equation}{}{}

\usepackage{dcolumn}
\newcolumntype{2}{D{.}{}{2.0}}

\crefname{defn}{Definition}{Definitions}
\crefname{thm}{Theorem}{Theorems}
\crefname{lem}{Lemma}{Lemmas}
\crefname{eg}{Example}{Examples}

\newcommand{\coker}{\mathop{\mathrm{coker}}}
\newcommand{\Eq}{\mathop{\mathrm{eq}}}
\newcommand{\coeq}{\mathop{\mathrm{coeq}}}
\newcommand{\simrightarrow}{\xrightarrow{\raisebox{-3pt}[0pt][0pt]{\ensuremath{\sim}}}} 

\newcommand{\N}{\mathbb{N}}
\newcommand{\Z}{\mathbb{Z}}
\newcommand{\R}{\mathbb{R}}

\newcommand{\cat}{\mathsf{Cat}}
\newcommand{\Set}{\mathsf{Set}}
\newcommand{\Gph}{\mathsf{Gph}}

\newcommand{\Comm}{\mathsf{Comm}}
\newcommand{\Mon}{\mathsf{Mon}}

\newcommand{\A}{\mathsf{A}}
\newcommand{\B}{\mathsf{B}}
\newcommand{\C}{\mathsf{C}}
\newcommand{\G}{\mathsf{G}}

\newcommand{\K}{\mathsf{K}}

\newcommand{\D}{\mathsf{D}}
\newcommand{\X}{\mathsf{X}}

\newcommand{\Fin}{\mathsf{Fin}}

\newcommand{\F}{\mathcal{F}}
\renewcommand{\P}{\mathcal{P}}
\newcommand{\Q}{\mathcal{Q}}
\newcommand{\Free}{\mathsf{Free}}
\newcommand{\Und}{\mathrm{Und}}
\newcommand{\disc}{\mathrm{disc}}
\newcommand{\Disc}{\mathsf{Disc}}
\renewcommand{\Vert}{\mathrm{vert}}

\newcommand{\bicat}{\mathbf}

\newcommand{\Cat}{\bicat{Cat}}

\newcommand{\double}[1]{\mathbf{\mathbb #1}}
\newcommand{\lCsp}{\double{Csp}}

\newcommand{\lOpen}{\double{Open}}

\newcommand{\lB}{\double{B}}

\newcommand{\lD}{\double{D}}

\newcommand{\define}[1]{{\bf \boldmath{#1}}}

\usetikzlibrary{arrows.meta}

\tikzset{
    regulations/.cd,
    act/.style={-{Stealth}}, 
    rep/.style={-{Bar}}, 
}

\definecolor{rewritecolor}{rgb}{0,.9,1}
\tikzset{rewritenode/.style={shape=circle,fill=rewritecolor,scale=0.25,font=\Huge}}
\tikzset{RWopen/.style={shape=circle,draw=black,fill=white,scale=0.5,font=\Huge}}
\tikzset{RWclosed/.style={shape=circle,fill=black,scale=0.5,font=\Huge}}
\tikzset{CDnode/.style={shape=circle,fill=white,scale=.5}}
\makeatletter
\let\ea\expandafter

\pgfdeclarelayer{edgelayer}
\pgfdeclarelayer{nodelayer}
\pgfsetlayers{edgelayer,nodelayer,main}

\tikzstyle{inarrow}=[->, >=stealth, shorten >=.03cm]
\tikzstyle{empty}=[circle,fill=none, draw=none]
\tikzstyle{inputdot}=[circle,fill=purple,draw=purple, scale=.25]
\tikzstyle{inputarrow}=[->,draw=purple, shorten >=.05cm]
\tikzstyle{simple}=[-,draw=purple,line width=1.000]
\tikzstyle{none}=[inner sep=0pt]

\tikzset{->-/.style={decoration={
  markings,
  mark=at position .5 with {\arrow{>}}},postaction={decorate}}}

\def\mdef#1#2{\ea\ea\ea\gdef\ea\ea\noexpand#1\ea{\ea\ensuremath\ea{#2}}}
\def\alwaysmath#1{\ea\ea\ea\global\ea\ea\ea\let\ea\ea\csname your@#1\endcsname\csname #1\endcsname
  \ea\def\csname #1\endcsname{\ensuremath{\csname your@#1\endcsname}}}

\mdef\fahat{\hat{\fa}}

\newcommand{\op}{\mathrm{op}}

\newcommand{\id}{\mathrm{id}}

\newcommand{\Ob}{\mathrm{Ob}}
\newcommand{\Mor}{\mathrm{Mor}}

\newcommand{\maps}{\colon}

\def\defthm#1#2{%
  \newtheorem{#1}{#2}[section]%
  \expandafter\def\csname #1autorefname\endcsname{#2}%
  \expandafter\let\csname c@#1\endcsname\c@thm}
\newtheorem{thm}{Theorem}[section]

\defthm{cor}{Corollary}
\defthm{prop}{Proposition}
\defthm{lem}{Lemma}
\defthm{conj}{Conjecture}
\defthm{hyp}{Hypothesis}
\defthm{fact}{Fact}
\theoremstyle{definition}
\defthm{defn}{Definition}
\defthm{eg}{Example}
\defthm{notn}{Notation}
\theoremstyle{remark}
\defthm{rmk}{Remark}

\mdef\fchk{\check{f}}

\definecolor{purple(x11)}{rgb}{0.5, 0.0, 0.5}

\title{Motifs and Emergent Feedback in Labeled Graphs}

\author{John\ C.\ Baez$^{1,2}$ and Adittya Chaudhuri$^{3,4}$}
\address{$^1$School of Mathematics, University of Edinburgh, James Clerk Maxwell Building, Peter Guthrie Tait Road, Edinburgh, UK EH9 3FD}
\address{$^2$Department of Mathematics, University of California, Riverside CA, USA 92521}
\address{$^3$ Institute of Computer Science, University of Rostock, Albert-Einstein-Str. 22, 18059 Rostock, Germany}
\address{$^4$ Statistics and Mathematics Unit, Indian Statistical Institute Kolkata, 203 Barrackpore Trunk Road, Kolkata 700108, India}
\email{baez@math.ucr.edu,  chaudhuriadittya@gmail.com}

\begin{document}

\begin{abstract}
In fields ranging from business to systems biology, directed graphs with edges labeled by signs are used to model systems in a simple way: the nodes represent entities of some sort, and an edge indicates that one entity directly affects another either positively or negatively.  Multiplying the signs along a directed path of edges lets us determine indirect positive or negative effects, and if the path is a loop we call this a positive or negative feedback loop.  Here we generalize this to graphs with edges labeled by a monoid, whose elements represent `polarities' possibly more general than simply `positive' or `negative'.  We study three notions of morphism between graphs with labeled edges, each with its own distinctive application: to refine a simple graph into a complicated one, to transform a complicated graph into a simple one, and to find recurring patterns called `motifs'.  We construct three corresponding symmetric monoidal double categories of `open' graphs.   We also study feedback loops using a generalization of the homology of a graph to homology with coefficients in a commutative monoid.  In particular, we describe the emergence of new feedback loops when we compose open graphs using a variant of the Mayer--Vietoris exact sequence for homology with coefficients in a commutative monoid.
\end{abstract}

\maketitle
\section{Introduction}

Graphs with edges labeled by elements of a fixed set are widely used as qualitative models of systems.   For example:
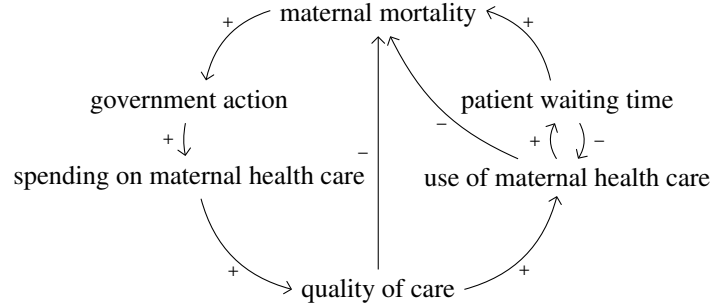
\[
\begin{tikzpicture}[scale=1]
\node (A) at (0,2.5) {\phantom{maternal mortality}};
\node (A') at (0,2.65) {maternal mortality};
\node (B) at (-2.5,1.5) {government action};
\node (C) at (-2.5,0.5) {spending on maternal health care};
\node (D) at (0,-1) {quality of care};
\node (E) at (2.5,0.5) {use of maternal health care};
\node (F) at (2.5,1.5) {patient waiting time};
\path[->,font=\scriptsize,>=angle 90]
(A) edge[bend right] node[above, pos=0.5]{$+$} (B)
(B) edge[bend right=10] node[left, pos=0.5]{$+$} (C)
(C) edge[bend right] node[below, pos=0.5]{$+$}(D)
(D) edge node[left, pos=0.5]{$-$}(A')
(D) edge[bend right] node[below, pos=0.5]{$+$}(E)
(E) edge[bend left=20] node[below, pos=0.5]{$-$} (A')
(E) edge[bend left] node[left, pos=0.5]{$+$}(F)
(F) edge[bend left] node[right, pos=0.5]{$-$}(E)
(F) edge[bend right] node[above, pos=0.5]{$+$}(A);
\end{tikzpicture}
\]
This is a model of how health care can affect maternal mortality,  a tiny simplified fragment of some larger
models in the literature \cite{Lembani2018, SookaRwashana2011}.  Vertices of this graph represent various variables in a system.   Labeled edges indicate how one variable can directly affect another.  The label set in this example is $\{+,-\}$.    An edge labeled by $+$ indicates that one variable has a direct positive effect on another: that is, increasing the former tends to increase the latter, everything else being equal.   Similarly, an edge labeled by $-$ indicates that one variable has a direct negative effect on another.   In fact the set $\{+,-\}$ naturally has the structure of a monoid, with this multiplication table:
\vskip 0.5em
\renewcommand\arraystretch{1.3}
\setlength\doublerulesep{0pt}
\begin{center}
\begin{tabular}{c||c|c}
$\cdot$ & $+$ & $-$ \\
\hline\hline
$+$ & $+$ & $-$ \\ 
\hline
$-$ & $-$ & $+$ \\ 
\end{tabular}
\end{center}
\vskip 0.5em
Other monoids are also commonly used to label edges in diagrams of this sort, and even more could be useful;  we discuss these alternatives throughout the paper.  The elements of such labeling sets are sometimes called `polarities'.

Graphs with edges labeled by elements of $\{+,-\}$ are called `causal loop diagrams' in the modeling tradition known as `system dynamics'.   System dynamics was first thoroughly explained in Forrester's 1961 book \textsl{Industrial Dynamics}, with the main application being to model fluctuations in supply chains \cite{Forrester1961}.  Later Forrester collaborated with the mayor of Boston and applied the ideas to urban planning in his book \textsl{Urban Dynamics} \cite{Forrester1969}.  In 2000, Sterman wrote an influential text \textsl{Business Dynamics} applying these ideas to the internal functioning of a business \cite{Sterman2000}.  Later, system dynamics became widely used in health care and other subjects, and in 2014, Hovmand emphasized the process of building models by gathering and synthesizing information from diverse stakeholders in \textsl{Community Based System Dynamics} \cite{Hovmand2014}.   By now there is a thriving community of systems dynamics practitioners, with an annual international conference, and software that helps them work with causal loop diagrams.  

In a parallel line of development, biologists began using graphs with labeled edges to indicate how one type of biomolecule can promote or inhibit the production of another.    When applied to the expression of genes, these graphs are called  `gene regulatory networks' \cite{Davidson2006,Kauffman1969}.   These and other networks became important in `systems biology',  a holistic approach to biological systems that aims to understand how the interactions between their components lead to emergent behaviors \cite{Alon2006,SystemsBiologyTextbook}.    Since the 1990s, large databases of biological networks have been created, integrating visualizations with molecular data, ontologies and references to the literature \cite{KEGG2025}.   Eventually standardization became necessary, and a committee was formed to develop Systems Biology Graphical Notation (SBGN) \cite{le_novere_systems_2009}.  This is a system of visual representations of biological networks to facilitate clear, consistent communication and the reuse of information. Over time, SBGN has become a widely adopted standard for visualizing biological networks at varying levels of detail.  It encompasses three distinct yet complementary visual languages:
\begin{itemize}
\item \emph{Process Description} (SBGN-PD) focuses on detailed biochemical reactions  \cite{rougny_systems_2019}.
\item \emph{Activity Flow} (SBGN-AF) illustrates 
the flow through which various biochemical entities and activities influence each other  \cite{mi_systems_2015}.
\item \emph{Entity Relationship} (SBGN-ER)  highlights how entities affect each other's actions and behaviors \cite{sorokin_systems_2015}.
\end{itemize}
The second author and his collaborators have recently attempted to develop a compositional framework for SBGN-PD diagrams via the theory of `open process networks' \cite{chaudhuri2024mathematicalframeworkstudyorganising}, motivated by the theory of open Petri nets \cite{BaezMaster2020}.   The current work is instead connected to SBGN-AF diagrams. These diagrams resemble regulatory networks but are more general, as they incorporate a broader range of influences. Alongside the usual positive and negative effects of one biological entity on another, they can also represent relationships such as necessary stimulation and unknown influence.

We study three kinds of morphisms between labeled graphs, which require increasing amounts of structure on the label set:
\begin{enumerate}
\item `maps' between $L$-labeled graphs for a set $L$.  Here is an example taking $L = \N$: 
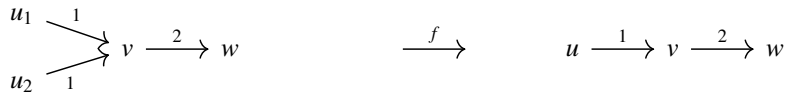
\[
\begin{tikzcd}[row sep = "0.1 em"]
u_1 \arrow[dr, "1", pos = 0.3] \\
                           &  v \arrow[r, "2"]  & w  &         &    {} \arrow[r, "f"] & {} & u \arrow[r, "1"]    & v \arrow[r, "2"] & w  \\
 u_2 \arrow[ur, "1", swap, pos = 0.2] 
\end{tikzcd}
\]
Here the edge labels simply get pulled back along a map of graphs.  As explained in \cref{Sec:Set-labeled_graphs}, maps between set-labeled graphs can be useful for refining a simple model of a system to a more complex model.   The more complex model is the \emph{source} of the map of $L$-labeled graphs; in the example above we refined a model with one entity called $u$ to one with two similar entities called $u_1$ and $u_2$.
\item `Kleisli morphisms' between $M$-labeled graphs for a monoid $M$.  Here is an example taking $M = \N$ made into a monoid using addition:
\[
\begin{tikzcd}
u \arrow[rr, "5"]  
& & w  &    
&    {} \arrow[r, "f"] & {} & u \arrow[r, "3"]    & v \arrow[r, "2"] & w          
\end{tikzcd}
\]
Here an edge can get mapped to a path of edges if its label is the sum of their labels.   In \cref{Sec:Monoid-labeled_graphs} we argue that many important sets of labels in systems dynamics and systems biology are in fact monoids.   \cref{Sec:Motifs} explains how Kleisli morphisms can be used to find `motifs': simple patterns which play important structural roles in many systems.
\item `additive morphisms' between finite $C$-labeled graphs whenever $C$ is a commutative monoid.  Here is an example taking $C = \N$ made into a commutative monoid using addition:
\[
\begin{tikzcd}
u \arrow[r, "1", bend left=40] \arrow[r, "1"] \arrow[r, "2", bend right=40] 
& v \arrow[r, "1", bend left=20] \arrow[r, "1", swap, bend right=20] & w  &                              
 &    {} \arrow[r, "f"] & {} & u \arrow[r, "4"]                                                 & v \arrow[r, "2"] & w          
\end{tikzcd}
\]
Here when several edges get mapped to the same edge, we sum their labels.   In systems dynamics and systems biology, most of the monoids of labels are commutative.   \cref{Sec:Commutative_monoid-labeled_graphs} explains how additive morphisms can be used to simplify a complex model of a system.
\end{enumerate}

We also study `open' labeled graphs, which are roughly labeled graphs with some vertices specified as `inputs' or `outputs'.  The advantage of these is that large labeled graphs built 
out of small open labeled graphs by a process of composition.  We consider three variants of this theme:

\begin{enumerate}
\item In \cref{Thm:Open_set-labeled_graphs}, for any set $L$ we construct a symmetric monoidal double category of open $L$-labeled graphs and maps between these.   This uses the theory of structured cospans \cite{BaezCourser2020,BaezCourserVasilakopoulou2022,CourserThesis}.   
\item In \cref{Thm:Open_monoid-labeled_graphs}, for any monoid $M$ we construct a symmetric monoidal double category of open $M$-labeled graphs and Kleisli morphisms between these.   This pushes the theory of structured cospans slightly beyond its usually stated level of generality.
\item In \cref{Thm:Open_commutative_monoid-labeled_graphs}, for any commutative monoid $C$ we construct a symmetric monoidal double category of open $C$-labeled finite graphs and additive morphisms between these.  The theory of structured cospans seems unable to accomplish this, but we can do it using the theory of decorated cospans \cite{BaezCourserVasilakopoulou2022, Fong2015,FongThesis}.
\end{enumerate}

We also study feedback loops in graphs, and how new feedback loops can emerge when one composes open graphs.   To study these, in \cref{Sec:Feedback_loops} we define the first homology monoid of a graph $G$ with coefficients in a commutative monoid $C$.  When $C$ is an abelian group this monoid reduces to the familiar homology group with coefficients in $C$, which does not depend on the direction of the edges of $G$.    When $C$ is a monoid without inverses, such as $\N$ with addition, the first homology detects \emph{directed}  loops in $G$.   This is crucial for the study of feedback, since an edge $A \to B$ indicates that $A$ affects $B$, but not necessarily vice versa.   We carry out a detailed study of the first homology of $G$ with coefficients in $\N$,  showing in \cref{Thm:simple_loops_and_minimal_cycles} that it is generated by homology classes of `simple loops': that is, directed loops that do not cross themselves.

In \cref{Sec:Emergent_feedback_loops} we study how directed loops emerge when we compose open graphs.  For example, this graph has no directed loops:  
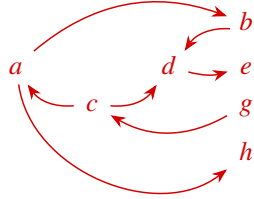
\[
\begin{tikzcd}[column sep = 1.5em, row sep = 0.2 em]
      &    &    & \re{b} \arrow[ld, darkred, -{Stealth[length=2mm]}, bend right, shift left]  
\\
\re{a} \arrow[rrru, darkred,  -{Stealth[length=2mm]}, bend left] 
\arrow[rrrdd, "", darkred,  -{Stealth[length=2mm]},  bend right=60] &  & 
\re{d} \arrow[r, darkred, -{Stealth[length=2mm]},  bend right=15] & \re{e}   
\\
 & \re{c} \arrow[lu, "", darkred,  -{Stealth[length=2mm]},  bend left] 
 \arrow[ru, ""', darkred, -{Stealth[length=2mm]},  bend right]   &  & 
 \re{g} \arrow[ll, "", darkred,  -{Stealth[length=2mm]}, bend left]       
 \\
 &   &    & \re{h}                                                                         
\end{tikzcd}
\]
but when we glue it to the following graph, which also has no directed loops:
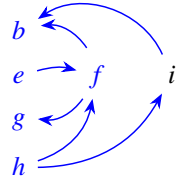
\[
\begin{tikzcd}[column sep = 1.5em, row sep = 0.2 em]
  &   &   & \bl{b} &                                                                        \\
 &    &  & \bl{e} \arrow[r, "", blue, -{Stealth[length=2mm]},  bend left=15]    
  & \bl{f} \arrow[ld, "", blue, -{Stealth[length=2mm]}, bend left] 
  \arrow[lu, ""', blue, -{Stealth[length=2mm]}, bend right, shift right] 
  & i   \arrow[ull, "", blue, -{Stealth[length=2mm]}, bend right=50]  
\\
  &  &  & \bl{g}  
 \\
  &  &   & \bl{h}  \arrow[uur, "", blue, -{Stealth[length=2mm]},  bend right]  
  \arrow[uurr, "", blue, -{Stealth[length=2mm]}, bend right]                             
\end{tikzcd}
\]
by identifying the vertices called $b,e,g$ and $h$, we obtain a graph with directed loops, one shown in boldface:
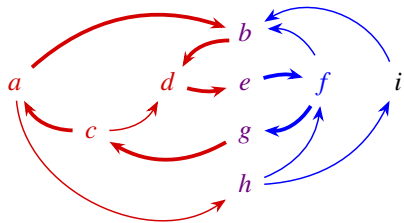
\[
\begin{tikzcd}[column sep = 1.5em, row sep = 0.4em]
  &  &   & \p{b} \arrow[ld, ""', darkred, -{Stealth[length=2mm]}, 
  line width=0.5mm, bend right, shift left] & 
\\
\re{a} \arrow[rrru, darkred, -{Stealth[length=2mm]}, line width=0.5mm,  "",  bend left] 
\arrow[rrrdd, "", darkred, -{Stealth[length=2mm]}, bend right=60] &  & 
\re{d} \arrow[r, darkred, -{Stealth[length=2mm]}, line width=0.5mm, bend right=15] 
 & \p{e}  \arrow[r, "", blue, -{Stealth[length=2mm]}, line width=0.5mm, bend left=15]               
 & \bl{f} \arrow[ld, "", blue, -{Stealth[length=2mm]}, line width=0.5mm, bend left] 
 \arrow[lu, ""', blue, -{Stealth[length=2mm]},  bend right, shift right] 
 & i   \arrow[ull, "", blue, -{Stealth[length=2mm]}, bend right=50]                
 \\
 & \re{c} \arrow[lu, "", darkred, -{Stealth[length=2mm]}, line width=0.5mm, bend left] 
 \arrow[ru, ""', darkred, -{Stealth[length=2mm]}, bend right]   & & 
 \p{g} \arrow[ll, "", darkred, -{Stealth[length=2mm]}, line width=0.5mm, bend left]                                      
 \\
  &  &   & \p{h}  \arrow[uur, "", -{Stealth[length=2mm]}, blue, bend right]   
  \arrow[uurr, "", blue, -{Stealth[length=2mm]}, bend right]                                                                                                             
\end{tikzcd}
\]
In the case of undirected loops the phenomenon of `emergent loops' is well understood by the van Kampen theorem.  For homology with coefficients in an abelian group the analogous phenomenon is well understood using the Mayer--Vietoris exact sequence  \cite{Hatcher2009}.   We prove some related theorems for directed loops and homology with coefficients in $\N$.

\subsection*{Notation}

In this paper we use a sans-serif font like $\C$ for categories, boldface like $\mathbf{B}$ for bicategories or 2-categories, and blackboard bold like $\lD$ for double categories.   Thus, $\cat$ denotes the category of small categories while $\Cat$ denotes the 2-category of small categories.  For double categories with names having more than one letter, like $\lCsp(\X)$, only the first letter is in blackboard bold.

\subsection*{Acknowledgements} 
We thank Nathaniel Osgood for teaching us about causal loop diagrams and system dynamics, and Olaf Wolkenhauer for suggesting the application of causal loop diagrams to Systems Biology Graphical Notation -- Activity Flow (SBGN-AF). We thank Ralf K\"ohl, Olaf Wolkenhauer, Cl\'emence R\'eda, Shailendra Gupta, and Rahul Bordoloi for valuable discussions on the applications of causal loop diagrams to SBGN-AF.    We thank Oscar Cunningham, Morgan Rogers and Zoltan Kocsis for dispelling our hope that the first homology monoid of a directed graph is always free, and Tobias Fritz for explaining a generalization of homology for commutative monoids.  We thank David Egolf for suggesting the idea of a monoid containing an element indicating that one vertex has a direct effect on another that we are deciding to neglect.  We thank Evan Patterson and Xiaoyan Li for discussions of Catcolab's and AlgebraicJulia's treatments of these ideas, Kevin Carlson and others for general help with category theory, and the referees for other improvements of this paper.

\section{Set-labeled graphs}
\label{Sec:Set-labeled_graphs}

Although a graph captures the interconnection between a set of elements, it does not say anything about the \emph{types of interconnection}.  We can do this using labeled graphs with edges labeled by elements of a set.

 First, some preliminaries.   For us, a \define{graph} $G=(E, V, s, t)$ consists of a set of \define{vertices} $V$ and a set of \define{edges} $E$ together with a \define{source map} $s\maps E \to V$ and a \define{target map} $t \maps E \to V$. We define a \define{map of graphs} $f=(f_0,f_1)$ from a graph $G=(E,V,s,t)$ to a graph $G'=(E',V',s',t')$ to be a pair of functions $f_0 \maps V \to V'$ and $f_1 \maps E \to E'$ such that the following two diagrams commute:
\[
\begin{tikzcd}
E \arrow[r, "f_1"] \arrow[d, "s"'] & E' \arrow[d, "s'"] \\
V \arrow[r, "f_0"']                & V'               
\end{tikzcd}
\qquad
\begin{tikzcd}
E \arrow[r, "f_1"] \arrow[d, "t"'] & E' \arrow[d, "t'"] \\
V \arrow[r, "f_0"']                & V'.               
\end{tikzcd}
\]
The collection of graphs and maps between them form a category, which we denote by $\Gph$.  The category $\Gph$ is equivalent to the functor category $\Set^\G$, where $\G$ is the category freely generated by two parallel morphisms $e \rightrightarrows v$.   Since $\G \cong \G^\op$ we can also say $\Gph$ is equivalent to $\Set^{\G^\op}$.  Thus, $\Gph$ is a presheaf topos and it enjoys all the privileges of such categories \cite{MacLaneMoerdijk2012}; in particular it is complete and cocomplete.  When we replace the category $\Set$ with the category of finite sets $\Fin\Set$, we obtain a category
\[  \Fin\Gph \simeq \Fin\Set^\G \]
called the category of \define{finite graphs}, whose objects are graphs with finitely many vertices and edges.  This category is also a topos, so it has finite limits and colimits.

Now we turn to the simplest sort of labeled graph: a graph whose edges are labeled by some fixed set.

\begin{defn}\label{Defn:Set-labeled-graph}
Given a set $L$, we define a \define{$L$-labeled graph} $(G,\ell)$
to be a graph $G=(E,V, s,t)$ equipped with a \define{labeling map} $\ell \maps E \to L$.  We call $L$ the set of \define{labels}.
\end{defn}

A map of labeled graphs is a map between their underlying graphs that preserves the labels of edges:

 \begin{defn}
 \label{Defn:morphisms-of-set-labeled-graphs}
 Given $L$-labeled graphs $(G,\ell)$ and $(G',\ell')$, define a \define{map of $L$-labeled graphs} to be a map of graphs $f \maps G \to G'$ that makes the following triangle commute:
\[
\begin{tikzcd}[column sep=0.8em]
E \arrow[rr, "f_1"] \arrow[dr, "\ell"'] && E' \arrow[dl, "\ell'"] \\
& L              
\end{tikzcd}
\]
 \end{defn}

We can pull back an $L$-labeling along a map of graphs:

\begin{lem}
\label{Lem:Pullback_along_a_map_of_graphs}
Let $L$ be a set and $(G,\ell)$ an $L$-labeled graph.  For any map of graphs $f \maps G' \to G$, there is a unique $L$-labeling of $G'$, called the \define{pullback of $\ell$ along $f$} and denoted $f^\ast \ell$, such that $f$ is a map of $L$-labeled graphs from $(G',f^\ast \ell)$ to $(G,\ell)$.   The pullback is contravariantly functorial in the following sense:
\[     (f \circ g)^\ast = g^\ast \circ f^\ast, \qquad 1_\ast = 1 .\]
\end{lem}

\begin{proof}
We are forced to take $f^\ast \ell = \ell \circ f_1$, and this works.
\end{proof}

To see how this fact can be used in system dynamics or systems biology, consider this map of graphs:
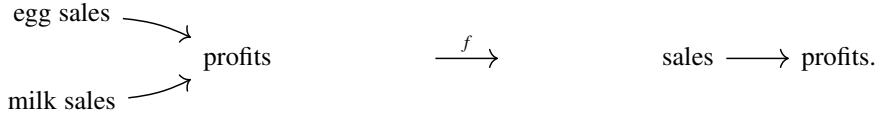
\[
\begin{tikzcd}[row sep = "0.1 em"]
\mathrm{egg \; sales} \arrow[dr, "", bend left =10]  \\
  &     \mathrm{profits}   &                  &    {} \arrow[r, "f"] & {} &                           & \mathrm{sales} \arrow[r, ""] & \mathrm{profits}.  \\
\mathrm{milk \; sales} \arrow[ur, "", bend right = 10]    
\end{tikzcd}
\]
We can think of the first graph as a more complicated model in which one vertex in the second
model has been differentiated into two.   Now suppose we label the edges of the second graph by elements of
$L = \{+,-\}$.   This graph has just one edge, and if increased sales leads to increased profits (all else being equal) we label this edge with $+$.   Then \cref{Lem:Pullback_along_a_map_of_graphs} says there is a unique way to label the edges of the first graph that lets us lift the map $f$ to a map of $L$-labeled graphs:
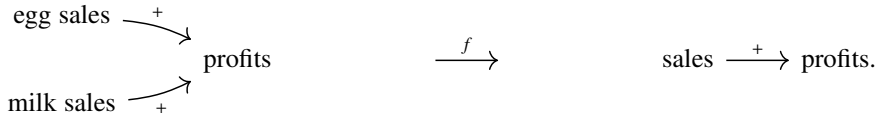
\[
\begin{tikzcd}[row sep = "0.1 em"]
\mathrm{egg \; sales} \arrow[dr, "+", pos = 0.3, bend left =10]  \\
  &     \mathrm{profits}   &                  &    {} \arrow[r, "f"] & {} &                           & \mathrm{sales} \arrow[r, "+"] & \mathrm{profits}.  \\
\mathrm{milk \; sales} \arrow[ur, "+", pos = 0.3, swap, bend right = 10]    
\end{tikzcd}
\]
The functoriality in \cref{Lem:Pullback_along_a_map_of_graphs} says that one can start with a model of a system as an $L$-labeled graph and successively refine it by choosing graphs that map to the original graph, with the $L$-labelings being automatically inherited by all these
 more refined graphs in a consistent way.  In system dynamics this and related processes of model refinement are called `stratification' \cite{BaezLiLibkindOsgoodRedekopp2023}.


In mathematical terms, there is a forgetful functor sending $L$-labeled graphs to their underlying graphs, and \cref{Lem:Pullback_along_a_map_of_graphs} says this functor is a `discrete fibration'.     For an introduction to discrete fibrations see \cite[Sec.\ 2.1]{LoregianRiehl}.    To better understand this and other facts about labeled graphs it can be useful to take a more sophisticated outlook.  

First, note that a $L$-labeled graph can be seen as a graph over $G_L$, the graph with only one vertex and an edge for each element $\ell \in L$.  That is, an $L$-labeled graph is the same as graph $G$ equipped with a map of graphs $p \maps G \to G_L$.  Similarly, a map of $L$-labeled graphs can be seen as a map of graphs over $L$, i.e., a map of graphs $f \maps G \to G'$ such that this triangle commutes:
 \[
\begin{tikzcd}[column sep=0.8em]
G \arrow[rr, "f"] \arrow[dr, "p"'] && G' \arrow[dl, "p'"] \\
& G_L             &     
\end{tikzcd} 
\]
Thus, the category of $L$-labeled graphs and maps between them is isomorphic to the slice category $\Gph/G_L$.  
This makes it easy to deduce a number of important facts from general principles. 

\begin{prop}
\label{Prop:Set-labeled_graph}
For any set $L$, we have the following:
\begin{itemize}
\item[(a)] The category $\Gph/G_L$ is a presheaf topos.
\item[(b)] The forgetful functor $U \maps \Gph/G_L \to \Gph$ is a discrete fibration.
\item[(c)] The presheaf  $\P \maps \Gph^\op \to \Set$ associated with $U$ is representable by $G_L$.
\item[(d)] The category $\Gph/G_L$ is equivalent to the category of elements $\int \P$. 
    \end{itemize}
\end{prop}

\begin{proof}
These follow from more general results.  If we write $\widehat{\C}$ for the category of presheaves on a category $\C$, and choose a presheaf $F \in \widehat{C}$, then:
\begin{itemize}
\item[(a)] The category $\widehat{\C}/F$ is a presheaf topos.
\item[(b)] The forgetful functor $U \maps \widehat{\C}/F \to \widehat{\C}$ is a discrete fibration.
\item[(c)] The functor $\P \maps \widehat{\C}^\op \to \Set$ associated with $U$ is representable by $F$.
\item[(d)] The category $\widehat{\C}/F$ is equivalent to the category of elements $\int \P$. 
\end{itemize}
The proposition follows taking $\C = \G$ and $F = G_L$. These more general results are well-known, but here we sketch them, with references:

(a) follows from the fact $\widehat{\C}/F \simeq \widehat{\int F}$, where $\int F$ is the category of elements of $F$ \cite[Exercise III.8]{MacLaneMoerdijk2012}.
In our example of interest, one can check directly that the category of $L$-labeled graphs is equivalent to the category of presheaves on a category with one object $v$, one object for each edge label $\ell \in L$, and two morphisms $s_\ell, t_\ell \maps v \to \ell$ for each $\ell$.

(b) For any object $d$ in a category $\D$, to say that the forgetful functor $U \maps \D/d \to D$ is a discrete fibration means that given any object over $d$, say $p \maps e \to d$, and any morphism $f \maps e' \to e$ in $\D$, there exists a unique morphism $g$ in $\D/d$ such that $U(g) = f$.  This is easily checked taking $g$ to be the following morphism from  $e' \xrightarrow{p \circ f} d$  to $e \xrightarrow{p} d$:
 \[
\begin{tikzcd}[column sep=0.8em]
e' \arrow[rr, "f"] \arrow[dr, "p \circ f"'] && e \arrow[dl, "p"] \\
& d.            &     
\end{tikzcd} 
\]

(c) This holds because the fiber of $U$ at any $G \in \widehat{C}$ is of the form $\hom(G,F)$.

(d) From Grothendieck's correspondence between discrete fibrations and presheaves \cite[Sec.\ 2.1]{LoregianRiehl} we have the following isomorphism of discrete fibrations
 \[
\begin{tikzcd}[column sep=0.8em]
\textstyle{\int} \P  \arrow[rr,"\cong"] \arrow[dr, ""'] && \widehat{\C}/F \arrow[dl, ""] \\
& \widehat{\C}            &     
\end{tikzcd} 
\]
and thus the category $\int \P$ is equivalent to $\widehat{\C}/F$.   \end{proof}

The discrete fibration in Proposition \ref{Prop:Set-labeled_graph} (b) implies that there is a functorial way to pull back $L$-labelings along maps of graphs, as we have already seen.  We can also push forward labelings along a map between labeling sets.  That is, given a map $\phi \maps L \to L'$, and an $L$-labeled graph $(G,\ell)$, we can define an $L'$-labeled graph 
\[         \phi_\ast(G,\ell) := (G, \phi \circ \ell) .\]
This process is functorial:
\[        (\phi \circ \psi)_\ast = \phi_\ast \circ \psi_\ast , \qquad 1_\ast = 1 .\]
More formally, we have the following result.

\begin{prop}
\label{Prop:Change_of_labeling_set}
The following assignment
\[
\begin{array}{cccl}
F \maps &  \Set &\to& \cat \\ \\
& L &\mapsto& \Gph/G_L\\
&(L \xrightarrow{\phi} L') &\mapsto& (\Gph/G_L \xrightarrow{\phi_{*}} \Gph/G_{L'})
\end{array}
\]
defines a functor.
\end{prop}

Proposition \ref{Prop:Change_of_labeling_set} lets us construct a single \define{category of set-labeled graphs} that allows for all possible choices of label set $L$.   The method for doing this is called the Grothendieck construction $\int F$ of the functor $F$, a generalization of the category of elements \cite[Chap.\ 10]{JohnsonYau2021}.  Explicitly:

\begin{prop}
\label{Prop:Category_of_set-labeled_graphs}
$\int F$ is the category where:
\begin{itemize}
\item an object is a pair $(L, p \maps G \to G_L)$,
\item a morphism from $(L, p \maps G \to G_L)$ to $(L', p' \maps G' \to G_{L'})$ is a pair $(\phi,f)$, where $\phi \maps L \to L'$ is a function and $f \maps G \to G'$ is a map of graphs such that this square commutes:
\[
\begin{tikzcd}
G \arrow[r, "f"] \arrow[d, "p"'] & G' \arrow[d, "p'"] \\
G_L \arrow[r, "\tilde{\phi}"']                & G_{L'}             
\end{tikzcd}
\]
where $\tilde{\phi} \maps G_L \to G_{L'}$ maps each edge $\ell \in L$ of $G_L$ to the edge $\phi(\ell) \in L'$ of $G_{L'}$.
\end{itemize}
\end{prop}

\section{Monoid-labeled graphs}
\label{Sec:Monoid-labeled_graphs}

We now turn to graphs where the edges are labeled by elements of a monoid.  We think of these monoid elements as `polarities': ways for the entity corresponding to one vertex to affect another.   If a graph has an edge from a vertex $u$ to a vertex $v$ labeled with a monoid element $m$, and an edge from $v$ to a vertex $w$ labeled with a monoid element $m'$, we say that $u$ has an indirect effect on $w$ equal to $m' m$.  

Here are some monoids that are useful for describing polarities:

\begin{eg}
\label{Eg:1}
The terminal monoid is a monoid containing just one element.  This is also known as the trivial group.  In the applications at hand we write the group operation as multiplication and call the one element $1$, so that $1 \cdot 1 = 1$.   Thus, we call this monoid $\{1\}$.   Any graph becomes a $\{1\}$-graph in a unique way, by labeling each edge with $1$, and this gives an isomorphism of categories
\[            \{1\} \Gph \cong \Gph .\]
We can use a graph to describe causality in at least two distinct ways.  Suppose $v$ and $w$ are vertices of a graph.
\begin{enumerate}
\item We can use the presence of an edge from $v$ to $w$ to indicate that the entity named by $v$ has a \emph{direct effect} on the entity named by $w$, and the absence of an edge to indicate that $v$ has \emph{no direct effect} on $w$.   (Even if there is no edge from $v$ to $w$, $v$ may still have an indirect effect on $w$ if there is path of edges from $v$ to $w$.)
\item We can use the presence of an edge from $v$ to $w$ to indicate that the entity named by $v$ has a \emph{direct effect} on the entity named by $w$,
 and the absence of an edge to indicate that $v$ has \emph{has no currently known direct effect} on $w$.  This interpretation is useful for situations where we start with a graph having no edges, and add an edge each time we discover that one vertex has a direct effect on another.
\end{enumerate}
We can also take at least three different attitudes to the presence of multiple edges from one vertex to another:
\begin{itemize}
\item[(a)] We can treat them as redundant, hence unnecessary, allowing us to simplify any graph so that it has at most one edge from one vertex to another.
\item[(b)]  We can treat them as indicating different ways in which one vertex directly affects another.   
\item[(c)] We can use the number of edges from $v$ to $w$ to indicate the amount by which $v$ affects $w$.  
\end{itemize}
All these subtleties of interpretation can also arise for $M$-graphs where $M$ is any other monoid.  We will not mention them each time, but in applications it can be important to clearly fix an interpretation.
\end{eg}

\begin{eg}
\label{Eg:Z/2}
The most widely used monoid in this field is the abelian group $\Z/2$.   It is typical to write the group operation as multiplication rather than addition, and call this group $\{+,-\}$:
\vskip 0.5em
\renewcommand\arraystretch{1.3}
\setlength\doublerulesep{0pt}
\begin{center}
\begin{tabular}{c||c|c}
$\cdot$ & $+$ & $-$ \\
\hline\hline
$+$ & $+$ & $-$ \\ 
\hline
$-$ & $-$ & $+$ \\ 
\end{tabular}
\end{center}
\vskip 0.5em
In the field called System Dynamics \cite[Chap.\ 5]{Sterman2002}, a graph with edges labeled by elements of $\{+,-\}$ is called a `causal loop diagram'.  Here is a causal loop diagram serving as a simple model of students doing homework:
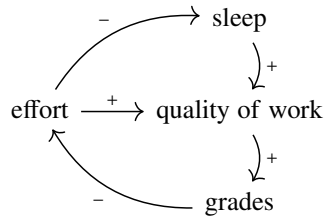
\[
\begin{tikzcd}
& \mathrm{sleep} \arrow[d, "+", bend left] \\
\mathrm{effort} \arrow[ur, "-", bend left] \arrow[r, "+"] & \mathrm{quality \; of \; work} \arrow[d, "+", bend left] &                        \\
& \mathrm{grades} \arrow[lu, "-", bend left]
\end{tikzcd}
\]
The effort a student makes has a direct positive effect on the quality of their work, so there is an edge from `effort' to `quality of work' labeled with a $+$.     However, effort also has a direct negative effect on the student's sleep, and sleep has a direct positive effect on quality of their work.   Thus, although the student's effort has a direct positive effect on the quality of their work, it also has an indirect negative effect, which arises from multiplying the negative effect of effort on sleep with the positive effect of sleep on the quality of work. 

In biology, the simplest sort of regulatory network \cite{Aduddell} is a graph with edges labeled by elements of $\{+,-\}$.  However, the notation is different.  Instead of an edge labeled with $+$, an arrow  
$A$\raisebox{0.1 em}{
\begin{tikzpicture}[scale=0.2, line width=0.8pt]
  \draw (-2,0) -- (1,0);
  \draw[fill=white] (0.4,0.3) -- (1.2,0) -- (0.4,-0.3) -- cycle;
\end{tikzpicture} 
}$B$
indicates that $A$  `promotes' or `stimulates' $B$, while an arrow
$A$\raisebox{0.1 em}{
\begin{tikzpicture}[scale=0.2, line width=0.8pt]
  \draw (-2,0) -- (0.8,0);
  \draw (0.70,-0.4) -- (0.70,0.4);
\end{tikzpicture} 
}$B$ indicates that $A$ `represses' or `inhibits' $B$.   In the official SBGN-AF terminology \cite{mi_systems_2015},
$A$\raisebox{0.1 em}{
\begin{tikzpicture}[scale=0.2, line width=0.8pt]
  \draw (-2,0) -- (1,0);
  \draw[fill=white] (0.4,0.3) -- (1.2,0) -- (0.4,-0.3) -- cycle;
\end{tikzpicture} 
}$B$ means that $A$ is a `positive influence' on $B$, while 
$A$\raisebox{0.1 em}{
\begin{tikzpicture}[scale=0.2, line width=0.8pt]
  \draw (-2,0) -- (0.8,0);
  \draw (0.70,-0.4) -- (0.70,0.4);
\end{tikzpicture} 
}$B$ means that $A$ has a `negative influence' on $B$.
\end{eg}

\begin{eg}
\label{Eg:Z/3}
Another important monoid of polarities is the multiplicative monoid of $\Z/3$, which we can write either as $\{1,0,-1\}$ or $\{+,0,-\}$.   Now the multiplication table is this:
\vskip 0.5em
\renewcommand\arraystretch{1.3}
\setlength\doublerulesep{0pt}
\begin{center}
\begin{tabular}{c||c|c|c}
$\cdot$ & $+$ & $0$ & $-$ \\
\hline\hline
$+$ & $+$ & $0$ & $-$ \\ 
\hline
$0$ & $0$ & $0$ & $0$ \\ 
\hline
$-$ & $-$  & $0$ & $+$ \\ 
\end{tabular}
\end{center}
\vskip 0.5em
 This monoid contains the one in the previous example as a submonoid, but also a new element $0$ that is \define{absorbing}: $0 \cdot x = x \cdot 0 = 0$ for all $x$.   So, unlike $\{+,-\}$, this monoid $\{+,0,-\}$ is not a group.

There are at least three distinct interpretations of $\{+,0,-\}$-graphs:

\begin{enumerate}
\item
We can use $+$ to denote a positive effect, $-$ to denote a negative effect, and $0$ to denote \emph{an effect whose value varies depending on the circumstances}.    
\item 
We can use $+$ to denote a positive effect, $-$ to denote a negative effect, and $0$ to denote \emph{an effect that we deem negligible}.
\item 
We can use $+$ to denote a positive effect, $-$ to denote a negative effect, and $0$ to denote \emph{an unknown effect}.  
\end{enumerate}
Our choice of interpretation affects how we interpret the absence of an edge from one vertex to another.  We discuss this in \cref{Eg:absorbing_element}.

SBGN-AF diagrams can be interpreted as monoid-labeled graphs.  In this formalization, this new element $0$ can be used to mean an `unknown influence', and denoted with a diamond-headed arrow.   
Thus,  
$A$ \raisebox{0.1 em}{
\begin{tikzpicture}[scale=0.2, line width=0.8pt]
  \draw (-2,0) -- (1,0);
  \draw[fill=white] (0.4,0) -- (0.8,0.4) -- (1.2,-0) -- (0.8,-0.4) -- cycle;
\end{tikzpicture}  
}$B$ means that $A$ has an unknown influence on $B$.

\end{eg}

\begin{eg}
\label{Eg:absorbing_element}
More generally, for any monoid $M$ we can form a new monoid $M \cup \{0\}$ where $0$ is a new absorbing element.   Thus, in $M \cup \{0\}$ the product of elements of $M$ is defined as before, but $m \cdot 0 = 0 \cdot m = 0$ for all $m \in M$, and $0 \cdot 0 = 0$.   This new monoid contains $M$ as a submonoid.   As before, there are at least three interpretations of the element $0$:
\begin{enumerate}
\item
We can use $0$ to denote an effect whose value varies depending on the circumstances.   In this case we can use the absence of an edge from a vertex $v$ to a vertex $w$ to indicate either:
\begin{itemize}
\item[(a)] there is \emph{no direct effect} of $v$ on $w$, or
\item[(b)] the effect of $v$ on $w$ is \emph{unknown}.
\end{itemize}
\item
We can use $0$ to denote an effect that we deem negligible.   In this case we can again use either interpretation (a) or (b) of the absence of an edge.
\item 
We can use $0$ to denote an unknown effect.   In this case we can use interpretation (a) of the absence of an edge, but interpretation (b) is awkward, because then the absence of an edge means the same thing as an edge labeled by $0$.
\end{enumerate}
\end{eg}

\begin{eg}
\label{Eg:necessary_stimulus}
For any monoid $M$ we can also form a new monoid $M \cup \{I\}$ where $I$ is a new identity element.   Thus, in $M \cup \{I\}$ the product of $M$ is defined as before, but $m \cdot I = I \cdot m = m$ for all $m \in M$, and $I \cdot I = I$.   When we apply this construction to $\{+,-\}$ we obtain a monoid $\{+,I,-\}$ with the following multiplication table:
\vskip 0.5em
\renewcommand\arraystretch{1.3}
\setlength\doublerulesep{0pt}
\begin{center}
\begin{tabular}{c||c|c|c}
$\cdot$ & $I$ &  $+$ & $-$ \\
\hline\hline
$I$ & $I$ & $+$ & $-$ \\ 
\hline
$+$ &$+$ & $+$ & $-$ \\ 
\hline
$-$ & $-$  & $-$ & $+$ \\ 
\end{tabular}
\end{center}
\vskip 0.5em
We can use an edge labeled by $I$ from $v$ to $w$ to indicate that $v$ is \emph{necessary} for $w$.    In SBGN-AF diagrams this relationship is called `necessary stimulation'.

If we adjoin an absorbing element to the monoid $\{I,+,-\}$, we obtain a  monoid $\{I,+,0,-\}$ that can handle a large fraction of SBGN-AF diagrams.    But in these diagrams:
\begin{itemize}
\item instead of an edge labeled $I$, we write
\raisebox{0.1 em}{
\begin{tikzpicture}[scale=0.2, line width=0.8pt]
  \draw (-2,0) -- (1,0);
  \draw (0.01,0) -- (0.399,0);  
 \draw (0.1,-0.3) -- (0.1,0.3);
  \draw[fill=white] (0.4,0.3) -- (1.2,0) -- (0.4,-0.3) -- cycle;
\end{tikzpicture} 
}
for `necessary stimulation'.
\item instead of an edge labeled $+$, we write
\raisebox{0.1 em}{
\begin{tikzpicture}[scale=0.2, line width=0.8pt]
  \draw (-2,0) -- (1,0);
  \draw[fill=white] (0.4,0.3) -- (1.2,0) -- (0.4,-0.3) -- cycle;
\end{tikzpicture} 
}
for `positive influence'.
\item Instead of an edge labeled $0$, we write
\raisebox{0.1 em}{
\begin{tikzpicture}[scale=0.2, line width=0.8pt]
  \draw (-2,0) -- (1,0);
  \draw[fill=white] (0.4,0) -- (0.8,0.4) -- (1.2,-0) -- (0.8,-0.4) -- cycle;
\end{tikzpicture} 
for `unknown influence'.
}
\item Instead of an edge labeled $-$, we write
\raisebox{0.1 em}{
\begin{tikzpicture}[scale=0.2, line width=0.8pt]
  \draw (-2,0) -- (0.8,0);
  \draw (0.70,-0.5) -- (0.70,0.5);
\end{tikzpicture} 
}
for `negative influence'.
\end{itemize}
We can also obtain this monoid $\{I,+,0,-\}$ by adjoining a new identity element to the monoid $\{+,0,-\}$ discussed in \cref{Eg:Z/3}.   Either way, the multiplication table is as follows:
\vskip 0.5em
\renewcommand\arraystretch{1.3}
\setlength\doublerulesep{0pt}
\begin{center}
\begin{tabular}{c||c|c|c|c}
$\cdot$ & $I$ &  $+$ & $0$ & $-$ \\
\hline\hline
$I$ & $I$ & $+$ & $0$& $-$ \\ 
\hline
$+$ &$+$ & $+$ & $0$ & $-$ \\ 
\hline
$0$ & $0$  & $0$& $0$ & $0$ \\ 
\hline
$-$ & $-$  & $-$& $0$ & $+$ \\ 
\end{tabular}
\end{center}
\vskip 0.5em
However, it is worth mentioning that our description of SBGN-AF diagrams as monoid-labeled graphs is far from complete. For instance, we have not addressed important concepts such as logic arcs and equivalence arcs that are present in the official SBGN-AF terminology.
\end{eg}

\begin{eg}
\label{Eg:multiplicative_group_of_reals}
The multiplicative group of the reals, $(\R - \{0\}, \cdot, 1)$, consists of all real numbers except zero, with multiplication as its monoid operation.  The monoid $\{+,-\}$ of \cref{Eg:Z/2} gives purely qualitative information about whether an effect is positive or negative.  We can label the edges of a graph with elements of $\R - \{0\}$ to give quantitative information about \emph{how much} of a positive or negative direct effect one vertex has on another. 
\end{eg}

\begin{eg}
\label{Eg:multiplicative_monoid_of_reals}
The multiplicative monoid of the reals, $(\R, \cdot, 1)$, consists of all real numbers with multiplication as its monoid operation.  This is obtained from the multiplicative group of the reals by adjoining an absorbing element as in \cref{Eg:absorbing_element}.
\end{eg}

\begin{eg} 
The monoid $([0,\infty), +, 0)$ consists of nonnegative
real numbers with addition as its monoid operation.   We can use this monoid to describe delayed effects: an edge labeled with a real number $t \ge 0$ can indicate that one vertex affects another after a delay of time $t$. 
\end{eg}

\begin{eg} 
The monoid $(\N, +, 0)$ consists of all natural numbers with addition as its monoid operation.   We can use this monoid to describe delayed effects in discrete-time systems, using an edge labeled with a natural number $n$ to indicate that one vertex affects another after a delay of $n$ time steps.   
\end{eg}

\begin{eg}
\label{Eg:products_of_monoids}
Products of monoids are also useful: for example, to describe both time delays and whether an effect is positive or negative, we can use the monoid $([0,\infty), +, 0) \times \{+, -\}$. 
\end{eg}

\begin{eg}
\label{Eg:semiautomata}
All the above monoids above are commutative, and indeed commutative monoids are by far the most commonly used for polarities.  We discuss special features of the commutative case in Sections \ref{Sec:Commutative_monoid-labeled_graphs} and \ref{Sec:Feedback_loops}.  However, graphs with edges labeled by not-necessarily-commutative monoids do show up naturally in some contexts.  For example, in computer science \cite[Sec.\ 2.1]{Ginzburg1968}, a \define{semiautomaton} consists of a set $V$ of \define{states}, a set $A$ of \define{inputs}, and a map $\alpha \maps A \to V^V$ that describes how each input acts on each state to give a new state.  Let $M$ be the monoid of maps from $V$ to itself generated by all the maps $\alpha(a)$ for $a \in A$.   Let $G$ be the graph where:
\begin{itemize}
\item The set of vertices is $V$.
\item The set of edges is $E = A \times V$.
\item The source map is given by
\[  \begin{array}{rccl}
 s \maps & E & \to & V \\
       & (a,v) & \mapsto & v .
\end{array}
\]
\item The target map is given by
\[  \begin{array}{rccl}
 t \maps & E & \to & V \\
      &  (a,v) & \mapsto & \alpha(a)(v).
\end{array}
\]
Since the monoid of maps $M \subseteq V^V$ is generated by elements $\alpha(a)$ for $a \in A$, there is an $M$-labeling of $G$ given by
\[    \ell(a,v) = \alpha(a) .\]
In short, a semi-automaton gives an monoid-labeled graph where the vertices represent states and for each input $a$ mapping a state $v$ to a state $w$ there is an edge labeled by the monoid element $\alpha(a) \in M$.  Note however that it also gives a set-labeled graph with the same vertices, where for each input $a$ mapping a state $v$ to a state $w$ there is an edge labeled by $a \in A$. 
\end{itemize}
\end{eg}

Given a monoid $(M, \cdot, 1)$, we define an $M$-labeled graph and a map of $M$-labeled graphs just as when $M$ is a mere set (see Definitions \ref{Defn:Set-labeled-graph} and \ref{Defn:morphisms-of-set-labeled-graphs}).  One advantage of using graphs with labelings $\ell \maps E \to M$ where $M$ is a monoid is that for any path of edges in $G$
\[   v_0 \xrightarrow{e_1} v_1 \xrightarrow{e_2} \cdots \xrightarrow{e_{m-1}} v_{m-1} \xrightarrow{e_m} v_m \]
we can form the product $\ell(e_m) \cdots \ell(e_1)$, and use this to describe how the vertex $v_0$ \emph{indirectly} affects the vertex $v_m$.

To formalize this, recall that there are adjoint functors
\[
\begin{tikzcd}
            \  \Gph \arrow[r, bend left = 20, shift left=1ex, "{\Free}"{name=G}] & \ \cat \arrow[l, bend left = 20, shift left=.5ex, pos = 0.5, "\Und"{name=F}]
            \arrow[phantom, from=F, to=G, "\scriptstyle\boldsymbol{\bot}"].
        \end{tikzcd}
    \]
The functor $\Und \maps \cat \to \Gph$ sends any category $\C$ to its underlying graph, whose vertices are objects of $\C$ and whose edges are morphisms.  This has a left adjoint $\Free \maps \Gph \to \cat$
sending each graph $G = (E,V,s,t)$ to the category $\Free(G)$ where:
\begin{itemize}
\item an object is a vertex $v \in V$.
\item a morphism from $v$ to $w$ is a \define{path} in $G$:
\[   v = v_0 \xrightarrow{e_1} v_1 \xrightarrow{e_2} \cdots \xrightarrow{e_{m-1}} v_{m-1} \xrightarrow{e_m} v_m = w \]
where $m \ge 0$.
\item composing paths
\[   v_0 \xrightarrow{e_1}  \cdots   \xrightarrow{e_m} v_m \]
and 
\[   v_m \xrightarrow{e_{m+1}} \cdots \xrightarrow{e_{m+n}} v_{m+n} \]
gives the path
\[  v_0 \xrightarrow{e_1}  \cdots   \xrightarrow{e_m} v_m \xrightarrow{e_{m+1}} \cdots \xrightarrow{e_{m+n}} v_{m+n} \]
\item the identity $1 \maps v \to v$ is the path of length $0$ starting and ending at $v$.
\end{itemize}

Suppose $M$ is a monoid.  Our next goal is to show that when $(G,\ell)$ is an $M$-labeled graph, $\Free(G)$ becomes a kind of $M$-labeled category---usually called an `$M$-graded' category.   Just as $M$-labeled graphs are conveniently treated as graphs over $G_M$, we can define an $M$-graded category to be a category $\C$ over $\B M$, the one-object category with one morphism for each element of $M$, with composition defined to be multiplication in $M$.  

\begin{defn}
\label{Def:Monoid-labeled_category}
Let $M$ be a monoid.  The \define{category of} $M$-\define{graded categories} is the slice category $\cat/\B M$.
\end{defn}

Thus, an $M$-graded category $p \maps \C \to \B M$ is a category $\C$ for which each morphism is mapped to an element $p(f)$ of $M$ called its \define{grade} in such a way that 
\[    p(f g) = p(f) p(g)  \]
for any pair of composable morphisms $f$ and $g$, and
\[    p(1_c) = 1 \in M \]
for each $c \in \C$.  A map of $M$-graded categories, say $F$ from $p \maps \C \to \B M$ to $p' \maps \C' \to \B M'$, works out to be a functor $F \maps \C \to \C'$ with the property that
\[    p'(F(f)) = p(f)   \]
for every morphism in $\C$.  In other words, $F$ preserves the grades of morphisms.

This category of monoid-graded categories enjoys certain nice properties analogous to the category of set-labeled graphs (see \cref{Prop:Set-labeled_graph}).

\begin{prop}
\label{Prop:Monoid-labeled_category}
For any monoid $M$, we have the following:
\begin{itemize}
\item[(a)] The category $\cat/\B M$ is locally finitely presentable, and thus complete and cocomplete.
\item[(b)] The forgetful functor $\bar{U} \maps \cat /\B M \to \cat$ is a discrete fibration.
\item[(c)] The presheaf  $\bar{\P} \maps \cat^\op \to \Set$ associated with $\F$ is representable by $\B M$.
\item[(d)] The category $\cat/\B M$ is same as the category of elements $\int \bar{\P}$. 
    \end{itemize}
\end{prop}

\begin{proof}
For (a), recall that a category is locally finitely presentable precisely when it is the category of models for a finite limit sketch
 \cite{AdamekRosicky}.  In the case of $\cat/\B M$, this follows because we can define a category over $\B M$ as having a set $\Ob$ of objects, for each $\ell \in M$ a set $\Mor_\ell$ of $\ell$-labeled morphisms, source and target maps $s_\ell, t_\ell \maps \Mor_\ell \to \Ob,$
an identity-assigning map $\id \maps \Ob \to \Mor_1$, and composition maps defined on pullbacks
\[   \circ_{\ell, \ell'} \maps \Mor_\ell \, {}_{s_\ell}\!\times_{t_\ell} \, \Mor_{\ell'} \to \Mor_{\ell' \ell} \]
satisfying associativity and the unit laws.  Every locally finitely presentable category is complete and cocomplete \cite{AdamekRosicky}.

Parts (b), (c) and (d) follow exactly the same way as in the proofs of the corresponding parts of Proposition \ref{Prop:Set-labeled_graph}.
\end{proof} 


\section{Motifs in monoid-labeled graphs}
\label{Sec:Motifs}

Small monoid-labeled graphs are often called `motifs' because, like motifs in a piece of music, they show up repeatedly in a meaningful way in a larger context \cite{Alon2007}.  For example take $M = \{+,-\}$.  Let $(G,\ell)$ be this $M$-labeled graph:
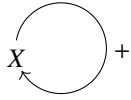
\[
\begin{tikzpicture}
    \node at (0,0) (X) {$X$};
    \draw [->] (X.north)arc(170:-150:0.6);
    \node at (1.4,0.1) {$+$};
\end{tikzpicture}
\]
Let $(H,m)$ be this larger $M$-labeled graph:
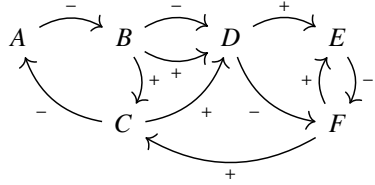
\[
\begin{tikzcd}
A \arrow[r, "-", bend left] & B \arrow[d, "+", bend left] \arrow[r, "+"', bend right] \arrow[r, "-", bend left] & D \arrow[r, "+", bend left] \arrow[rd, "-"', bend right] & E \arrow[d, "-", bend left]                            \\
& C \arrow[lu, "-", bend left] \arrow[ru, "+"', bend right]         &                                                         
& F \arrow[u, "+", bend left] \arrow[ll, "+", bend left]
\end{tikzcd}
\]
In a sense, there is a copy of $(G,\ell)$ hiding in $(H,m)$.   Here it is:
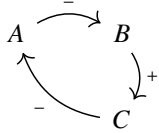
\[
\begin{tikzcd}
A \arrow[r, "-", bend left] & B \arrow[d, "+", bend left]                    \\
& C \arrow[lu, "-", bend left]          
\end{tikzcd}
\]
Here the entity $A$ affects itself in a positive way, but \emph{indirectly}, through $B$ and $C$.  There is not a map of $M$-labeled graphs from $(G,\ell)$ to $(H,m)$.  Instead, there is a functor from the free $M$-graded category on $(G,\ell)$ to the free $M$-graded category on $(H,m)$, say
\[    F \maps \Free_M(G,\ell) \to \Free_M(H,m) .\]
This functor sends $X$ to $A$ and it sends the edge from $X$ to itself, which is a morphism in $\Free_M(G,\ell)$, to the path from $A$ to $B$ to $C$ to $A$, which is a morphism in $\Free_M(H,m)$.  Moreover $F$ is a map of $M$-graded categories.  This is the sense in which the motif $(G,\ell)$ appears in the larger $M$-labeled graph $(H,m)$.

To make this precise we need to understand the free $M$-graded category on an $M$-labeled graph.  

\begin{lem}
\label{Lem:Free_monoid-labeled_category}
Let $M$ be a monoid.  The forgetful functor sending $M$-labeled categories to their underlying $M$-labeled graphs,
\[   \Und_M \maps \cat/\B M \to \Gph/G_M ,\] 
has a left adjoint 
\[   \Free_M \maps \Gph/G_M \to \cat/\B M .\]
\end{lem}

\begin{proof}
We define $\Free_M$ as follows.  For any $M$-labeled graph $(G,\ell)$ we define $\Free_M(G,\ell)$ to have $\Free(G)$ as its underlying category, with the labeling of any morphism
\[  v_0 \xrightarrow{e_1} v_1 \xrightarrow{e_2} \cdots \xrightarrow{e_{n-1}} v_{n-1} \xrightarrow{e_n} v_n \]
in $\Free(G)$ being the product $\ell(e_n) \cdots \ell(e_1) \in M$.    Recall that a map of $L$-labeled graphs $f \maps (G,\ell) \to (G',\ell')$ is a map $f_0$ sending vertices to vertices and a map $f_1$ sending edges to edges, preserving the labeling.  Thus we define $\Free_M(f)$ to equal $f_0$ on objects of $\Free_M(G,\ell)$ and to send any morphism
\[  v_0 \xrightarrow{e_1} v_1 \xrightarrow{e_2} \cdots \xrightarrow{e_{n-1}} v_{n-1} \xrightarrow{e_n} v_n \]
to 
\[  f_0(v_0) \xrightarrow{f_1(e_1)} f_0(v_1) \xrightarrow{f_1(e_2)} \cdots \xrightarrow{f_1(e_{n-1})} f_0(v_{n-1}) \xrightarrow{f_1(e_n)} f_0(v_n).\]
One can check that $\Free_M$ is left adjoint to $\Und_M$.
\end{proof}

\begin{defn}
\label{Defn:Kleisli_category}
For any monoid $M$, the \define{Kleisli category of $M$-labeled graphs}, $\K(\Gph/G_M)$ has:
\begin{itemize}
\item as objects, $M$-labeled graphs,
\item as morphisms from an $M$-labeled graph $(G,\ell)$ to an $M$-labeled graph $(G', \ell')$, maps of $M$-graded categories
\[    f \maps \Free_M(G,\ell) \to \Free_M(G',\ell'), \]
\item as composition, the usual composition of maps of $M$-graded categories,
\item as identity morphism, the usual identity maps of $M$-graded categories.
\end{itemize}
We call a morphism in this category a \define{Kleisli morphism} between $M$-labeled graphs.
\end{defn}

The concept of Kleisli category shows up whenever we have a pair of adjoint functors.  Usually the Kleisli category is described in terms of a monad, which here would be the monad $\Und_M \circ \Free_M$.  There is a bijection between morphisms
\[   \Free_M(G,\ell) \to \Free_M(G',\ell') \]
and morphisms
\[   (G,\ell) \to (\Und_M \circ \Free_M)(G',\ell') . \]
It is common to describe composition of morphisms of the former sort indirectly, in terms of morphisms of the latter sort.  However, this is not needed for our purposes.

Aduddell et al. \cite{Aduddell} introduced the Kleisli category of $\{+,-\}$-labeled graphs to study motifs in biochemical regulatory networks. 

\begin{eg}
Below we list some $\{+,-\}$-labeled graphs with special names, which serve as motifs in the work of Aduddell et al.\ \cite{Aduddell} and Tyson et al.\ \cite{tyson2010functional} :

\begin{itemize}
\item \define{positive autoregulation} and \define{negative autoregulation}:
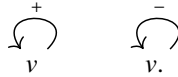
\[
\begin{tikzcd}
v \arrow["+"', loop, distance=2em, in=125, out=45]
\end{tikzcd}
\quad
\begin{tikzcd}
v. \arrow["-"', loop, distance=2em, in=125, out=45]
\end{tikzcd} 
\]
\item \define{positive stimulation} and \define{negative stimulation}:
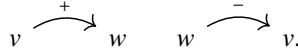
\[
\begin{tikzcd}
v \arrow[r, "+", bend left] & w
\end{tikzcd}
\quad
\begin{tikzcd}
w \arrow[r, "-", bend left] & v.
\end{tikzcd}
\]
\item the \define{positive feedback 
loop}, \define{negative feedback loop} and \define{double-negative feedback loop}:
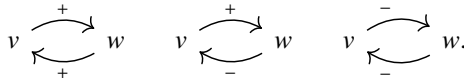
\[
\begin{tikzcd}
v \arrow[r, "+", bend left] & w \arrow[l, "+", bend left]
\end{tikzcd}
\quad
\begin{tikzcd}
v \arrow[r, "+", bend left] & w \arrow[l, "-", bend left]
\end{tikzcd}
\quad
\begin{tikzcd}
v \arrow[r, "-", bend left] & w. \arrow[l, "-", bend left]
\end{tikzcd}
\]
\item
the \define{coherent feedforward loop} and \define{incoherent feedforward loop}:
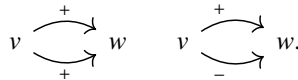
\[
\begin{tikzcd}
v \arrow[r, "+", bend left] \arrow[r, "+"', bend right] & w
\end{tikzcd}
\quad
\begin{tikzcd}
v \arrow[r, "+", bend left] \arrow[r, "-"', bend right] & w.
\end{tikzcd}
\]
\item
three kinds of \define{branches}:
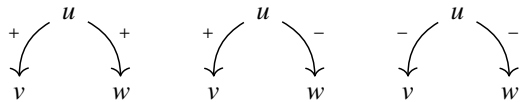
\[
\begin{tikzcd}[column sep=0.5em]
& u \arrow[ld, "+"', bend right] \arrow[rd, "+", bend left] &   \\
v                                                       && w
\end{tikzcd}
\qquad
\begin{tikzcd}[column sep=0.5em]
& u \arrow[ld, "+"', bend right] \arrow[rd, "-", bend left] &   \\
v                                                       && w
\end{tikzcd}
\qquad
\begin{tikzcd}[column sep=0.5em]
& u \arrow[ld, "-"', bend right] \arrow[rd, "-", bend left] &   \\
v                                                       && w
\end{tikzcd}
\]
\item
three kinds of \define{logic gates}:
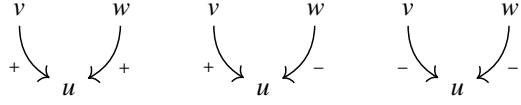
\[
\begin{tikzcd}[column sep=0.5em]
v \arrow[dr, "+", bend right, swap] && w \arrow[dl, "+"', bend left, swap] \\
& u
\end{tikzcd}
\qquad
\begin{tikzcd}[column sep=0.5em]
v \arrow[dr, "+", bend right, swap] && w \arrow[dl, "-"', bend left, swap] \\
& u
\end{tikzcd}
\qquad
\begin{tikzcd}[column sep=0.5em]
v \arrow[dr, "-", bend right, swap] && w \arrow[dl, "-", bend left] \\
& u
\end{tikzcd}
\]
\end{itemize}
A third feedforward loop not mentioned by the above authors could be called the
\define{double-negative feedforward loop}:
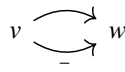
\[
\begin{tikzcd}
v \arrow[r, "-", bend left] \arrow[r, "-"', bend right] & w
\end{tikzcd}
\]
Starting from the fundamental motifs listed above, one can build other important ones, such as the \define{overlapping feedforward loops} formed by combining feedforward loops with branches or logic gates:
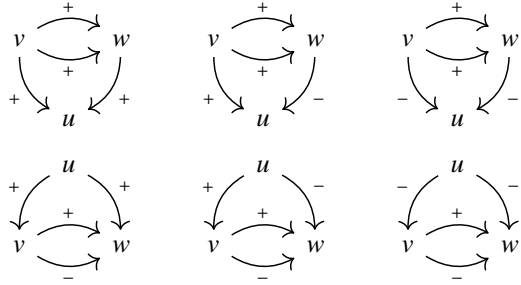
\[
\begin{tikzcd}[column sep=0.5em]
v \arrow[dr, "+", swap, bend right] \arrow[rr, "+"', bend right] \arrow[rr, "+", bend left] &&
w \arrow[dl, "+"', swap, bend left]
\\
& u
\end{tikzcd}
\qquad
\begin{tikzcd}[column sep=0.5em]
v \arrow[dr, "+", swap, bend right] \arrow[rr, "+"', bend right] \arrow[rr, "+", bend left] &&
w \arrow[dl, "-"', swap, bend left]
\\
& u
\end{tikzcd}
\qquad
\begin{tikzcd}[column sep=0.5em]
v \arrow[dr, "-", swap, bend right] \arrow[rr, "+"', bend right] \arrow[rr, "+", bend left] &&
w \arrow[dl, "-"', swap, bend left]
\\
& u
\end{tikzcd}
\]
\[
\begin{tikzcd}[column sep=0.5em]
& u \arrow[dl, "+"', bend right] \arrow[dr, "+", bend left] &   \\
v \arrow[rr, "-"', bend right] \arrow[rr, "+", bend left]  && w
\end{tikzcd}
\qquad
\begin{tikzcd}[column sep=0.5em]
& u \arrow[dl, "+"', bend right] \arrow[dr, "-", bend left] &   \\
v \arrow[rr, "-"', bend right] \arrow[rr, "+", bend left]  && w
\end{tikzcd}
\qquad
\begin{tikzcd}[column sep=0.5em]
& u \arrow[dl, "-"', bend right] \arrow[dr, "-", bend left] &   \\
v \arrow[rr, "-"', bend right] \arrow[rr, "+", bend left]  && w
\end{tikzcd}
\]
\end{eg}

For a monoid $M$, denote the Kleisli category of $M$-labeled graphs (see \cref{Defn:Kleisli_category}) by $\K(\Gph/G_M)$. Given a homomorphism $\phi \maps M \to M'$ of monoids, there is a functor 
\[ 
\begin{array}{cccl}
\phi_\ast \maps & \K(\Gph/G_M) &\to& \K(\Gph/G_{M'}) \\ \\
& (G, \ell) &\mapsto& (G, \phi \circ \ell)\\
&\Big( F \maps \Free_M(G,\ell_1) \to \Free_M(H,\ell_2) \Big) &\mapsto& \Big( F \maps \Free_{M'}(G,\phi \circ \ell_1) \to \Free_{M'}(H,\phi \circ \ell_2) \Big)
\end{array}
\]
which allows us to change labelings on the edges  \emph{functorially}, as we observe below:

\begin{prop}
\label{Prop:Change_of_labeling_monoid}
The following assignment
\[
\begin{array}{cccl}
F_{\mathrm{m}} \maps &  \Mon &\to& \cat \\ \\
& M &\mapsto& \K(\Gph/G_M) \\
&(M \xrightarrow{\phi} M') &\mapsto& (\K(\Gph/G_M) \xrightarrow{\phi_\ast} \K(\Gph/G_{M'}))
\end{array}
\]
defines a functor, where $\Mon$ is the category of monoids.
\end{prop}
\begin{proof}
 This follows because for any composable pair of monoid homomorphisms $M \xrightarrow{\phi} M' \xrightarrow{\phi'} M''$ we have 
\[\phi'_\ast(G, \phi \circ \ell)= (G, \phi' \circ \phi \circ \ell) . \qedhere \]
\end{proof}

Applying the Grothendieck construction to the functor in \cref{Prop:Change_of_labeling_monoid}, we obtain the \define{category of monoid-labeled graphs and Kleisli morphisms} $\int F_{\text{m}}$.    This bundles up the categories of $M$-labeled graphs and Kleisli morphisms for all monoids $M$ into a single category.
\begin{prop}
\label{Prop:Category_of_monoid-labeled_graphs}
$\int F_{\mathrm{m}}$ is the category where:
\begin{itemize}
\item an object is a pair $(M, p \maps G \to G_M)$,
\item a morphism from $(M, p \maps G \to G_M)$ to $(M', p' \maps G' \to G_{M'})$ is a pair $(\phi,\theta)$, where $\phi \maps M \to M'$ is a homomorphism of monoids and $\theta \maps \Free(G) \to \Free(G')$ is a functor such that this square commutes:
\[
\begin{tikzcd}
\Free(G) \arrow[r, "\theta"] \arrow[d, "\tilde{p}"'] & \Free(G') \arrow[d, "\tilde{p'}"] \\
\B M \arrow[r, "\tilde{\phi}"']                & \B M'          
\end{tikzcd}
\]
where 
\begin{itemize}
\item the functor $\tilde{\phi} \maps \B M \to \B M' :=\B(\phi)$, where $\B \maps \Mon \to \cat$ is the functor that takes a monoid to its associated 1-object category,
\item $\tilde{p} \colon \Free(G) \to \B M$ is induced from from $p \colon G \to G_M$, which takes any morphism
\[  v_0 \xrightarrow{e_1} v_1 \xrightarrow{e_2} \cdots \xrightarrow{e_{m-1}} v_{m-1} \xrightarrow{e_m} v_m \]
in $\Free(G)$ to the unique morphism in $\B M$ associated with
\[   \ell(e_m) \cdots \ell(e_1) \in M ,\] 
\item $\tilde{p}' \colon \Free(G') \to \B M'$ is defined similarly to $\tilde{p}$.
\end{itemize}
\end{itemize}
\end{prop}

\section{Commutative monoid-labeled graphs}
\label{Sec:Commutative_monoid-labeled_graphs}

Now we turn to graphs with edges labeled by elements of a \emph{commutative} monoid.  This allows for a new concept of morphism between labeled graphs, and permits a deeper study of feedback loops.   We use additive notation for the operation in a commutative monoid.

Given a commutative monoid $C$, we define $C$-labeled graphs and maps of $C$-labeled graphs just as we do as when $C$ is a mere set (Definitions \ref{Defn:Set-labeled-graph} and \ref{Defn:morphisms-of-set-labeled-graphs}).   However, there is another useful concept of morphism between \define{finite} $C$-labeled graphs, meaning those with finitely many edges and vertices.   (In applications, $C$-labeled graphs are usually finite.)

\begin{defn}
\label{Defn:Morphism-of-commutative-monoid-labeled-graphs}
Let $C$ be a commutative monoid, and let $(G,\ell) = (E,V,s,t,\ell)$ and $(G',\ell') =(E',V',s',t',\ell')$ be two finite $C$-labeled graphs.  We define an \textbf{additive morphism} from $(G,\ell)$ to $(G',\ell')$ to be a map of graphs $f \maps G \to G'$ such that the following condition holds for all $e' \in E'$:
\begin{equation}
\ell'(e')= \sum_{\lbrace e \in E \maps f_1(e)= e'\rbrace} \ell(e).
\label{Eq:Pushforward_along_a_map_of_finite_graphs}
\end{equation}
The collection of finite $C$-labeled graphs and their additive morphisms forms a category whose composition law is induced from the composition law in the category $\Fin\Gph$.  We denote this category by $C\Fin\Gph$.
\end{defn}

Equation \cref{Eq:Pushforward_along_a_map_of_finite_graphs} says that each edge $e'$ of $G'$ is labeled by the sum of the labels of edges of $G$ that map to $e'$.   We restrict ourselves to commutative monoids to make sure that the sum is independent of the order of summation, and to finite graphs to make sure the sum is a finite one. Furthermore, observe that $\ell(e')=0$ if there is no $e \in E$ such that $f(e)=e'$, since any sum over the empty set vanishes. 

We can push forward an $C$-labeling along any morphism of finite graphs. More precisely, we have the following:

\begin{lem}
\label{Lem:Pushforward_along_a_map_of_finite_graphs}
Let $C$ be a commutative monoid and $(G,\ell)$ a finite $C$-labeled graph.  Then, for any map of finite graphs $f \maps G \to G'$, there is a unique $C$-labeling $f_\ast \ell$ of $G'$, called the \define{pushforward of $\ell$ along $f$}, such that $f$ is an additive morphism of $C$-labeled graphs from $(G,\ell)$ to $(G',f_\ast \ell)$.   This pushforward is covariantly functorial in the following sense:
\[      (f \circ g)_\ast = f_\ast \circ g_\ast, \qquad 1_\ast = 1 . \]
\end{lem}

\begin{proof}
By the definition of additive morphism we are forced to take
\[  (f_\ast \ell)(e'):= \sum_{\lbrace e \in E \maps f_1(e)= e'\rbrace} \ell(e) \]
for all edges $e'$ of $G'$, and this choice works.
\end{proof}

To see how this fact can be used in system dynamics or systems biology, suppose that each hour a coffee shop is open it sells $\$ 150$ of coffee and $\$ 25$ of tea.  We can model this with 
an $(\R,+,0)$-labeled graph:
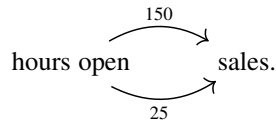
\[
\begin{tikzcd}[row sep = "0.1 em"]
 \mathrm{hours\; open}  \arrow[r, bend left, pos = 0.5, "150"]   \arrow[r, bend right, swap, pos = 0.45, "25"] & \mathrm{sales.}
\end{tikzcd}
\]
We can simplify this model by mapping its underlying graph to a simpler graph:
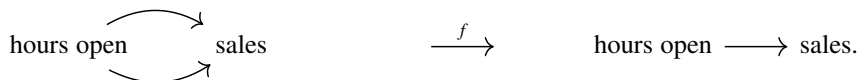
\[
\begin{tikzcd}[row sep = "0.1 em"]
 \mathrm{hours\; open}  \arrow[r, bend left, ""]   \arrow[r, bend right, swap, ""] & \mathrm{sales} 
 & & {} \arrow[r, "f"] & {} &   \mathrm{hours\; open}  \arrow[r, ""] & \mathrm{sales.}
\end{tikzcd}
\]
\cref{Lem:Pushforward_along_a_map_of_finite_graphs} says there is a unique way to label the edges of the 
second graph that lets us lift the map $f$ to an additive morphism of $(\R,+,0)$-labeled graphs:
\[
\begin{tikzcd}[row sep = "0.1 em"]
 \mathrm{hours\; open}  \arrow[r, bend left, "150"]   \arrow[r, bend right, pos = 0.45, swap, "25"] & \mathrm{sales} 
 & & {} \arrow[r, "f"] & {} &   \mathrm{hours\; open}  \arrow[r, "175"] & \mathrm{sales.}
\end{tikzcd}
\]

Just as \cref{Lem:Pullback_along_a_map_of_graphs} says that for any set $L$ the forgetful functor from $L$-labeled graphs to graphs is a discrete fibration, \cref{Lem:Pushforward_along_a_map_of_finite_graphs} says that for any commutative monoid $C$ the forgetful functor from $C$-labeled finite graphs to finite graphs is a discrete opfibration.

\begin{prop}
\label{Prop:Commutative_monoid-labeled_graph}
For any commutative monoid $C$, we have the following:
\begin{itemize}
\item[(a)] The forgetful functor $U \maps C\Fin\Gph \to  \Fin\Gph$ is a discrete opfibration.
\item[(b)] The category of elements $\int \Q$ of the covariant functor $\Q \maps \Fin\Gph \to \Set$ associated with $U$ is equivalent to the category $C\Fin\Gph$.
\end{itemize}
\end{prop}

\begin{proof}
(a) was proved in \cref{Lem:Pushforward_along_a_map_of_finite_graphs}.  (b) follows from general principles, but one can see it concretely as follows.  First note that the functor $\Q \maps \Fin\Gph \to \Set$ associated with the discrete fibration $U \maps C\Fin\Gph \to \Fin\Gph$ maps any finite graph to its set of $C$-labelings, and any map of finite graphs $f \maps G \to G'$ to the map sending $C$-labelings of $G$ to $C$-labelings of $G'$ given by Equation \eqref{Eq:Pushforward_along_a_map_of_finite_graphs}.  Thus, the category of elements of $\Q$ is equivalent to $C\Fin\Gph$. 
\end{proof}

 \cref{Prop:Commutative_monoid-labeled_graph} (a) implies that there is a functorial way to push forward $C$-labelings along maps of graphs.  We can also push forward labelings along a homomorphism of commutative monoids. Given a homomorphism $\phi \maps C \to C'$ of commutative monoid and an $C$-labeled finite graph $(G,\ell)$, we can define an $C'$-labeled graph 
\begin{equation}
\label{Eq:Change_of_labeling_commutative_monoid} 
     \phi_\ast(G,\ell) = (G, \phi \circ \ell) .
\end{equation}

\begin{prop}
\label{Prop:Change_of_labeling_commutative_monoid}
The following assignment
\[
\begin{array}{cccl}
F_{\mathrm{cm}} \maps &  \Comm\Mon &\to& \cat \\ \\
& C &\mapsto& C\Fin\Gph \\
&(C \xrightarrow{\phi} C') &\mapsto& (C\Fin\Gph \xrightarrow{\phi_\ast} C'\Fin\Gph)
\end{array}
\]
defines a functor, where $\Comm\Mon$ is the category of commutative monoids.
\end{prop}

The functors $\phi_\ast$ have many practical applications:

\begin{eg}
\label{Eg:Map_to_terminal_monoid}
Every commutative monoid $C$ has a unique homomorphism $\phi \maps C \to 1$ where $\{1\}$ is the terminal monoid discussed in \cref{Eg:1}.    The resulting functor
\[   \phi_\ast \maps C\Fin\Gph \to \{1\}\Fin\Gph \cong \Gph \]
takes any $C$-labeled finite graph and discards the labeling, giving a finite graph.   This can be used to discard information about \emph{how} one vertex directly affects another and merely retain the fact \emph{that} it does.
\end{eg}

\begin{eg}
\label{Eg:Map_to_Z/2}
There is a homomorphism $\phi \maps \R - \{0\} \to \{+,-\}$ from the multiplicative group of the reals (see \cref{Eg:multiplicative_group_of_reals}) to the group $\{+,-\}$ (see \cref{Eg:Z/2}) sending all positive numbers to $+$ and all negative numbers to $-$.   The resulting functor $\phi_\ast$ turns \emph{quantitative} information about how much one vertex directly affects another into purely \emph{qualitative} information.
\end{eg}

\begin{eg}
\label{Eg:Map_to_Z/3}
There is a homomorphism $\phi \maps \R \to \{+,0,-\}$ from the multiplicative monoid of the reals (see \cref{Eg:multiplicative_monoid_of_reals}) to the monoid $\{+,0,-\}$ (see \cref{Eg:Z/3}) sending all positive numbers to $+$, all negative numbers to $-$, and $0$ to $0$.   The resulting functor $\phi_\ast$ again turns quantitative information into qualitative information.
\end{eg}

\begin{eg}
\label{Eg:Right_inverses}
The homomorphisms in Examples \ref{Eg:Map_to_terminal_monoid}--\ref{Eg:Map_to_Z/3} all have right inverses. For example, there is a homomorphism $\psi \maps \{+,0,-\} \to \R$ sending $+$ to $1$, $-$ to $-1$ and $0$ to $0$, and this has
\[                  \phi \circ \psi = 1  .\]
The functor $\psi_\ast$ can be used to convert qualititative information about how one vertex directly  affects another into quantitative information in a simple, default manner.  Of course this should be taken with a grain of salt: since $\psi \circ \phi \ne 1$, quantitative information that has been converted into qualitative information cannot be restored.
\end{eg}

Applying the Grothendieck construction to the functor in \cref{Prop:Change_of_labeling_commutative_monoid} we obtain the \define{category of commutative monoid-labeled graphs} $\int F_{\text{cm}}$.

\begin{prop}
\label{Prop:Category_of_commutative_monoid_labeled_graphs}
$\int F_{\mathrm{cm}}$ is the category where:
\begin{itemize}
\item an object is a triple $(C, G, \ell)$ where $C$ is a commutative monoid and $(G,\ell)$ is an $C$-labeled graph;
\item a morphism from $(C, G, \ell)$ to $(C', G', \ell')$ is a pair $(\phi,f)$, where $\phi \maps C \to C'$ is a homomorphism of commutative monoids and $f \maps G \to G'$ is a map of finite graphs such $f_\ast(\phi_\ast \ell)= \ell'$, where $f_\ast$ is defined in Lemma \ref{Lem:Pushforward_along_a_map_of_finite_graphs} and $\phi_\ast$ is defined in Equation (\ref{Eq:Change_of_labeling_commutative_monoid}).
\end{itemize}
\end{prop}

\section{Rig-labeled graphs}
\label{Sec:Rig-labeled_graphs}

We have seen that passing from set-labeled graphs (\cref{Sec:Set-labeled_graphs}) to monoid-labeled graphs (\cref{Sec:Monoid-labeled_graphs}) lets us study of how one entity affects another indirectly through a path of edges, and gives a general way to analyze motifs (\cref{Sec:Motifs}).   When the labeling monoid is commutative, we can also define `additive morphisms' between finite labeled graphs, which can be used to describe ways of simplifying labeled graphs (\cref{Sec:Commutative_monoid-labeled_graphs}). In the commutative case we can also study feedback loops using homology theory (\cref{Sec:Feedback_loops}).

Given all this, it is mathematically tempting to study graphs whose edges are labeled by elements of a rig.  A \define{rig} is a set $R$ with the structure of both a commutative monoid $(R,+,0)$ and a monoid $(R, \cdot, 1)$, obeying 
\[           r \cdot (s + t) = r \cdot s + r \cdot t , \qquad (r + s) \cdot t = r \cdot t + s \cdot t \]
\[            0 \cdot r = 0 = r \cdot 0 \]
for all $r,s,t \in R$.  It has this name because it is a `ring without negatives', or more precisely a ring
that may not have negatives.   The classic example is $\N$ with its usual addition and multiplication.

Despite the mathematically natural quality of rigs, their use of rig elements as polarities seems new and is somewhat speculative.  What is the point of having two operations on the set of labels? 

\begin{eg}
The initial object in the category of rigs is $\N$ with its usual addition and multiplication.  Given an $\N$-labeled graph, we can interpret an edge labeled by $n \in \N$
\[                        v \xrightarrow{n} w \]
as saying there are $n$ ways for $v$ to directly affect $w$.   Given a path in an $\N$-labeled graph, say
\[                      v_0 \xrightarrow{n_1} v_1 \xrightarrow{n_2} v_2 \xrightarrow{n_3} \cdots \xrightarrow{n_k} v_k \]
we can argue that there are $n_1 \cdots n_k$ ways for $v_0$ to affect $v_k$.   This uses multiplication in the rig $\N$.    Note also that using multiplication as the monoid operation we have a Kleisli morphism (\cref{Defn:Kleisli_category}) from 
\[                     v_0 \xrightarrow{n_1 \cdots n_k} v_k\]
to
\[                      v_0 \xrightarrow{n_1} v_1 \xrightarrow{n_2} v_2 \xrightarrow{n_3} \cdots \xrightarrow{n_k} v_k. \]
On the other hand, we can use addition to define additive morphisms (\cref{Defn:Morphism-of-commutative-monoid-labeled-graphs}).  Then there is an additive morphism from
\[
\begin{tikzcd}
v \arrow[rr, "n_1", bend left = 40] \arrow[rr,  "n_2"] \arrow[rr, "n_3"', bend right]  && w
\end{tikzcd}
\]
to
\[
\begin{tikzcd}
v \arrow[rr,  "n_1 + n_2 + n_3"]   && w .
\end{tikzcd}
\]
This is consistent with the idea that there are $n_1 + n_2 + n_3$ ways for $v$ to affect $w$ in the first $\N$-labeled graph, and that the second $\N$-labeled graph presents this information in a simplified manner.
\end{eg}

\begin{eg}
\label{Eg:B}
The boolean rig $\lB = \{F,T\}$ has `or' as addition, `and' as multiplication, $F$ as $0$, and $T$ as $1$.    Equivalently, we can take $\lB = \{0,1\}$ with the following addition and multiplication:
\vskip 0.5em
\renewcommand\arraystretch{1.3}
\setlength\doublerulesep{0pt}
\begin{center}
\begin{tabular}{c||c|c}
+ & 0 & 1 \\
\hline\hline
0 & 0 & 1 \\ 
\hline
1 & 1 & 1 \\ 
\end{tabular}
\qquad  \qquad
\begin{tabular}{c||c|c}
$\cdot$ & 0 & 1 \\
\hline\hline
0 & 0 & 0 \\ 
\hline
1 & 0 & 1 \\ 
\end{tabular}
\end{center}
\vskip 0.5em
We can use:
\begin{itemize}
\item the absence of an edge from $v$ to $w$ to mean that there is \emph{no knowledge of whether} $v$ directly affects $w$,
\item an edge labeled by $0$ from $v$ to $w$ to indicate the \emph{known absence of a direct effect}
of some specific sort of $v$ on $w$, 
\item an edge labeled by $1$ from $v$ to $w$ to indicate that there is \emph{known presence of a direct effect} of some specific sort of $v$ on $w$.
\end{itemize}
For example, 
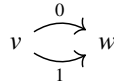
\[
\begin{tikzcd}[column sep = 1em]
v \arrow[rr, "0", bend left]  \arrow[rr, "1"', bend right]  && w
\end{tikzcd}
\]
means it is known that $v$ does not affect $w$ in one way, but that it does affect it in some other way.  We can simplify this using an additive morphism to
\[
\begin{tikzcd}[column sep = 1em]
v  \arrow[rr, "1"]  && w.
\end{tikzcd}
\]
\end{eg}

\begin{eg}
We can generalize \cref{Eg:B} as follows.  A \define{quantale} is defined to be a poset $Q$ with least upper bounds of arbitrary subsets, equipped with a multiplication $\cdot \maps Q \times Q \to Q$ that preserves least upper bounds in each argument.    A quantale is \define{unital} if it has a unit $1$ for the multiplication.   Any unital quantale becomes a rig where the addition is defined to be the binary join $\vee \maps Q \times Q \to Q$.  

The boolean rig $\lB$ is an example of a unital quantale, but there are many others.  For example, $\lB^n$ is a unitary quantale that allows us to generalize \cref{Eg:B} to a situation where there are $n$ researchers, each with their own research on whether one vertex affects another.  Master has carried out a detailed study of networks using quantale theory \cite{MasterThesis, Master2021}, which overlaps in interesting ways with our work here.   Instead of considering an arbitrary graph with edges labeled by elements of a rig, she considers a set of vertices $V$ together with a map $\ell \maps V \times V \to Q$ for some quantale $Q$.  We can  can think of this as a $Q$-labeled complete directed graph.   In this approach the absence of an effect of $v \in V$ on $w \in V$ is indicated, not by the absence of an edge from $v$ to $w$, but by setting $\ell(v,w) = 0$, where $0 \in Q$ is the bottom element (the least upper bound of the empty set).
\end{eg}

We have seen that the multiplicative monoid of $\Z/3$ is useful for describing positive and negative effects as well as the absence of an effect (\cref{Eg:Z/3}).   While $\Z/3$ becomes a ring with its usual addition, this is problematic in the applications we are considering because $-1 + 1 = 0$, suggesting that a positive effect necessarily cancels a negative effect.   To get around this we can introduce a new element, $i$, for `indeterminate', and set $-1 + 1 = i$.  

\begin{eg}
\label{Eg:S}
There is a 4-element commutative rig $S = \{1,0,-1,i\}$ with addition and multiplication given as follows:
\vskip 0.5em
\renewcommand\arraystretch{1.3}
\setlength\doublerulesep{0pt}
\begin{center}
\begin{tabular}{c||c|c|c|c}
$+$ & $1$ & $0$ & $-1$ & $i$ \\
\hline\hline
$1$ & $1$ & $1$ & $i$ & $i$ \\
\hline
$0$ & $1$ & $0$ & $-1$ & $i$ \\
\hline
$-1$ & $i$ & $-1$ & $-1$ & $i$ \\
\hline
$i$ & $i$ & $i$ & $i$ & $i$ \\
\end{tabular}
\qquad  \qquad
\begin{tabular}{c||c|c|c|c}
$\cdot$ & $1$ & $0$ & $-1$ & $i$ \\
\hline\hline
$1$ & $1$ & $0$ & $-1$ & $i$ \\
\hline
$0$ & $0$ & $0$ & $0$ & $0$ \\
\hline
$-1$ & $-1$ & $0$ & $1$ & $i$ \\
\hline
$i$ & $i$ & $0$ & $i$ & $i$ \\
\end{tabular}
\end{center}
\vskip 0.5em
We can use $1$ to indicate a positive effect, $-1$ to indicate a negative effect, $0$ to indicate the absence of an effect, and $i$ to indicate an indeterminate effect: one that could be positive, negative or absent.   The addition and multiplication rules in this rig capture the following intuitions:
\begin{itemize}
\item 
the sum of positive effects is positive, while the sum of negative effects is negative,
\item 
the sum of a positive and a negative effect is indeterminate,
\item
the product of positive effects is positive, as is the product of negative effects,
\item
the product of a positive and a negative effect is negative,
\item 
any operation applied to an indeterminate effect produces an indeterminate effect, except for multiplication by $0$.
\end{itemize}
To check that $S$ really is a rig, we can use a construction of Golan \cite[Ex.\ 1.10]{Golan1999}.  For any monoid $M$, its  power set $PM$ becomes a rig with union as addition and with multiplication defined by
\[   X \cdot Y = \{ x y \; \vert \; x \in X, y \in Y \}  \]
for $X, Y \in PM$.   The rig $S$ can then be seen as $P(\Z/2)$ using the following identifications, where we write
the group $\Z/2$ multiplicatively as $\{+,-\}$ as in \cref{Eg:Z/2}:
\[              0 = \{ \}, \qquad 1 = \{+\}, \qquad -1 = \{-\}, \qquad i = \{+,-\}   .\] 
\end{eg}

\begin{eg}
We can generalize \cref{Eg:S} as follows.   For any monoid $M$ and any commutative unital quantale $Q$, the set $Q^M$ becomes a rig where addition is defined pointwise:
\[    (f + g)(x) = f(x) \vee g(x)  , \qquad x \in M, \]
where $\vee$ stands for the greatest lower bound, and multiplication is defined by a kind of convolution:
\[     (f \cdot g)(x) = \bigvee_{\{y, z \in M \, \vert \, x = y z \}}  f(y) \cdot g(z) , \qquad x \in M. \]
It is interesting to apply this construction to any quantale $Q$, thought of as describing `generalized truth values', and any monoid $M$ from Examples \ref{Eg:1}--\ref{Eg:semiautomata}.   \cref{Eg:S} arises from taking $Q = \lB$ and $M = \Z/2$.
\end{eg}

Finally, there is the obvious example:

\begin{eg}
The ring $\R$ of real numbers
 with its usual addition and multiplication can be used to describe effects in a quantitative rather than purely qualitative way.
\end{eg}

\section{Open labeled graphs}

Experience has shown that `open' systems---systems that can interact with their environments---are well modeled using cospans \cite{Baez2025,FongThesis}.   A \define{cospan} in some category $\A$ is a diagram of this form:
\[
\begin{tikzpicture}[scale=1]
\node (A) at (0,0) {$A$};
\node (B) at (1,1) {$X$};
\node (C) at (2,0) {$B$};
\path[->,font=\scriptsize,>=angle 90]
(A) edge node[above, pos = 0.3]{$i$} (B)
(C) edge node[above, pos = 0.3]{$o$} (B);
\end{tikzpicture}
\]
We call $X$ the \define{apex}, $A$ and $B$ the \define{feet}, and $i$ and $o$ the \define{legs} of the cospan.   The apex describes the system itself.  The feet describe `interfaces'  through which the system can interact with the outside world.    The legs describe how the interfaces are included in the system.   If the category $\A$ has finite colimits, we can compose cospans using pushouts and tensor them using coproducts.  Composition describes the operation of attaching two open systems together in series by identifying one interface of the first with one of the second.  Tensoring describes setting open systems side by side, in parallel.   

This suggests that cospans in $\A$ are the morphisms in some symmetric monoidal category.  However,  we can also define morphisms \emph{between} cospans, and some of the laws of a symmetric monoidal category hold only up to invertible morphisms of this kind.   Thus, we say cospans in $\A$ are `horizontal 1-cells' in a more elaborate structure, called a `symmetric monoidal double category', in which morphisms between cospans are the `2-morphisms'.    Symmetric monoidal double categories take a while to get used to, but they nicely capture many of the operations on cospans, and the laws they obey.   To learn more about symmetric monoidal double categories, see \cite{CourserThesis, HS, Shulman2008}.

However, we often want to study systems that have more structure than their interfaces.  We can sometimes model these using structured cospans \cite{FiadeiroSchmitt}.   Here we start with a functor $F \maps \A \to \X$, typically the left adjoint of some functor $R \maps \X \to \A$ that forgets extra structure possessed by objects of $\X$ but not by those of $\A$.   Then a \define{$F$-structured cospan} is a diagram in $\X$ of the form:
\[
\begin{tikzpicture}[scale=1.2]
\node (A) at (0,0) {$F(A)$};
\node (B) at (1,1) {$X$};
\node (C) at (2,0) {$F(B)$.};
\path[->,font=\scriptsize,>=angle 90]
(A) edge node[above, pos=0.3]{$i$} (B)
(C)edge node[above, pos=0.3]{$o$}(B);
\end{tikzpicture}
\]
When $\A$ and $\X$ have finite colimits and $F$ preserves them, there is a symmetric monoidal double
category where the horizontal 1-morphisms are $F$-structured cospans \cite{BaezCourser2020,BaezCourserVasilakopoulou2022, CourserThesis, Patterson2023}.

One of the simplest examples of this theory involves open $L$-labeled graphs for some set $L$.  For this we use the adjoint functors
\[
\begin{tikzcd}
            \ \phantom{X} \Set\phantom{X} \arrow[r, bend left = 20, shift left=1ex, "{\disc}"{name=G}] & \ \Gph/G_L \arrow[l, bend left = 20, shift left=.5ex, pos = 0.4, "\Vert"{name=F}]
            \arrow[phantom, from=F, to=G, "\scriptstyle\boldsymbol{\bot}"].
        \end{tikzcd}
    \]
where
\begin{itemize} 
\item $\disc \maps \Set \to \Gph/G_L$ takes a set $S$ to the unique $L$-labeled graph with vertex set 
$S$ and no edges;
\item $ \Vert \maps \Gph/G_L \to \Set$ takes an $L$-labeled graph to its set of vertices.
 \end{itemize}
 
\begin{defn}
\label{Defn:Open_set-labeled_graphs}
Given a set $L$, an \define{open} $L$\define{-labeled graph} is a diagram in $\Gph/G_L$ of the form
\[
\begin{tikzpicture}[scale=1.2]
\node (A) at (0,0) {$\disc(A)$};
\node (B) at (1,1) {$X$};
\node (C) at (2,0) {$\disc(B)$};
\path[->,font=\scriptsize,>=angle 90]
(A) edge node[above, pos=0.3]{$i$} (B)
(C)edge node[above, pos=0.3]{$o$}(B);
\end{tikzpicture}
\]
for some sets $A$ and $B$.   We call this an open $L$-labeled graph \define{from} $A$ \define{to} $B$.
\end{defn}

For example, here is an open $L$-labeled graph with $L = \{+,-\}$:
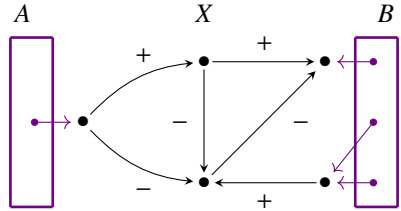
\[
\begin{tikzpicture}[scale=0.8]
	\begin{pgfonlayer}{nodelayer}
		\node [contact] (n1) at (-2,0) {$\bullet$};
		\node [style = none] at (-2.1,0.3) {$$};
		\node [contact] (n2) at (0,1) {$\bullet$};
		\node [style = none] at (0,1.3) {$$};
		\node [contact] (n3) at (0,-1) {$\bullet$};
		\node [style = none] at (0,-1.3) {$$};
		\node [contact] (n4) at (2,1) {$\bullet$};
		\node [style = none] at (2.1,0.3) {$$};
		\node [contact] (n5) at (2,-1) {$\bullet$};
		\node [style = none] at (2.1,0.3) {$$};

		\node [style = none] at (-1,1.1) {$+$};
		\node [style = none] at (-1,-1.1) {$-$};
		\node [style = none] at (1,1.3) {$+$};
		\node [style = none] at (1,-1.3) {$+$};
	        \node [style = none] at (-0.4,0) {$-$};
	        \node [style = none] at (1.6,0) {$-$};

		\node [style=none] (1) at (-3,0) {};
		\node [style=none] (4) at (3,0) {};

		\node [style=none] (ATL) at (-3.2,1.4) {};
		\node [style=none] (ATR) at (-2.5,1.4) {};
		\node [style=none] (ABR) at (-2.5,-1.4) {};
		\node [style=none] (ABL) at (-3.2,-1.4) {};

		\node [style=none] (X) at (-3,1.8) {$A$};
		\node [style=inputdot] (inI) at (-2.8,0) {};

		\node [style=none] (Z) at (3,1.8) {$B$};
	 	\node [style=inputdot] (outI') at (2.8,1) {};
		 \node [style=inputdot] (outI'') at (2.8,0) {};
	 	\node [style=inputdot] (outI''') at (2.8,-1) {};

		\node [style=none] (MTL) at (2.5,1.4) {};
		\node [style=none] (MTR) at (3.2,1.4) {};
		\node [style=none] (MBR) at (3.2,-1.4) {};
		\node [style=none] (MBL) at (2.5,-1.4) {};
		
		\node[style=none] at (0,1.8) {$X$};

	\end{pgfonlayer}
	\begin{pgfonlayer}{edgelayer}
		\draw [style=inarrow, bend left=20, looseness=1.00] (n1) to (n2);
		\draw [style=inarrow, bend right=20, looseness=1.00] (n1) to (n3);
		\draw [style=inarrow, bend left=0, looseness=1.00] (n2) to (n4);
		\draw [style=inarrow, bend right=0, looseness=1.00] (n3) to (n4);
		\draw [style=inarrow, bend right=0, looseness=1.00] (n5) to (n3);
		\draw [style=inarrow] (n2) to (n3);
		\draw [style=simple] (ATL.center) to (ATR.center);
		\draw [style=simple] (ATR.center) to (ABR.center);
		\draw [style=simple] (ABR.center) to (ABL.center);
		\draw [style=simple] (ABL.center) to (ATL.center);
		\draw [style=simple] (MTL.center) to (MTR.center);
		\draw [style=simple] (MTR.center) to (MBR.center);
		\draw [style=simple] (MBR.center) to (MBL.center);
		\draw [style=simple] (MBL.center) to (MTL.center);
		\draw [style=inputarrow] (inI) to (n1);
		\draw [style=inputarrow] (outI') to (n4);
		\draw [style=inputarrow] (outI'') to (n5);
		\draw [style=inputarrow] (outI''') to (n5);
	\end{pgfonlayer}
\end{tikzpicture}
\]
$A$ and $B$ are sets, which we can think of as graphs with no edges using the functor $\disc$.  The purple arrows show how $\disc(A)$ and $\disc(B)$ are mapped to the $L$-graph $X$ drawn in black.   

The main point of open $L$-labeled graphs is that we can compose them to form bigger ones.    For example, if we compose the above open $L$-labeled graph from $A$ to $B$ with this $L$-labeled graph from $B$ to $C$:
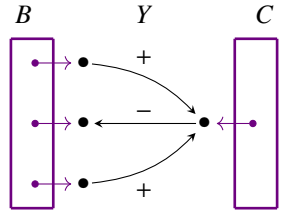
\[   \begin{tikzpicture}[scale=0.8]
	\begin{pgfonlayer}{nodelayer}
		\node [contact] (n6) at (3.7,1) {$\bullet$};
		\node [style = none] at (3.6,0.3) {$$};
		\node [contact] (n7) at (3.7,0) {$\bullet$};
		\node [style = none] at (3.7,1.3) {$$};
		\node [contact] (n8) at (3.7,-1) {$\bullet$};
		\node [style = none] at (5.7,1.3) {$$};
		\node [contact] (n9) at (5.7,0) {$\bullet$};
 
		\node [style = none] at (4.7,1.1) {$+$};
		\node [style = none] at (4.7,0.2) {$-$};
		\node [style = none] at (4.7,-1.1) {$+$};
 
		\node [style=none] (1) at (1.7,0) {};
		\node [style=none] (4) at (7.7,0) {};
 
		\node [style=none] (ATL) at (2.5,1.4) {};
		\node [style=none] (ATR) at (3.2,1.4) {};
		\node [style=none] (ABR) at (3.2,-1.4) {};
		\node [style=none] (ABL) at (2.5,-1.4) {};
 
		\node [style=none] (X) at (2.7,1.8) {$B$};
	 	\node [style=inputdot] (inI') at (2.9,1) {};
		 \node [style=inputdot] (inI'') at (2.9,0) {};
	 	\node [style=inputdot] (inI''') at (2.9,-1) {};
 
		\node [style=none] (Z) at (6.7,1.8) {$C$};
 
		\node [style=none] (MTL) at (6.2,1.4) {};
		\node [style=none] (MTR) at (6.9,1.4) {};
		\node [style=none] (MBR) at (6.9,-1.4) {};
		\node [style=none] (MBL) at (6.2,-1.4) {};
		
		\node [style=inputdot] (outI) at (6.5,0) {};
		
		\node[style=none] at (4.7,1.8) {$Y$};
 
	\end{pgfonlayer}
	\begin{pgfonlayer}{edgelayer}
		\draw [style=inarrow, bend left=20, looseness=1.00] (n6) to (n9);
		\draw [style=inarrow] (n9) to (n7);
		\draw [style=inarrow, bend right=20, looseness=1.00] (n8) to (n9);
 
		\draw [style=simple] (ATL.center) to (ATR.center);
		\draw [style=simple] (ATR.center) to (ABR.center);
		\draw [style=simple] (ABR.center) to (ABL.center);
		\draw [style=simple] (ABL.center) to (ATL.center);
		\draw [style=simple] (MTL.center) to (MTR.center);
		\draw [style=simple] (MTR.center) to (MBR.center);
		\draw [style=simple] (MBR.center) to (MBL.center);
		\draw [style=simple] (MBL.center) to (MTL.center);
 
		\draw [style=inputarrow] (inI') to (n6);
		\draw [style=inputarrow] (inI'') to (n7);
		\draw [style=inputarrow] (inI''') to (n8);
		\draw [style=inputarrow] (outI) to (n9);
	\end{pgfonlayer}
\end{tikzpicture}
\]
we get this open $L$-labeled graph from $A$ to $C$:
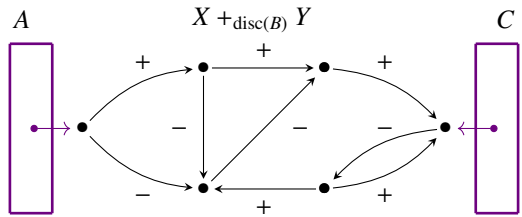
\[
\begin{tikzpicture}[scale=0.8]
	\begin{pgfonlayer}{nodelayer}
		\node [contact] (n1) at (-2,0) {$\bullet$};
		\node [style = none] at (-2.1,0.3) {$$};
		\node [contact] (n2) at (0,1) {$\bullet$};
		\node [style = none] at (0,1.3) {$$};
		\node [contact] (n3) at (0,-1) {$\bullet$};
		\node [style = none] at (0,-1.3) {$$};
		\node [contact] (n4) at (2,1) {$\bullet$};
		\node [style = none] at (2.1,0.3) {$$};
		\node [contact] (n5) at (2,-1) {$\bullet$};
		\node [style = none] at (2.1,0.3) {$$};

		\node [style = none] at (-1,1.1) {$+$};
		\node [style = none] at (-1,-1.1) {$-$};
		\node [style = none] at (1,1.3) {$+$};
		\node [style = none] at (1,-1.3) {$+$};
	        \node [style = none] at (-0.4,0) {$-$};
	        \node [style = none] at (1.6,0) {$-$};

		\node [style=none] (1) at (-3,0) {};
		\node [style=none] (4) at (3,0) {};

		\node [style=none] (ATL) at (-3.2,1.4) {};
		\node [style=none] (ATR) at (-2.5,1.4) {};
		\node [style=none] (ABR) at (-2.5,-1.4) {};
		\node [style=none] (ABL) at (-3.2,-1.4) {};

		\node [style=none]  at (-3,1.8) {$A$};
		\node [style=inputdot] (inI) at (-2.8,0) {};
		
		\node[style=none] at (0.8,1.8) {$X+_{\disc(B)} Y$};

	\end{pgfonlayer}
	\begin{pgfonlayer}{edgelayer}
		\draw [style=inarrow, bend left=20, looseness=1.00] (n1) to (n2);
		\draw [style=inarrow, bend right=20, looseness=1.00] (n1) to (n3);
		\draw [style=inarrow, bend left=0, looseness=1.00] (n2) to (n4);
		\draw [style=inarrow, bend right=0, looseness=1.00] (n3) to (n4);
		\draw [style=inarrow, bend right=0, looseness=1.00] (n5) to (n3);
		\draw [style=inarrow] (n2) to (n3);
		\draw [style=simple] (ATL.center) to (ATR.center);
		\draw [style=simple] (ATR.center) to (ABR.center);
		\draw [style=simple] (ABR.center) to (ABL.center);
		\draw [style=simple] (ABL.center) to (ATL.center);

		\draw [style=inputarrow] (inI) to (n1);
	\end{pgfonlayer}
	
		\begin{pgfonlayer}{nodelayer}
		\node [contact] (n6) at (2,1) {$\bullet$};
		\node [style = none] at (2.1,0.3) {$$};
		\node [contact] (n8) at (2,-1) {$\bullet$};
		\node [style = none] at (4,1.3) {$$};
		\node [contact] (n9) at (4,0) {$\bullet$};
 
		\node [style = none] at (3,1.1) {$+$};
		\node [style = none] at (3,0) {$-$};
		\node [style = none] at (3,-1.1) {$+$};

		\node [style=none] (Z) at (5,1.8) {$C$};
 
		\node [style=none] (MTL) at (4.5,1.4) {};
		\node [style=none] (MTR) at (5.2,1.4) {};
		\node [style=none] (MBR) at (5.2,-1.4) {};
		\node [style=none] (MBL) at (4.5,-1.4) {};
		
		\node [style=inputdot] (outI) at (4.8,0) {};

	\end{pgfonlayer}
	\begin{pgfonlayer}{edgelayer}
		\draw [style=inarrow, bend left=20, looseness=1.00] (n6) to (n9);
		\draw [style=inarrow, bend right=20, looseness=1.00] (n9) to (n8);
		\draw [style=inarrow, bend right=20, looseness=1.00] (n8) to (n9);
 
		\draw [style=simple] (MTL.center) to (MTR.center);
		\draw [style=simple] (MTR.center) to (MBR.center);
		\draw [style=simple] (MBR.center) to (MBL.center);
		\draw [style=simple] (MBL.center) to (MTL.center);

		\draw [style=inputarrow] (outI) to (n9);
	\end{pgfonlayer}
\end{tikzpicture}
\]
Besides being able to compose open $L$-labeled graphs we can also tensor them, i.e., put them in parallel, using coproducts.   The upshot is that open $L$-labeled graphs form a symmetric monoidal double category.   This is captured in the following already known result:

\begin{thm}
\label{Thm:Open_set-labeled_graphs}
For any set $L$, there is a symmetric monoidal double category of open $L$-labeled graphs, $\lOpen(\Gph/G_L)$, in which
\begin{itemize}
\item an object is a set,
\item a vertical 1-morphism from $A$ to $B$ is a function $f \maps A \to B$,
\item a horizontal 1-cell from $A$ to $B$ is an open $L$-labeled graph from $A$ to $B$:
\begin{displaymath}
\begin{tikzpicture}[scale=1.5]
\node (A) at (0,0) {$\disc(A)$};
\node (B) at (1,0) {$X$};
\node (C) at (2,0) {$\disc(B)$,};
\path[->,font=\scriptsize,>=angle 90]
(A) edge node[above]{$i$} (B)
(C) edge node[above]{$o$} (B);
\end{tikzpicture}
\end{displaymath}
\item a 2-morphism is a \define{map of open $L$-labeled graphs}, that is, a commutative diagram in $\Gph/G_L$ of the form
\begin{displaymath}
\begin{tikzpicture}[scale=1.5]
\node (A) at (-0.5,0) {$\disc(A)$};
\node (B) at (1,0) {$X$};
\node (C) at (2.5,0) {$\disc(B)$};
\node (A') at (-0.5,-1) {$\disc(A')$};
\node (B') at (1,-1) {$X'$};
\node (C') at (2.5,-1) {$\disc(B')$.};
\path[->,font=\scriptsize,>=angle 90]
(A) edge node[above]{$i$} (B)
(C) edge node[above]{$o$} (B)
(A') edge node[below]{$i'$} (B')
(C') edge node[below]{$o'$} (B')
(A) edge node [left]{$\disc(f)$} (A')
(B) edge node [left]{$\alpha$} (B')
(C) edge node [right]{$\disc(g)$} (C');
\end{tikzpicture}
\end{displaymath}
\end{itemize}
Vertical composition is done using composition in $\Set$, while horizontal composition
is done using pushouts in $\Gph/G_L$. The tensor product of two open $L$-labeled graphs is
\[
\begin{tikzcd}[column sep={.4in,between origins}, row sep=.08in]
& X &&&& X' &&&& X+X' & \\
&&&\otimes&&&& = &&&& \\
\disc(A)\ar[uur,"i" pos=0.3]&& \disc(B)\ar[uul,"o"' pos=0.3] && \disc(A')\ar[uur,"i'" pos=0.3] && \disc(B')\ar[uul,"o'"' pos=0.3] && \disc(A+A')\ar[uur,"i+i'" pos=0.3] && \disc(B+B')\ar[uul,"o+o'"' pos=0.3]
\end{tikzcd}
\]
where $i + i'$ and $o + o'$ are defined using the fact that $\disc$ preserves binary coproducts, and the tensor product of two 2-morphisms is given by:
\[
\begin{tikzpicture}[scale=1.5]
\node (E) at (2,0) {$\disc(A_1)$};
\node (F) at (4,0) {$\disc(B_1)$};
\node (G) at (3,0) {$X_1$};
\node (E') at (2,-1) {$\disc(A_2)$};
\node (F') at (4,-1) {$\disc(B_2)$};
\node (G') at (3,-1) {$X_2$};
\node (E'') at (6.5,0) {$\disc(A_1')$};
\node (F'') at (8.5,0) {$\disc(B_1')$};
\node (G'') at (7.5,0) {$X_1'$};
\node (E''') at (6.5,-1) {$\disc(A_2')$};
\node (F''') at (8.5,-1) {$\disc(B_2')$};
\node (G''') at (7.5,-1) {$X_2'$};
\node (X) at (5.25,-0.5) {$\otimes$};
\node (E'''') at (3.5,-2) {$\disc(A_1 + A_1')$};
\node (F'''') at (7,-2) {$\disc(B_1 + B_1')$};
\node (G'''') at (5.25,-2) {$X_1 + X_1'$};
\node (E''''') at (3.5,-3) {$\disc(A_2 + A_2')$};
\node (F''''') at (7,-3) {$\disc(B_2 + B_2').$};
\node (G''''') at (5.25,-3) {$X_2 + X_2'$};
\node (Y) at (2,-2.5) {$=$};
\path[->,font=\scriptsize,>=angle 90]
(F) edge node[above]{$o_1$} (G)
(E) edge node[left]{$\disc(f)$} (E')
(F) edge node[right]{$\disc(g)$} (F')
(G) edge node[left]{$\alpha$} (G')
(E) edge node[above]{$i_1$} (G)
(E') edge node[below]{$i_2$} (G')
(F') edge node[below]{$o_2$} (G')
(F'') edge node[above]{$o_1'$} (G'')
(E'') edge node[left]{$\disc(f')$} (E''')
(F'') edge node[right]{$\disc(g')$} (F''')
(G'') edge node[left]{$\alpha'$} (G''')
(E'') edge node[above]{$i_1'$} (G'')
(E''') edge node[below]{$i_2'$} (G''')
(F''') edge node[below]{$o_2'$} (G''')
(F'''') edge node[above]{$o_1 + o_1'$} (G'''')
(E'''') edge node[left]{$\disc(f + f')$} (E''''')
(F'''') edge node[right]{$\disc(g + g')$} (F''''')
(G'''') edge node[left]{$\alpha + \alpha'$} (G''''')
(E'''') edge node[above]{$i_1 + i_1'$} (G'''')
(E''''') edge node[below]{$i_2 + i_2'$} (G''''')
(F''''') edge node[below]{$o_2 + o_2'$} (G''''');
\end{tikzpicture}
\]
\end{thm}

\begin{proof}
The categories $\Set$ and $\Gph/G_L$ have finite colimits (see \cref{Prop:Set-labeled_graph}(a)) and $\disc \maps \Set \to \Gph/G_L$ preserves them, since it is a left adjoint.   The theorem thus follows from the theory of structured cospans: see   \cite[Sec. 6.1]{BaezCourser2020} and \cite[Sec. 6.2]{BaezCourserVasilakopoulou2022} for two different proofs.   Using the same assumptions,  Theorem 2.3 of \cite{Patterson2023} yields a stronger result: $\lOpen(\Gph/ G_L)$ is a cocartesian equipment, and thus a cocartesian double category.
\end{proof}

If $M$ is a monoid we can define open $M$-labeled graphs as in \cref{Defn:Open_set-labeled_graphs}, not using the monoid structure, only the underlying set of $M$.  However, we have the extra ability to convert any open $M$-labeled graph into an `open $M$-graded category', which we now define.   First, note that the categories of sets and $M$-labeled categories are related by a pair of adjoint functors
\[
\begin{tikzcd}
            \ \phantom{X} \Set\phantom{X} \arrow[r, bend left = 20, shift left=1ex, "{\Disc}"{name=G}] & \ \cat/\B M \arrow[l, bend left = 20, shift left=.5ex, pos = 0.4, "\mathrm{Vert}"{name=F}]
            \arrow[phantom, from=F, to=G, "\scriptstyle\boldsymbol{\bot}"].
        \end{tikzcd}
    \]
where
\begin{itemize}
\item $\Disc \maps \Set \to \cat/\B M$ takes a set $S$ to the unique $M$-graded category with object set $S$ and only identity morphisms;
\item $\mathrm{Vert} \maps \cat/\B M \to \Set$ takes an $M$-graded category to its set of objects.
\end{itemize}
We can apply the theory of structured cospans to the left adjoint $\Disc$, as follows. 

\begin{defn}
\label{Defn:Open_monoid-labeled_category}
Given a monoid $M$, an \define{open} $M$\define{-graded category} is a diagram in $\cat/\B M$ of the form
\[
\begin{tikzpicture}[scale=1.2]
\node (A) at (0,0) {$\Disc(A)$};
\node (B) at (1,1) {$X$};
\node (C) at (2,0) {$\Disc(B)$};
\path[->,font=\scriptsize,>=angle 90]
(A) edge node[above, pos=0.3]{$i$} (B)
(C)edge node[above, pos=0.3]{$o$}(B);
\end{tikzpicture}
\]
for some sets $A$ and $B$.   We call this an open $M$-graded category \define{from} $A$ \define{to} $B$.
\end{defn}

\begin{thm}
\label{Thm:Open_monoid-labeled_categories}
For any monoid $M$, there is a symmetric monoidal double category of open $M$-labeled categories,
$\lOpen(\cat/\B M)$, in which
\begin{itemize}
\item an object is a set,
\item a vertical 1-morphism from $A$ to $B$ is a function $f \maps A \to B$,
\item a horizontal 1-cell from $A$ to $B$ is an open $M$-graded category  from $A$ to $B$:
\begin{displaymath}
\begin{tikzpicture}[scale=1.5]
\node (A) at (0,0) {$\Disc(A)$};
\node (B) at (1,0) {$X$};
\node (C) at (2,0) {$\Disc(B)$,};
\path[->,font=\scriptsize,>=angle 90]
(A) edge node[above]{$i$} (B)
(C) edge node[above]{$o$} (B);
\end{tikzpicture}
\end{displaymath}
\item a 2-morphism is a \define{map of open} $M$\define{-graded categories}, that is, a commutative diagram in $\cat/\B M$ of the form
\begin{displaymath}
\begin{tikzpicture}[scale=1.5]
\node (A) at (-0.5,0) {$\Disc(A)$};
\node (B) at (1,0) {$X$};
\node (C) at (2.5,0) {$\Disc(B)$};
\node (A') at (-0.5,-1) {$\Disc(A')$};
\node (B') at (1,-1) {$X'$};
\node (C') at (2.5,-1) {$\Disc(B')$.};
\path[->,font=\scriptsize,>=angle 90]
(A) edge node[above]{$i$} (B)
(C) edge node[above]{$o$} (B)
(A') edge node[below]{$i'$} (B')
(C') edge node[below]{$o'$} (B')
(A) edge node [left]{$\Disc(f)$} (A')
(B) edge node [left]{$\alpha$} (B')
(C) edge node [right]{$\Disc(g)$} (C');
\end{tikzpicture}
\end{displaymath}
\end{itemize}
Vertical composition is done using composition in $\Set$, while horizontal composition
is done using pushouts in $\cat/ \B M$. The tensor product of two horizontal 1-cells is
\[
\begin{tikzcd}[column sep={.4in,between origins}, row sep=.08in]
& X &&&& X' &&&& X+X' & \\
&&&\otimes&&&& = &&&& \\
\Disc(A)\ar[uur,"i" pos=0.3]&& \Disc(B)\ar[uul,"o"' pos=0.3] && \Disc(A')\ar[uur,"i'" pos=0.3] && \Disc(B')\ar[uul,"o'"' pos=0.3] && \Disc(A+A')\ar[uur,"i+i'" pos=0.3] && \Disc(B+B')\ar[uul,"o+o'"' pos=0.3]
\end{tikzcd}
\]
where $i + i'$ and $o + o'$ are defined using the fact that $\Disc$ preserves binary coproducts, and the tensor product of two 2-morphisms is given by:
\[
\begin{tikzpicture}[scale=1.5]
\node (E) at (2,0) {$\Disc(A_1)$};
\node (F) at (4,0) {$\Disc(B_1)$};
\node (G) at (3,0) {$X_1$};
\node (E') at (2,-1) {$\Disc(A_2)$};
\node (F') at (4,-1) {$\Disc(B_2)$};
\node (G') at (3,-1) {$X_2$};
\node (E'') at (6.5,0) {$\Disc(A_1')$};
\node (F'') at (8.5,0) {$\Disc(B_1')$};
\node (G'') at (7.5,0) {$X_1'$};
\node (E''') at (6.5,-1) {$\Disc(A_2')$};
\node (F''') at (8.5,-1) {$\Disc(B_2')$};
\node (G''') at (7.5,-1) {$X_2'$};
\node (X) at (5.25,-0.5) {$\otimes$};
\node (E'''') at (3.5,-2) {$\Disc(A_1 + A_1')$};
\node (F'''') at (7,-2) {$\Disc(B_1 + B_1')$};
\node (G'''') at (5.25,-2) {$X_1 + X_1'$};
\node (E''''') at (3.5,-3) {$\Disc(A_2 + A_2')$};
\node (F''''') at (7,-3) {$\Disc(B_2 + B_2').$};
\node (G''''') at (5.25,-3) {$X_2 + X_2'$};
\node (Y) at (2,-2.5) {$=$};
\path[->,font=\scriptsize,>=angle 90]
(F) edge node[above]{$o_1$} (G)
(E) edge node[left]{$\Disc(f)$} (E')
(F) edge node[right]{$\Disc(g)$} (F')
(G) edge node[left]{$\alpha$} (G')
(E) edge node[above]{$i_1$} (G)
(E') edge node[below]{$i_2$} (G')
(F') edge node[below]{$o_2$} (G')
(F'') edge node[above]{$o_1'$} (G'')
(E'') edge node[left]{$\Disc(f')$} (E''')
(F'') edge node[right]{$\Disc(g')$} (F''')
(G'') edge node[left]{$\alpha'$} (G''')
(E'') edge node[above]{$i_1'$} (G'')
(E''') edge node[below]{$i_2'$} (G''')
(F''') edge node[below]{$o_2'$} (G''')
(F'''') edge node[above]{$o_1 + o_1'$} (G'''')
(E'''') edge node[left]{$\Disc(f + f')$} (E''''')
(F'''') edge node[right]{$\Disc(g + g')$} (F''''')
(G'''') edge node[left]{$\alpha + \alpha'$} (G''''')
(E'''') edge node[above]{$i_1 + i_1'$} (G'''')
(E''''') edge node[below]{$i_2 + i_2'$} (G''''')
(F''''') edge node[below]{$o_2 + o_2'$} (G''''');
\end{tikzpicture}
\]
\end{thm}

\begin{proof}
The categories $\Set$ and $\Gph/\B M$ have finite colimits (see \cref{Prop:Monoid-labeled_category}(a)) and $\Disc \maps \Set \to \Gph/\B M$ preserves them, since it is a left adjoint.  The theory of structured cospans thus implies that $\lOpen(\Gph/G_L)$ is a symmetric monoidal double category as above.  In fact \cite[Thm.\ 2.3]{Patterson2023} implies that $\lOpen(\Gph/ G_L)$ is a cocartesian equipment.
\end{proof}

Now we construct a map that turns open $M$-labeled graphs into open $M$-graded categories.
To do this we use \cref{Lem:Free_monoid-labeled_category}, which says that there is an adjunction
\[
\begin{tikzcd}
            \ \Gph/G_M\arrow[r, bend left = 20, shift left=1ex, "\Free_M"{name=G}] & \ \cat/\B M \arrow[l, bend left = 20, shift left=.5ex, "\Und_M"{name=F}]
            \arrow[phantom, from=F, to=G, "\scriptstyle\boldsymbol{\bot}"].
        \end{tikzcd}
\]

\begin{thm}
\label{Thm:Converting_open_graphs_to_open_categories}
For any monoid $M$ there is a symmetric monoidal double functor 
\[     F_M \maps \lOpen(\Gph/G_M) \to \lOpen(\cat/\B M) \]
acting as follows:
\begin{itemize}
\item on objects, $F_M$ sends any set to itself,
\item on vertical 1-morphisms, $F_M$ sends any function to itself
\item on horizontal 1-cells, $F_M$ sends any open $M$-labeled graph
\begin{displaymath}
\begin{tikzpicture}[scale=1.5]
\node (A) at (0,0) {$\disc(A)$};
\node (B) at (1,0) {$X$};
\node (C) at (2,0) {$\disc(B)$,};
\path[->,font=\scriptsize,>=angle 90]
(A) edge node[above]{$i$} (B)
(C) edge node[above]{$o$} (B);
\end{tikzpicture}
\end{displaymath}
to the open $M$-graded category
\begin{displaymath}
\begin{tikzpicture}[scale=1.5]
\node (A) at (0,0) {$\Disc(A)$};
\node (B) at (1.5,0) {$\Free_M(X)$};
\node (C) at (3,0) {$\Disc(B)$,};
\path[->,font=\scriptsize,>=angle 90]
(A) edge node[above]{$i$} (B)
(C) edge node[above]{$o$} (B);
\end{tikzpicture}
\end{displaymath}
\item on 2-morphisms, $F_M$ sends any map of open $M$-labeled graphs
\begin{displaymath}
\begin{tikzpicture}[scale=1.5]
\node (A) at (-0.5,0) {$\disc(A)$};
\node (B) at (1,0) {$X$};
\node (C) at (2.5,0) {$\disc(B)$};
\node (A') at (-0.5,-1) {$\disc(A')$};
\node (B') at (1,-1) {$X'$};
\node (C') at (2.5,-1) {$\disc(B')$.};
\path[->,font=\scriptsize,>=angle 90]
(A) edge node[above]{$i$} (B)
(C) edge node[above]{$o$} (B)
(A') edge node[below]{$i'$} (B')
(C') edge node[below]{$o'$} (B')
(A) edge node [left]{$\disc(f)$} (A')
(B) edge node [left]{$\alpha$} (B')
(C) edge node [right]{$\disc(g)$} (C');
\end{tikzpicture}
\end{displaymath}
to the map of open $M$-graded categories
\begin{displaymath}
\begin{tikzpicture}[scale=1.5]
\node (A) at (-0.5,0) {$\Disc(A)$};
\node (B) at (1,0) {$\Free_M(X)$};
\node (C) at (2.5,0) {$\Disc(B)$};
\node (A') at (-0.5,-1) {$\Disc(A')$};
\node (B') at (1,-1) {$\Free_M(X')$};
\node (C') at (2.5,-1) {$\Disc(B')$.};
\path[->,font=\scriptsize,>=angle 90]
(A) edge node[above]{$i$} (B)
(C) edge node[above]{$o$} (B)
(A') edge node[below]{$i'$} (B')
(C') edge node[below]{$o'$} (B')
(A) edge node [left]{$\Disc(f)$} (A')
(B) edge node [left]{$\Free_M(\alpha)$} (B')
(C) edge node [right]{$\Disc(g)$} (C');
\end{tikzpicture}
\end{displaymath}
\end{itemize}
\end{thm}
        
\begin{proof} 
Theorems 4.2 and 4.3 of \cite{BaezCourser2020} give a general way to construct maps between structured cospan double categories.   The present theorem is a straightforward application of these, which only requires checking that the following diagram commutes up to natural isomorphism
\begin{displaymath}
\begin{tikzpicture}[scale=1]
\node (A) at (0,0) {$\Set$};
\node (B) at (2,0.8) {$\Gph/G_M$};
\node (C) at (2,-0.8) {$\cat/\B M$};
\path[->,font=\scriptsize,>=angle 90]
(A) edge node[above]{$\disc$} (B)
(A) edge node[below]{$\Disc$} (C)
(B) edge node [right]{$\Free_M$} (C);
\end{tikzpicture}
\end{displaymath}
 and that all the arrows in this diagram preserve finite colimits (since they are left adjoints).  
 \end{proof}
 
Recall from \cref{Defn:Kleisli_category} that for any monoid $M$ there is a category $\K(\Gph/G_M)$ of $M$-labeled graphs and Kleisli morphisms between these, where a Kleisli morphism from $(G,\ell)$ to $(G',\ell')$ is defined to be a map between free $M$-graded categories
 \[   f \maps \Free_M(G,\ell) \to \Free_M(G',\ell')  .\]
We can use the theory of structured cospans to define a double category of \emph{open} $M$-labeled graphs and Kleisli morphisms.  To do this, we use the functor
\[    \Phi \maps \Set \to K(\Gph/G_M) \]
defined as follows:
\begin{itemize}
\item $\Phi$ sends any set $S$ to the unique $M$-graded category with object set $S$ and only identity morphisms.  Note that this $M$-graded category is $\Free_M(G,\ell)$ where $G$ is the graph with vertex set $S$ and no edges, and $\ell$ is the only possible $\ell$-labeling of $G$. 
\item $\Phi$ sends any function $f \maps S \to T$ to the unique map of $M$-graded categories acting as $f$ on objects.
\end{itemize}
 
 \begin{thm}
\label{Thm:Open_monoid-labeled_graphs}
For any monoid $M$, there is a symmetric monoidal double category of open $M$-labeled graphs and Kleisli
morphisms, $\lOpen(K(\Gph/G_M))$, in which
\begin{itemize}
\item an object is a set,
\item a vertical 1-morphism from $A$ to $B$ is a function $f \maps A \to B$,
\item a horizontal 1-cell from $A$ to $B$ is an open $M$-labeled graph from $A$ to $B$:
\begin{displaymath}
\begin{tikzpicture}[scale=1.5]
\node (A) at (0,0) {$\Phi(A)$};
\node (B) at (1,0) {$X$};
\node (C) at (2,0) {$\Phi(B)$,};
\path[->,font=\scriptsize,>=angle 90]
(A) edge node[above]{$i$} (B)
(C) edge node[above]{$o$} (B);
\end{tikzpicture}
\end{displaymath}
\item a 2-morphism is a \define{Kleisli morphism of open $M$-labeled graphs}, that is, a commutative diagram in $\K(\Gph/G_M)$ of the form
\begin{displaymath}
\begin{tikzpicture}[scale=1.5]
\node (A) at (-0.5,0) {$\Phi(A)$};
\node (B) at (1,0) {$X$};
\node (C) at (2.5,0) {$\Phi(B)$};
\node (A') at (-0.5,-1) {$\Phi(A')$};
\node (B') at (1,-1) {$X'$};
\node (C') at (2.5,-1) {$\Phi(B')$.};
\path[->,font=\scriptsize,>=angle 90]
(A) edge node[above]{$i$} (B)
(C) edge node[above]{$o$} (B)
(A') edge node[below]{$i'$} (B')
(C') edge node[below]{$o'$} (B')
(A) edge node [left]{$\Phi(f)$} (A')
(B) edge node [left]{$\alpha$} (B')
(C) edge node [right]{$\Phi(g)$} (C');
\end{tikzpicture}
\end{displaymath}
\end{itemize}
Vertical composition is done using composition in $\Set$, while horizontal composition
is done using pushouts in $\K(\Gph/G_M)$. The tensor product of two open $M$-labeled graphs is
\[
\begin{tikzcd}[column sep={.4in,between origins}, row sep=.08in]
& X &&&& X' &&&& X+X' & \\
&&&\otimes&&&& = &&&& \\
\Phi(A)\ar[uur,"i" pos=0.3]&& \Phi(B)\ar[uul,"o"' pos=0.3] && \Phi(A')\ar[uur,"i'" pos=0.3] && \Phi(B')\ar[uul,"o'"' pos=0.3] && \Phi(A+A')\ar[uur,"i+i'" pos=0.3] && \Phi(B+B')\ar[uul,"o+o'"' pos=0.3]
\end{tikzcd}
\]
where $i + i'$ and $o + o'$ are defined using the fact that $\Phi$ preserves binary coproducts, and the tensor product of two 2-morphisms is given by:
\[
\begin{tikzpicture}[scale=1.5]
\node (E) at (2,0) {$\Phi(A_1)$};
\node (F) at (4,0) {$\Phi(B_1)$};
\node (G) at (3,0) {$X_1$};
\node (E') at (2,-1) {$\Phi(A_2)$};
\node (F') at (4,-1) {$\Phi(B_2)$};
\node (G') at (3,-1) {$X_2$};
\node (E'') at (6.5,0) {$\Phi(A_1')$};
\node (F'') at (8.5,0) {$\Phi(B_1')$};
\node (G'') at (7.5,0) {$X_1'$};
\node (E''') at (6.5,-1) {$\Phi(A_2')$};
\node (F''') at (8.5,-1) {$\Phi(B_2')$};
\node (G''') at (7.5,-1) {$X_2'$};
\node (X) at (5.25,-0.5) {$\otimes$};
\node (E'''') at (3.5,-2) {$\Phi(A_1 + A_1')$};
\node (F'''') at (7,-2) {$\Phi(B_1 + B_1')$};
\node (G'''') at (5.25,-2) {$X_1 + X_1'$};
\node (E''''') at (3.5,-3) {$\Phi(A_2 + A_2')$};
\node (F''''') at (7,-3) {$\Phi(B_2 + B_2').$};
\node (G''''') at (5.25,-3) {$X_2 + X_2'$};
\node (Y) at (2,-2.5) {$=$};
\path[->,font=\scriptsize,>=angle 90]
(F) edge node[above]{$o_1$} (G)
(E) edge node[left]{$\Phi(f)$} (E')
(F) edge node[right]{$\Phi(g)$} (F')
(G) edge node[left]{$\alpha$} (G')
(E) edge node[above]{$i_1$} (G)
(E') edge node[below]{$i_2$} (G')
(F') edge node[below]{$o_2$} (G')
(F'') edge node[above]{$o_1'$} (G'')
(E'') edge node[left]{$\Phi(f')$} (E''')
(F'') edge node[right]{$\Phi(g')$} (F''')
(G'') edge node[left]{$\alpha'$} (G''')
(E'') edge node[above]{$i_1'$} (G'')
(E''') edge node[below]{$i_2'$} (G''')
(F''') edge node[below]{$o_2'$} (G''')
(F'''') edge node[above]{$o_1 + o_1'$} (G'''')
(E'''') edge node[left]{$\Phi(f + f')$} (E''''')
(F'''') edge node[right]{$\Phi(g + g')$} (F''''')
(G'''') edge node[left]{$\alpha + \alpha'$} (G''''')
(E'''') edge node[above]{$i_1 + i_1'$} (G'''')
(E''''') edge node[below]{$i_2 + i_2'$} (G''''')
(F''''') edge node[below]{$o_2 + o_2'$} (G''''');
\end{tikzpicture}
\]
\end{thm}

\begin{proof}
As mentioned in the proof of \cref{Thm:Open_set-labeled_graphs}, a commonly stated theorem for constructing a symmetric monoidal double category of $F$-structured cospans assumes that we have categories $\A$ and $\X$ with finite colimits and a functor $F \maps \A \to \X$ that preserves them.   We would like to apply this theorem to the functor $\Phi \maps  \Set \to \K(\Gph/G_M)$.  Unfortunately, while the category $\Set$ has finite colimits, $\Phi$ preserves them, and the Kleisli category $\K(\Gph/G_M)$ has finite coproducts, this Kleisli category does not have pushouts.  Luckily, if one examines the proof of the commonly stated theorem, it becomes clear that we only need pushouts at one point: to compose structured cospans.  In the case at hand, this means taking pushouts of this sort:
\[
\begin{tikzpicture}[scale=1]
\node (A) at (0,0) {$\Phi(A)$};
\node (B) at (1,1) {$X$};
\node (C) at (2,0) {$\Phi(B)$};
\node (D) at (3,1) {$Y$};
\node (E) at (4,0) {$\Phi(C)$};
\node (F) at (2,2) {$X +_{\Phi(B)} Y$};
\node (push) at (2,1.4) {\rotatebox{135}{$\ulcorner$}};
\path[->,font=\scriptsize,>=angle 90]
(A) edge node[above, pos=0.3]{$i$} (B)
(C) edge node[above, pos=0.3]{$o$}(B)
(C) edge node[above, pos=0.3]{$i'$} (D)
(E) edge node[above, pos=0.3]{$o'$}(D)
(B) edge node[above, pos=0.3]{$$} (F)
(D) edge node[above, pos=0.3]{$$}(F);
\end{tikzpicture}
\]
where $X$ and $Y$ are free $M$-graded categories.   This pushout exists in $\cat/\B M$, and the result is a free $M$-graded category, because in forming the pushout we do not need to identify any morphisms, only objects.  It follows that this pushout also exists in $\K(\Gph/G_M)$.   
\end{proof}

Finally, given a commutative monoid $C$, we can define a symmetric monoidal double category of open $C$-labeled finite graphs and additive morphisms.  However,  we cannot do this using structured cospans and the left adjoint $\Fin\Set \to \C\Fin\Gph$ to the functor assigning any $C$-labeled finite graph its set of vertices, since $\C\Fin\Gph$ does not have finite coproducts or pushouts, not even the limited class of pushouts that would be used to compose structured cospans.    Here we instead use \emph{decorated} cospans, introduced by Fong \cite{Fong2015,FongThesis} and later brought up to the level of double categories \cite{BaezCourserVasilakopoulou2022}.

In this approach we view an open $C$-labeled finite graph 
\[
\begin{tikzpicture}[scale=1.2]
\node (A) at (0,0) {$\disc(A)$};
\node (B) at (1,1) {$X$};
\node (C) at (2,0) {$\disc(B)$};
\path[->,font=\scriptsize,>=angle 90]
(A) edge node[above, pos=0.3]{$i$} (B)
(C)edge node[above, pos=0.3]{$o$}(B);
\end{tikzpicture}
\]
in a new way.   We think of it as a cospan of finite sets
\[
\begin{tikzpicture}[scale=1.2]
\node (A) at (0,0) {$A$};
\node (B) at (1,1) {$V$};
\node (C) at (2,0) {$B$};
\path[->,font=\scriptsize,>=angle 90]
(A) edge node[above, pos=0.3]{$i$} (B)
(C)edge node[above, pos=0.3]{$o$}(B);
\end{tikzpicture}
\]
where $V$ is `decorated' with some extra stuff: namely, a $C$-labeled finite graph $X$ having $V$ as its set of vertices.   We think of this `decoration' $X$ as an object of the category $F(V)$, where $F \maps \Fin\Set \to \Cat$ maps each finite set to the category of $C$-labeled finite graphs with that set of vertices, and additive morphisms between these.

In general, the theory of decorated cospans gives a symmetric monoidal double category  $F\lCsp$ from a category $\A$ with finite colimits and a symmetric lax monoidal pseudofunctor $F \maps (\A,+) \to (\Cat,\times)$.     In this double category $F\lCsp$:
\begin{itemize}
\item an object is an object of $\A$;
\item a vertical 1-morphism is a morphism of $\A$;
\item a horizontal 1-cell is an $F$\define{-decorated cospan}, that is,
a diagram in $\A$ of the form
\[
\begin{tikzpicture}[scale=1]
\node (A) at (0,0) {$A$};
\node (B) at (1,1) {$M$};
\node (C) at (2,0) {$B$};
\path[->,font=\scriptsize,>=angle 90]
(A) edge node[above, pos = 0.25]{$i$} (B)
(C) edge node[above, pos = 0.25]{$o$} (B);
\end{tikzpicture}
\]
together with a \define{decoration} $d \in F(M)$;
\item a 2-morphism is a \define{map of} $F$\define{-decorated cospans}, that is,
a commutative diagram in $\A$ of the form
\[
\begin{tikzpicture}[scale=1.5,baseline=(current bounding box.center)]
\node (A) at (0,0.5) {$A$};
\node (A') at (0,-0.5) {$A'$};
\node (B) at (1,0.5) {$M$};
\node (C) at (2,0.5) {$B$};
\node (C') at (2,-0.5) {$B'$};
\node (D) at (1,-0.5) {$M'$};
\node (E) at (3,0.5) {$d \in F(M)$};
\node (F) at (3,-0.5) {$d' \in F(M')$};
\path[->,font=\scriptsize,>=angle 90]
(A) edge node[above]{$i$} (B)
(C) edge node[above]{$o$} (B)
(A) edge node[left]{$f$} (A')
(C) edge node[right]{$g$} (C')
(A') edge node[below] {$i'$} (D)
(C') edge node[below] {$o'$} (D)
(B) edge node [left] {$h$} (D);
\end{tikzpicture}
\]
together with a \define{decoration morphism} $\tau \maps F(h)(d) \to d'$ in $F(M')$.
\end{itemize}
A key piece of the symmetric lax monoidal pseudofunctor $F$ is its `laxator', which gives for each pair of objects $M,M' \in \A$ a functor
\[   \phi_{M,M'} \maps F(M) \times F(M') \to F(M+M')  .\] 
This is used to compose $F$-decorated cospans, as follows.    Given a composable pair of $F$-decorated cospans
\[
\begin{tikzpicture}[scale=1]
\node (A) at (0,0) {$A$};
\node (B) at (1,1) {$M$};
\node (C) at (2,0) {$B$};
\node (D) at (3,1) {$d \in F(M)$};
\node (A') at (5,0) {$B$};
\node (B') at (6,1) {$N$};
\node (C') at (7,0) {$C$};
\node (D') at (8,1) {$d' \in F(M')$};
\path[->,font=\scriptsize,>=angle 90]
(A) edge node[above, pos = 0.25]{$i$} (B)
(C) edge node[above, pos = 0.25]{$o$} (B)
(A') edge node[above, pos = 0.25]{$i'$} (B')
(C') edge node[above, pos = 0.25]{$o'$} (B');
\end{tikzpicture}
\]
we define their composite to be 
\[
\begin{tikzpicture}[scale=1]
\node (A) at (0,0) {$A$};
\node (B) at (1,1) {$M$};
\node (C) at (2,0) {$B$};
\node (B') at (3,1) {$M'$};
\node (C') at (4,0) {$C$.};
\node (D) at (2,2) {$M+_B M'$};
\node (E) at (7,2) { $F(j) \left( \phi_{M,M'} (d,d') \right) \in F(M +_B M') $};
\node (push) at (2,1.4) {\rotatebox{135}{$\ulcorner$}};
\path[->,font=\scriptsize,>=angle 90]
(A) edge node[above, pos = 0.25]{$i$} (B)
(C) edge node[above, pos = 0.25]{$o$} (B)
(C) edge node[above, pos = 0.25]{$i'$} (B')
(C') edge node[above, pos = 0.25]{$o'$} (B')
(B) edge node[above, pos = 0.25]{} (D)
(B') edge node[above, pos = 0.25]{} (D);
\end{tikzpicture}
\]
Here the cospans are composed using a pushout in $\A$, while the decoration is defined using
the laxator and $j \maps M + M' \to M +_B M'$, the canonical map from the coproduct to the pushout.  For details see \cite[Sec.\ 2]{BaezCourserVasilakopoulou2022}.

To define the double category of $C$-labeled finite graphs using this machinery we take $\A = \Fin\Set$.  We want a decoration of $V \in \Fin\Set$ to be a $C$-labeled finite graph with $V$ as its set of vertices.   Thus we let $F \maps (\Fin\Set, +) \to (\Cat, \times)$ assign to $V$ the category $F(V)$ where:
\begin{itemize}
\item An object is a $C$-labeled finite graph with vertex set $V$, say $(V, E, s, t, \ell) $.
\item A morphism is an additive morphism of $C$-labeled graphs that is the identity on vertices.
\end{itemize}
We let $F$ assign to each function $g \maps V \to V'$ the functor $F(g) \maps F(V) \to F(V')$ where:
\begin{itemize}
\item $F(g)$ sends any object $(V, E, s, t, \ell)$ to $(V', E, g \circ s, g \circ t, \ell)$.
\item $F(g)$ sends any morphism $f \maps (V, E, s, t, \ell) \to (V, E', s', t', \ell') $ in $F(V)$, which is determined by a map of graphs $1_V \maps V \to V$, $f_1 \maps E \to E'$, to the morphism $F(g)(f) \maps (V', E, g \circ s, g \circ t, \ell) \to (V', E', g \circ s', g \circ t', \ell')$ in $F(V')$ 
determined  by the map of graphs $1_{V'} \maps V' \to V'$ , $f_1 \maps E \to E'$.
\end{itemize}
To parse this, it is useful to review the definition of `additive morphism', \cref{Defn:Morphism-of-commutative-monoid-labeled-graphs}.   

Next we must equip $F$ with the structure of a lax monoidal pseudofunctor \cite{BaezCourserVasilakopoulou2022}.  We omit some details here.  The most important structure is the laxator
\[   \phi_{V,V'} \maps F(V) \times F(V') \to F(V + V')  .\]
On objects, this takes a $C$-labeled finite graph with vertex set $V$, say $(V,E,s,t,\ell)$, and a $C$-labeled finite graph with vertex set $V'$, say $(V',E',s',t',\ell')$, and gives the $C$-labeled finite graph 
\[        (V+V', E+E', s+s', t+t', \langle \ell, \ell' \rangle) \]
with vertex set $V + V'$.  Here $\langle \ell, \ell'\rangle \maps E + E' \to C$ is defined to equal $\ell$ on $E$ and $\ell'$ on $E'$.

Given that $F \maps (\Fin\Set, +) \to (\Cat, \times)$ is a symmetric lax monoidal pseudofunctor, the 
theory of decorated cospans gives this result:

 \begin{thm}
\label{Thm:Open_commutative_monoid-labeled_graphs}
For any commutative monoid $C$, there is a symmetric monoidal double category of open $C$-labeled finite graphs and additive morphisms, $\lOpen(C\Fin\Gph)$, in which
\begin{itemize}
\item an object is a finite set;
\item a vertical 1-morphism from $A$ to $B$ is a function $f \maps A \to B$;
\item a horizontal 1-cell from $A$ to $B$ is an open $C$-labeled finite graph from $A$ to $B$:
\begin{displaymath}
\begin{tikzpicture}[scale=1.5]
\node (A) at (0,0) {$A$};
\node (B) at (1,0) {$M$};
\node (C) at (2,0) {$B$};
\node (D) at (3,0) {$d \in F(M);$};
\path[->,font=\scriptsize,>=angle 90]
(A) edge node[above]{$i$} (B)
(C) edge node[above]{$o$} (B);
\end{tikzpicture}
\end{displaymath}
\item a 2-morphism is an \define{additive morphism of open} $C$\define{-labeled graphs}, that is, a commutative diagram in $\Fin\Set$ of the form
\begin{displaymath}
\begin{tikzpicture}[scale=1.5]
\node (A) at (-0.5,0) {$A$};
\node (B) at (1,0) {$M$};
\node (C) at (2.5,0) {$B$};
\node (D) at (4,0) {$d \in F(M)$};
\node (A') at (-0.5,-1) {$A'$};
\node (B') at (1,-1) {$M'$};
\node (C') at (2.5,-1) {$B'$};
\node (D') at (4,-1) {$d' \in F(M')$};
\path[->,font=\scriptsize,>=angle 90]
(A) edge node[above]{$i$} (B)
(C) edge node[above]{$o$} (B)
(A') edge node[below]{$i'$} (B')
(C') edge node[below]{$o'$} (B')
(A) edge node [left]{$f$} (A')
(B) edge node [left]{$\alpha$} (B')
(C) edge node [right]{$g$} (C');
\end{tikzpicture}
\end{displaymath}
\end{itemize}
together with a  `decoration morphism' $\tau \maps F(\alpha)(d) \to d'$.
Vertical composition is done using composition in $\Fin\Set$ and the appropriately defined
composition of decoration morphisms, while horizontal composition
is done using composition of $F$-decorated cospans. 
\end{thm}

\begin{proof}
The double category structure follows from \cite[Thm. 2.1]{BaezCourserVasilakopoulou2022}.   
The symmetric monoidal structure, which we have not described here, follows from  \cite[Thm. 2.2]{BaezCourserVasilakopoulou2022}, and is detailed there.
\end{proof}

\section{Feedback loops and homology}
\label{Sec:Feedback_loops}

In system dynamics, graphs labeled by signs are often used to study feedback loops.  We can approach this systematically for graphs labeled by elements of a commutative monoid using homology theory.  Homology lets us detect cycles in a graph.  The homology of a graph with coefficients in an abelian group is well understood, and independent of the direction of the graph's edges, but in fact we can define the homology of a graph with coefficients in a commutative monoid, in a way that depends on the edge directions.  We are most interested in taking coefficients in $\N$, because this gives a way to study feedback loops.

If $C$ is a commutative monoid and $X$ is a set, 
we use $C[X]$ to stand for the set of formal finite linear combinations of elements of $X$ with coefficients in $C$.  Thus, 
\[   C[X] = \left\{ \sum_{x \in X} a_x x \; \Big{\vert} \; a_x \in C, \text{ all but finitely many } a_x \text{ are zero} \right\}.\]
This is a commutative monoid under addition.  In particular, if we make $\N$ into a commutative monoid with addition as the monoid operation, $\N[X]$ is the free commutative monoid on $X$.  

\begin{defn} 
Given a graph $G = (E,V,s,t)$ and a commutative monoid $C$, define
\[    C_0(G,C) = C[V] , \qquad     C_1(G,C) = C[E]  .\]
We call $C_0(G,C)$ (resp.\ $C_1(G,C)$) the commutative monoid of \define{$0$-chains} (resp.\ \define{$1$-chains}) on $G$ with coefficients in $C$.
\end{defn}

The source and target maps of the graph give two maps sending $1$-chains to $0$-chains.   Namely, 
starting from the maps $s,t \maps E \to V$ we can extend them by linearity to monoid homomorphisms 
\[    C[s], C[t] \maps C_1(G,C) \to C_0(G,C). \]
Indeed, there is a functor $C[-] \maps \Set \to \Comm\Mon$.  We have already defined this on objects, and given a map of sets $f \maps X \to Y$ we define $C[f] \maps C[X] \to C[Y]$ by
\[    C[f] \left( \sum_{x \in X} a_x x \right) = \sum_{x \in X} a_x f(x) .\]

When $C$ is an abelian group we can define a group homomorphism $d \maps C_1(G,C) \to C_0(G,C)$ by
\[   d = C[s] - C[t]. \]
In this context the usual expressions for homology groups as quotient groups simplify, and we can define
\[   H_1(G,C) = \ker d , \qquad H_0(G,C) = \coker \, d  .\]
But when working with commutative monoids that are not abelian groups, we cannot form the difference $C[s] - C[t]$, so we use the equalizer and coequalizer of $C[s]$ and $C[t]$ rather than the kernel and cokernel of their difference, defining
\[    H_1(G,C) = \Eq(C[s], C[t]), \qquad
      H_0(G,C) = \coeq(C[s], C[t]) .\]
More formally:

\begin{defn}
Given a graph $G$ and a commutative monoid $C$, we define $H_1(G,C)$, the \define{first homology} of $G$ with coefficients in $C$, to be the equalizer of the morphisms
\[  \begin{tikzcd}
C_1(G,C) \arrow[r, "C\lbrack s\rbrack", shift left = 1.5] \arrow[r, "C\lbrack t\rbrack", swap, shift right = 1.5]  &
C_0(G,C).
\end{tikzcd} 
\]
We call an element of $H_1(G,C)$ a \define{1-cycle} or simply a \define{cycle} with coefficients in $C$. 
We define $H_0(G,C)$, the \define{zeroth homology} of $G$ with coefficients in $C$, to be the coequalizer of $s$ and $t$. 
\end{defn}

We can fit everything into one diagram as follows:      
\[  \begin{tikzcd}
H_1(G,C) \arrow[r, "i"] & C_1(G,C) \arrow[r, "C\lbrack s \rbrack", shift left = 1.5] \arrow[r, "C \lbrack t \rbrack", swap, shift right = 1.5]  &
C_0(G,C)  \arrow[r,"p"]  & H_0(G,C).
\end{tikzcd} 
\]
where $i$ is the inclusion of the equalizer of $C[s]$ and $C[t]$ and $p$ is the projection onto their coequalizer.  When $C$ is an abelian group, we can rewrite this as an exact sequence
\[  0 \to H_1(G,C) \xrightarrow{i} C_1(G,C) \xrightarrow{d} C_0(G,C) \xrightarrow{p} H_0(G,C) \to 0. \]

\begin{eg}
\label{Eg:Homology_of_graphs}
This graph:
\[
G = \begin{tikzcd}
u  \arrow[r, bend left] & v \arrow[l, bend left]
\end{tikzcd}
\]
has $H_1(G,\N) \cong \N$, while this graph:
\[
G' = \begin{tikzcd}
u \arrow[r, bend left] \arrow[r, bend right] & v
\end{tikzcd}
\]
has $H_1(G',\N) \cong \{0\}$.  In simple terms, this says that the first graph offers the opportunity for a feedback loop, while the second does not.  On the other hand, we have $H_1(G,\Z) \cong H_1(G',\Z) \cong \Z$.   We have $H_0(G,\N) \cong H_0(G',\N) \cong \{0\}$ and similarly with $\Z$ coefficients, since both graphs are connected.
\end{eg}

The zeroth homology of a graph with coefficients in a commutative monoid is easily understood.  
There are various concepts of `connected component' for a graph, but define the set of \define{undirected components} of a graph $G$, $\pi_0(G)$, to be the coequalizer
\begin{equation}
\label{Eq:Pi_0_as_coequalizer}
 \begin{tikzcd}
 E \arrow[r, "s", shift left = 1.5] \arrow[r, "t", swap, shift right = 1.5]  & V \arrow[r, "p"] & \pi_0(G) .
\end{tikzcd} 
\end{equation}
Thus, two vertices $v, v'$ lie in the same undirected component iff there is an \define{undirected path} from $v$ to $v'$: a sequence of vertices $v = v_0, \dots, v_n = v'$ such that for each $i =1, \dots ,n$ there is either an edge $e \maps v_{i-1} \to v_i$ or an edge $e \maps v_i \to v_{i-1}$.

\begin{thm}
\label{Prop:Zeroth_homology}
Let $G$ be a graph and $C$ a commutative monoid.  Then $H_0(G,C) \cong C[\pi_0(G)]$.
\end{thm}

\begin{proof}
This can be shown directly, but we sketch a proof using some facts, well known for abelian groups, which have easy analogues for commutative monoids \cite{MeseguerMontanari1990} and even more general algebraic structures \cite{Kock1971}.  First, given $A, B \in \Comm\Mon$, the set of homomorphisms $A \to B$ can be made into a commutative monoid $[A,B]$ using pointwise operations:
\[    (f + g)(a) = f(a) + g(a)  ,  \qquad  \forall f,g \in [A,B], \; a \in A .\]
Second, we can define a tensor product $A \otimes B$ of commutative monoids such that homomorphisms $A \otimes B \to C$ correspond naturally to maps $A \times B \to C$ that are homomorphisms in each argument.    This tensor product obeys hom-tensor adjointness: that is, there is 
a natural isomorphism
\[    \Comm\Mon(A \otimes B, C) \cong \Comm\Mon(B, [A,C]) .\]

For any commutative monoid $C$ there is a natural isomorphism 
\[         C[-] \cong C \otimes \N[-]  \]
of functors from $\Comm\Mon$ to $\Set$.  The functor $\N[-]$ is a left adjoint to the forgetful functor from $\Comm\Mon$ to $\Set$, and the functor $C \otimes -$ is left adjoint to the functor $[C, -] $.  It follows that $C[-] \cong C \otimes \N[-] \maps \Set \to \Comm\Mon$ is a left adjoint, so it preserves colimits.  Applying this functor to the coequalizer diagram \eqref{Eq:Pi_0_as_coequalizer}, we get a coequalizer diagram
\[  \begin{tikzcd}
 C[E] \arrow[r, "C\lbrack s \rbrack", shift left = 1.5] \arrow[r, "C \lbrack t \rbrack", swap, shift right = 1.5]  & C[V] \arrow[r, "C \lbrack t \rbrack"] & C[\pi_0(G)].
\end{tikzcd} 
\]
On the other hand, we have defined $H_0(G,C)$ as the coequalizer
\[  \begin{tikzcd}
 C[E] \arrow[r, "C\lbrack s \rbrack", shift left = 1.5] \arrow[r, "C \lbrack t \rbrack", swap, shift right = 1.5]  & C[V] \arrow[r, "C \lbrack t \rbrack"] & H_0(G,C)
\end{tikzcd} 
\]
though we  have denoted $C[E]$ as $C_1(G,C)$ and $C[V]$ as $C_0(G,C)$.    It follows that $H_0(G) \cong C[\pi_0(G)]$.
\end{proof}

Now let us turn to the first homology of a graph $G$ with coefficients in a commutative monoid.  When $C$ is an abelian group, $H_1(G,C)$ is isomorphic to the homology of the graph $G$ viewed as a topological space, which is well  understood \cite[Sec.\ III.3]{Massey}.    For example,  $H_1(G,\Z)$ is a free abelian group whose rank is the genus of the graph.   More generally, whenever $C$ is an abelian group, the universal coefficient theorem implies
\[              H_1(G,C) \cong C \otimes_{\Z} H_1(G,\Z). \]
The novelty lies in the case when $C$ is \emph{not} a group, since then the directions of the edges matter, as in \cref{Eg:Homology_of_graphs}.   We are then doing a simple sort of `directed algebraic topology'.  We are mainly interested in the case $C = \N$, since this captures the structure of possible feedback loops in the graph.  We now turn toward analyzing the structure of $H_1(G,\N)$.  As we shall see in
 \cref{Eg:Nonfreely_generated_homology}, it is not always a free commutative monoid.

To begin, note that there is a canonical preorder on any commutative monoid $C$, given by
\[   x \le y \iff   x + a = y \text{ for some } a \in C.\]
If $C$ is an abelian group then $x \le y$ for all $x, y \in C$.  But for $C = \N$, the canonical preorder is the usual linear ordering on natural numbers.  More generally, in the free commutative monoid $\N[X]$ on any set $X$ the canonical preorder is given by
\[   \sum_{x \in X} a_x x \le \sum_{x \in X} B_x x  \iff a_x \le B_x \text{ for all } x \in X .\]

\begin{defn} Given a commutative monoid $C$, an element $x \in C$ is \define{minimal} if it is nonzero and $y \le x$ implies that $y = x$ or $y = 0$.
\end{defn}

Any free commutative monoid $\N[X]$ is freely generated by its minimal elements, which correspond to the elements of $X$.  A weaker statement holds for the first homology of a graph with coefficients in $\N$.

\begin{thm}
\label{Thm:generated_by_minimal_elements}
For any graph $G$, $H_1(G,\N)$ is generated by its minimal elements.
\end{thm}

\begin{proof}
First recall that $H_1(G,\N)$ is a submonoid of $C_1(G,\N) \cong \N[E]$ where $E$ is the set of edges of $G$.  The canonical preorder on $H_1(G,\N)$ is a restriction of that in $C_1(G,\N)$, since given $x,y \in H_1(G,\N)$ there exists $z \in C_1(G,\N)$ with $x + z = y$ if and only if
there exists $z \in H_1(G,\N)$ obeying this equation.  As a result, since there are no infinite descending chains of elements $c_1 > c_2 > c_3 > \cdots $ in $C_1(G,\N) \cong \N[E]$, there are no infinite descending chains in $H_1(G,\N)$.

Next we use this to show that any $c \in H_1(G,\N)$ is a sum of minimal elements.  If $c$ is minimal, we are done.  Otherwise we can write $c$ as a sum of two smaller but nonzero elements.  If these are both minimal then we are done; if not, we can express either one that is not minimal as a sum of two smaller nonzero elements.  If we could keep repeating the process of breaking each non-minimal element into a sum of two smaller nonzero elements forever, then $H_1(G,\N)$ would have an infinite descending chain.  But since this is impossible, the process must terminate, expressing $c$ as a finite sum of minimal elements.
\end{proof}

\begin{eg}
\label{Eg:Nonfreely_generated_homology}
Consider the graph
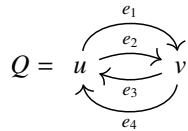
\[
Q = 
\begin{tikzcd}
u  \arrow[r, bend left = 80, "e_1"] 
\arrow[r, bend left = 20, "e_2"]  & 
v \arrow[l, bend left = 20, "e_3"] 
\arrow[l, bend left = 80, "e_4"]
\end{tikzcd}
\]
One can see that the minimal elements of $H_1(Q,\N)$ are the cycles 
\[   a = e_1 + e_3, \;\; 
     b = e_1 + e_4, \;\;
     c = e_2 + e_3, \;\;
     d = e_2 + e_4 . \]
Thus these generate $H_1(Q,\N)$, but not freely, since $a + c = b + d$.   Moreover, note that a free commutative monoid is always freely generated by its minimal elements, since for any set $S$, the minimal elements of the free commutative monoid $\N[S]$ are precisely the linear combinations of elements of $S$ where one coefficient equals $1$ and all the rest are zero.   Thus, $H_1(Q, \N)$ is not a free commutative monoid.
\end{eg}

We call the minimal elements of $H_1(G,\N)$ \define{minimal cycles}.  We are interested in them because they determine the feedback around all cycles in $G$.   Let us make this precise.  For any commutative monoid $C$ and any $C$-labeled graph $(G,\ell)$, there is a map
\[        C_1(G,\N) \to C \]
defined by
\[    \sum_{e \in E} n_e e  \mapsto \sum_{e \in E} n_e \ell(e) .\]
Indeed this is the unique homomorphism of commutative monoids extending $\ell$ from $E$ to the free commutative monoid on $E$, namely $C_1(G,\N)$.   If we restrict this map to cycles we get a homomorphism that we call 
\[   \tilde{\ell} \maps H_1(G,\N) \to C.\]
Given a cycle $c \in H_1(G,\N)$, we call $\tilde{\ell}(c)$ the \define{feedback} of $(G,\ell)$ around $c$.

\begin{cor}
\label{Cor:Feedback_around_minimal_cycles}
For any commutative monoid $C$ and any $C$-labeled graph $(G,\ell)$, the homomorphism $\tilde{\ell} \maps H_1(G,\N) \to C$ is determined by its values on minimal cycles.
\end{cor}

\begin{proof}
This follows from \cref{Thm:generated_by_minimal_elements}: minimal cycles generate $H_1(G,\N)$.  
\end{proof}

While they are defined in an order-theoretic way, minimal cycles have an appealing `geometrical' characterization as homology classes of certain `simple'  loops in $G$.

\begin{defn}
Let $G$ be a graph.  A \define{path} in $G$ is a finite sequence of edges $e_1, \dots e_n$ such that $t(e_i) = s(e_{i+1})$ for $i = 1, \dots, n-1$.   An \define{loop} is a path that ends where it starts, meaning $s(e_1) = t(e_n)$.
\end{defn}

We can denote any path in $G$ as follows:
\[  v_0 \xrightarrow{e_1} v_1 \xrightarrow{e_2} \cdots \xrightarrow{e_{n-1}} v_{n-1} \xrightarrow{e_n} v_n \]
and a loop as follows:
\[  v_0 \xrightarrow{e_1} v_1 \xrightarrow{e_2} \cdots \xrightarrow{e_{n-1}} v_{n-1} \xrightarrow{e_n} v_0 \]
Note from \cref{Sec:Motifs} that the paths in a graph $G$ are precisely the morphisms in the free category on $G$.

Any path $\gamma$ in $G$ gives an element of $[\gamma] \in C_1(G,\N)$, defined to be the sum of that path's edges.  That is, if
\[   \gamma = \left( v_0 \xrightarrow{e_1} v_1 \xrightarrow{e_2} \cdots \xrightarrow{e_{n-1}} v_{n-1} \xrightarrow{e_n} v_n \right)
\]
then we define
\[   [\gamma] = e_1 + \cdots + e_n .\]
If $\gamma$ is a loop then $[\gamma]$ is a cycle since then
\[  s(e_1) + \cdots + s(e_n) = t(e_1) + \cdots + t(e_n) .\]

\begin{defn}
Two loops $\gamma, \delta$ in a graph $G$ are \define{homologous} if $[\gamma] = [\delta]$.
\end{defn}

\begin{defn}
A loop 
\[  v_0 \xrightarrow{e_1} v_1 \xrightarrow{e_2} \cdots \xrightarrow{e_{n-1}} v_{n-1} \xrightarrow{e_n} v_0 \]
is \define{simple} if all the vertices $v_0, \dots, v_{n-1}$ are distinct.
\end{defn}

It is easy to see that if a loop $\gamma$ is not simple, we can chop it where it crosses itself, obtaining two loops $\delta$ and $\eta$ such that 
\[      [\gamma] = [\delta] + [\eta] \]
and $[\delta], [\eta] \ne 0$.  Thus, $[\gamma]$ cannot be minimal if $\gamma$ is not simple.  In fact a much stronger result holds:

\begin{thm}
\label{Thm:simple_loops_and_minimal_cycles}
For any graph $G$, there is a bijection between homology classes of simple loops in $G$ and minimal cycles in $G$, which sends all loops equivalent to the simple loop $\gamma$ to the cycle $[\gamma]$.
\end{thm}

\begin{proof}
This follows from Lemmas \ref{Lem:simple_loops_and_minimal_cycles_1} and \ref{Lem:simple_loops_and_minimal_cycles_2} below.
\end{proof}

\begin{lem}
\label{Lem:simple_loops_and_minimal_cycles_1}
If $\gamma$ is a simple loop in a graph $G$ then $[\gamma]$ is a minimal cycle.
\end{lem}

\begin{proof}
Consider a simple loop
\[  \gamma = \Big( v_0 \xrightarrow{e_1} v_1 \xrightarrow{e_2} \cdots \xrightarrow{e_n} v_0 \Big). \]
Then
\[  [\gamma] = \sum_{i=1}^n e_i .\]
Any cycle $c \le [\gamma] $ must be a chain less than or equal to $c$ in $C_1(G,\N)$, and since $C_1(G,\N)$ is the free commutative monoid on the set of edges of $G$, any chain $c \le [\gamma]$ must be of the form 
\[ c = \sum_{i \in S} e_i  .\]
where $S \subseteq \{1, \dots, n\} $.  We have 
\[ \N[s](c) = \sum_{i \in S} v_{i-1} , \quad \N[t](c) = \sum_{i \in S} v_i .\]
For $c$ to be a cycle we must have $\N[s](c) = \N[t](c)$.  But since all the vertices $v_1, \dots, v_n$ are distinct (while $v_0 = v_n$), the two sums above can only be equal if $S$ is all of $\{1, \dots, n\}$, in which case $c = [\gamma]$, or $S$ is empty, in which case $c = 0$.   (Note that since $C_0(G,\N)$ is free on the set of vertices of $G$, the two sums can only be equal if they are `visibly' equal: there are no extra relations.)  Thus $[\gamma]$ is minimal.  
\end{proof}

\begin{lem}
\label{Lem:simple_loops_and_minimal_cycles_2}
For each minimal cycle $c$ in a graph $G$ there exists a simple loop $\gamma$ in $G$ such that $[\gamma] = c$.
\end{lem}

\begin{proof}
Let $c$ be a minimal cycle.  It is nonzero, so choose an edge with $e_1 \le c$ in $C_1(G,\N)$.  Denote this edge as $v_0 \xrightarrow{e_1} v_1$.  If $v_1 = v_0$ the path $v_0 \xrightarrow{e_1} v_1$ is a loop, say $\gamma$, so $[\gamma] = e_1$ is a nonzero cycle, so by the minimality of $c$ we must have $c = [\gamma]$ and we are done. If on the other hand $v_1 \ne v_0$ then $e_1$ is not a cycle, so $c$ must be the sum of $e_1$ and one or more edges, and at least one of these edges must have source $v_1$, since otherwise it would be impossible to have $s(c) = t(c)$.    Choose one such edge and call it $v_1 \xrightarrow{e_2} v_2$.  We now have a path 
\[   \delta = \Big( v_0 \xrightarrow{e_1} v_1 \xrightarrow{e_2} v_2 \Big) \] 
with $\delta \le c$ and with $v_0, v_1$ distinct.

We continue along these lines by carrying out an inductive procedure.  Assume we have a path 
\[  \delta = \Big( v_0 \xrightarrow{e_1} \cdots \xrightarrow{e_n} v_n \Big) \]
with $[\delta] \le c$ and $v_0, \dots, v_{n-1}$ distinct.  If $v_n$ equals any of the previous vertices, say $v_n = v_i$ with $0 \le i < n$, then we obtain a simple loop 
\[  \gamma = \Big( v_i \xrightarrow{e_{i+1}} \cdots \xrightarrow{e_n} v_n \Big). \] 
Since $[\gamma]$ is a nonzero cycle less than or equal to $[\delta]$ and thus $c$, by the minimality of $c$ we must have $c = [\gamma]$ and we are done.   Otherwise all the vertices $v_0, \dots, v_n$ are distinct, so $[\delta]$ is not a cycle, so $c$ must be a sum of $[\delta]$ and one or more edges, and at least one of these edges must have source $v_n$, since otherwise it would be impossible to have $s(c) = t(c)$.   Choose one such edge and denote it by $v_n \xrightarrow{e_{n+1}} v_{n+1}$.  We now have a path
\[  \delta' = \Big( v_0 \xrightarrow{e_1} v_1 \xrightarrow{e_2} \cdots \xrightarrow{e_n} v_n \xrightarrow{e_{n+1}} v_{n+1} \Big) \]
with $[\delta'] \le c$ and $v_0, \dots , v_n$ distinct.

Since $c$ is a finite sum of edges this procedure must eventually terminate: i.e., eventually $v_k$ must equal one of the previous vertices $v_0, \dots, v_{k-1}$, which are themselves all distinct.  We thus obtain a simple loop $\gamma$ giving a nonzero cycle $[\gamma] \le c$, and by the minimality of $c$ we must have $[\gamma] = c$. 
\end{proof}

Having established a bijection between minimal cycles and homology classes of simple loops, we naturally want to know when two simple loops are homologous.  It turns out that the only way to obtain a loop homologous to a simple loop is to change where the loop starts.  More precisely:

\begin{prop}
\label{Lem:simple_loop_equivalence}
Every loop homologous to a simple loop
\[  v_0 \xrightarrow{e_1} v_1 \xrightarrow{e_2} \cdots \xrightarrow{e_{n-1}} v_{n-1} \xrightarrow{e_n} v_0 \]
is of the form
\[  v_k \xrightarrow{e_{k+1}} v_{k+1} \xrightarrow{e_{k+2}} \cdots \xrightarrow{e_{n+k-1}} v_{n+k-1} \xrightarrow{e_{n+k}} v_k  \]
where we treat the subscripts as elements of $\Z/n$ and do addition mod $n$.
\end{prop}

\begin{proof}
Suppose 
\[ \gamma = \Big( v_0 \xrightarrow{e_1} v_1 \xrightarrow{e_2} \cdots \xrightarrow{e_{n-1}} v_{n-1} \xrightarrow{e_n} v_0 \Big) \]
is a simple loop and 
\[ \delta = \Big( w_0 \xrightarrow{f_1} w_1 \xrightarrow{f_2} \cdots \xrightarrow{f_{m-1}} w_{m-1} \xrightarrow{f_m} w_0 \Big) \]
is a loop homologous to $\gamma$, so 
\[   e_1 + \cdots + e_n 
= f_1 + \cdots + f_m . \]
Since $\gamma$ is simple, all the vertices $v_0, \dots, v_{n-1}$ are distinct, so the edges $e_1, \dots, e_n$ are distinct.  We must thus have $m = n$, with the list of edges $f_1, \dots, f_n$ being some permutation of the list of edges $e_1, \dots, e_n$.  Since all the vertices $v_0, \dots, v_{n-1}$ are distinct, the only permutations that make $\delta$ into a loop are cyclic permutations.
\end{proof}

\section{Emergent feedback loops}
\label{Sec:Emergent_feedback_loops}

When we glue together two graphs, the resulting graph can have loops that are not contained in either of the original graphs.   These are called `emergent loops'.   In applications, these represent new possible feedback loops that arise when we combine two systems.  Since feedback loops are fundamental in system dynamics, it is important to pay close attention to this phenomenon.

There are several ways to study emergent loops.  The first is to study emergent paths in graphs formed by composing open graphs.  We described a double category $\lOpen(\Gph/G_L)$ of open $L$-labeled graphs in \cref{Thm:Open_set-labeled_graphs}, and we can define the double category of \define{open graphs},  $\lOpen(\Gph)$, to be the special case where $L$ is the one-element set, so the labeling becomes trivial.  To compose two open graphs:
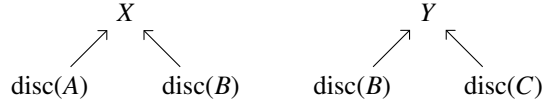
\[
\begin{tikzpicture}[scale=1]
\node (A) at (0,0) {$\disc(A)$};
\node (B) at (1,1) {$X$};
\node (C) at (2,0) {$\disc(B)$};
\node (D) at (4,0) {$\disc(B)$};
\node (E) at (5,1) {$Y$};
\node (F) at (6,0) {$\disc(C)$};
\path[->,font=\scriptsize,>=angle 90]
(A) edge node[above, pos=0.3]{$$} (B)
(C) edge node[above, pos=0.3]{$$}(B)
(D) edge node[above, pos=0.3]{$$} (E)
(F) edge node[above, pos=0.3]{$$}(E);
\end{tikzpicture}
\]
we take their pushout over the discrete graph on the set $B$ defining their shared interface:
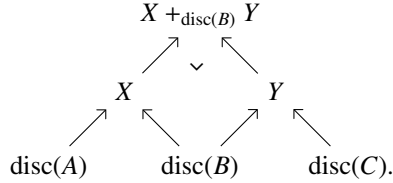
\[
\begin{tikzpicture}[scale=1]
\node (A) at (0,0) {$\disc(A)$};
\node (B) at (1,1) {$X$};
\node (C) at (2,0) {$\disc(B)$};
\node (D) at (3,1) {$Y$};
\node (E) at (4,0) {$\disc(C).$};
\node (F) at (2,2) {$X +_{\disc(B)} Y$};
\node (push) at (2,1.4) {\rotatebox{135}{$\ulcorner$}};
\path[->,font=\scriptsize,>=angle 90]
(A) edge node[above, pos=0.3]{$$} (B)
(C) edge node[above, pos=0.3]{$$}(B)
(C) edge node[above, pos=0.3]{$$} (D)
(E) edge node[above, pos=0.3]{$$}(D)
(B) edge node[above, pos=0.3]{$$} (F)
(D) edge node[above, pos=0.3]{$$}(F);
\end{tikzpicture}
\]
To study paths in the pushout graph, let us simplify the diagram to the relevant part:
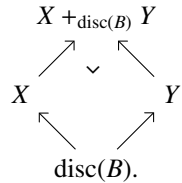
\[
\begin{tikzpicture}[scale=1]
\node (B) at (1,1) {$X$};
\node (C) at (2,0) {$\disc(B).$};
\node (D) at (3,1) {$Y$};
\node (F) at (2,2) {$X +_{\disc(B)} Y$};
\node (push) at (2,1.4) {\rotatebox{135}{$\ulcorner$}};
\path[->,font=\scriptsize,>=angle 90]
(C) edge node[left, pos=0.3]{$ $}(B)
(C) edge node[right, pos=0.3]{$ $} (D)
(B) edge node[above, pos=0.3]{$ $} (F)
(D) edge node[above, pos=0.3]{$ $}(F);
\end{tikzpicture}
\]
Paths in the pushout graph are morphisms in the free category $\Free(X +_{\disc(B)} Y)$.   The left adjoint functor $\Free \maps \Gph \to \cat$ preserves pushouts, so we can apply this functor to the above diagram and obtain the following pushout diagram:
\[
\begin{tikzpicture}[scale=1.2]
\node (B) at (1,1) {$\Free(X)$};
\node (C) at (2,0) {$\Disc(B)$};
\node (D) at (3,1) {$\Free(Y)$};
\node (F) at (2,2) {$\Free(X) +_{\Disc(B)} \Free(Y)$};
\node (push) at (2,1.4) {\rotatebox{135}{$\ulcorner$}};
\path[->,font=\scriptsize,>=angle 90]
(C) edge node[left, pos=0.3]{$ $}(B)
(C) edge node[right, pos=0.3]{$ $} (D)
(B) edge node[above, pos=0.3]{$ $} (F)
(D) edge node[above, pos=0.3]{$ $}(F);
\end{tikzpicture}
\]
since $\Disc(B) \cong \Free(\disc(B))$.  For convenience let us make the abbreviation
\[   \C = \Free(X) +_{\Disc(B)}  \Free(Y) .\]
Then the set of paths from a vertex $v$ of the pushout graph to the vertex $w$ is the homset $\C(v,w)$.  We would like to understand which of these paths are `emergent', i.e., not already paths in $X$ or $Y$.

Let $M\{x,y\}$ be the free monoid on two generators $x$ and $y$.  Then the category $\C$ has a unique $M\{x,y\}$-grading where each edge of $X$ has grade $x$ and each edge of $Y$ has grade $y$.    Let $I\{x,y\}$ be the quotient of the monoid $M\{x,y\}$ by relations saying that $x$ and $y$ are idempotent:
\[                x^2 = x, \quad y^2 = y .\]
By
 \cref{Prop:Change_of_labeling_monoid}, the quotient map  $M\{x,y\} \to I\{x,y\}$ lets us push forward the $M\{x,y\}$-grading on $\C$ to obtain an $I\{x,y\}$-grading on this category.   This has a simplifying effect: now the grade of a path records only how it goes back and forth between the graphs $X$ and $Y$. 

We can write
\[    \C(v,w) = \bigsqcup_{m \in I\{g,h\}}  \C_m(v,w)  \]
where $\C_m(v,w)$ is the set of paths of grade $m$ from $v$ to $w$.    If $v$ and $w$ are vertices in $X$, the paths in $\C_x(v,w)$ are those that already existed in $X$, while all the other paths in $C(v,w)$
are `emergent'.  Similarly, if $v$ and $w$ are vertices in $Y$, all the paths not in $\C_y(v,w)$ are emergent.   The grading gives a measure of the `amount of emergence' involved in each path.   

\begin{eg} 
\label{Eg:emergence}
Suppose $X$ is the graph in red below while $Y$ is the graph in blue, and $B$ is the set of purple vertices:
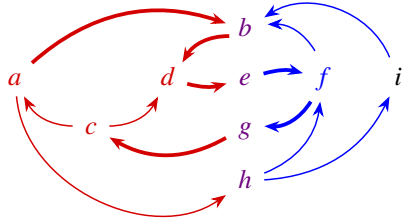
\[
\begin{tikzcd}[column sep = 1.5em, row sep = 0.4em]
  &  &   & \p{b} \arrow[ld, ""', darkred, -{Stealth[length=2mm]}, 
  line width=0.5mm, bend right, shift left] & 
\\
\re{a} \arrow[rrru, darkred, -{Stealth[length=2mm]}, line width=0.5mm,  bend left] 
\arrow[rrrdd, "", darkred, -{Stealth[length=2mm]}, bend right=60] &  & 
\re{d} \arrow[r, darkred, -{Stealth[length=2mm]}, line width=0.5mm, bend right=15] 
 & \p{e}  \arrow[r, "", blue, -{Stealth[length=2mm]}, line width=0.5mm, bend left=15]               
 & \bl{f} \arrow[ld, "", blue, -{Stealth[length=2mm]}, line width=0.5mm, bend left] 
 \arrow[lu, ""', blue, -{Stealth[length=2mm]},  bend right, shift right] 
 & i   \arrow[ull, "", blue, -{Stealth[length=2mm]}, bend right=50]                
 \\
 & \re{c} \arrow[lu, "", darkred, -{Stealth[length=2mm]}, bend left] 
 \arrow[ru, ""', darkred, -{Stealth[length=2mm]}, bend right]   & & 
 \p{g} \arrow[ll, "", darkred, -{Stealth[length=2mm]}, line width=0.5mm, bend left]                                      
 \\
  &  &   & \p{h}  \arrow[uur, "", -{Stealth[length=2mm]}, blue, bend right]   
  \arrow[uurr, "", blue, -{Stealth[length=2mm]}, bend right]                                                                                                             
\end{tikzcd}
\]
If we treat the category $\C$ as $M\{x,y\}$-graded, the path in bold from $a$ to $c$ has grade $x^3 y^2 x$, if we adopt the convention of multiplying elements in $M\{x, y\}$ from left to right as we move along the path.
If we switch to treating $\C$ as $I\{x,y\}$-graded, the same path has grade $xyx \in I\{x,y\}$.  This indicates that the path starts in $X$, then goes into $Y$, and then comes back into $X$.
\end{eg}

Another approach to emergence involves understanding how the first homology of the pushout graph $X +_{\disc(B)} Y$ is related to that of the graphs $X$ and $Y$.     For this let us assume that the maps $j$ and $k$ in the pushout diagram
\[
\begin{tikzpicture}[scale=1]
\node (B) at (1,1) {$X$};
\node (C) at (2,0) {$\disc(B)$};
\node (D) at (3,1) {$Y$};
\node (F) at (2,2) {$X +_{\disc(B)} Y$};
\node (push) at (2,1.4) {\rotatebox{135}{$\ulcorner$}};
\path[->,font=\scriptsize,>=angle 90]
(C) edge node[left, pos=0.3]{$j$}(B)
(C) edge node[right, pos=0.3]{$k$} (D)
(B) edge node[left, pos=0.5]{$i_X$} (F)
(D) edge node[right, pos=0.5]{$i_Y$}(F);
\end{tikzpicture}
\]
are monic.   It follows that 
all the arrows in the diagram are monic, so $X$ and $Y$ are subgraphs of 
$X +_{\disc(B)} Y$ having no edges in common, and their intersection is the discrete graph $\disc(B)$.   
This is the case in \cref{Eg:emergence}.   To simplify the notation further, let us define
\[        X \cup Y = X +_{\disc(B)} Y, \qquad X \cap Y = \disc(B).   \]

In this situation, the first homology with coefficients in an abelian group $A$ is well-understood.  Here the classical Mayer--Vietoris sequence \cite{Hatcher2009} reduces to the following exact sequence of abelian groups:
\[   0 \to H_1(X,A)  \oplus H_1(Y,A) \xrightarrow{} H_1(X \cup Y, A) \xrightarrow{} \]
\[ H_0(X \cap Y,A)  \xrightarrow{} H_0(X,A) \oplus H_0(Y,A) \xrightarrow{} H_0(X \cup Y, A)
\xrightarrow{}  0 .\]
The abelian group $H_1(X,A) \oplus H_1(Y,A)$ consists of \define{non-emergent cycles}, 
which existed in $X$ and $Y$ before we glued these  graphs together along the vertices in $B$.
It is a subgroup of the group of actual interest, $H_1(X \cup Y)$.   The quotient of $H_1(X \cup Y)$ by this subgroup may thus be seen as the group of `emergent' 1-cycles. 

We thus focus on this portion of the Mayer--Vietoris exact sequence:
\begin{equation}
\label{Eq:Mayer--Vietoris}
0 \to H_1(X,A)  \oplus H_1(Y,A) \xrightarrow{\iota} H_1(X \cup Y, A) \xrightarrow{\partial} 
 H_0(X \cap Y,A) 
\end{equation}
and define the group of \define{emergent 1-cycles} to be 
\[  \displaystyle{  \textrm{coker} \, \iota =  \frac{H_1(X \cup Y, A)}{\text{im} \, \iota} = \frac{H_1(X \cup Y, A)}{\ker \partial} } . \]
There is a short exact sequence expressing the group of 1-cycles on $X \cup Y$ as an extension of the group of emergent 1-cycles by the group of non-emergent 1-cycles:
\[  0 \to H_1(X,A)  \oplus H_1(Y,A) \xrightarrow{\iota} H_1(X \cup Y, A) \to \textrm{coker} \, \iota \to 0.\]

To go further we need explicit formulas for $\iota$ and $\partial$:
 \begin{itemize}
 \item $\iota = \iota_X + \iota_Y$ where $\iota_X \maps H_1(X,A) \to H_1(X \cup Y, A) $ is the map induced from the inclusion $i_X \maps X \to X \cup Y$, and similarly $\iota_Y \maps H_1(Y,A) \to H_1(X \cup Y, A)$ is the map induced from $i_Y$.
 \item $\partial$ may be defined as follows.  We can uniquely express $c \in H_1(X \cup Y,A) \subseteq C_1(X \cup Y,A)$ as a sum of 1-chains 
 \[   c_X \in C_1(X,A) \subseteq C_1(X \cup Y,A), \qquad
  c_Y \in C_1(Y,A) \subseteq C_1(X \cup Y,A) .\]   
 These 1-chains are generally not themselves 1-cycles, but their boundaries sum to zero, and the boundary of $c_X$ lies in $C_1(X,A)$ while that of $c_Y$ lies in $C_1(Y,A)$,  so
 \[    dc_X = -dc_Y \in C_0(X \cap Y,A) \subseteq C_0(X \cup Y,A). \]
 We thus define $\partial c$ by
 \[  \partial c = dc_X = -dc_Y \in C_0(X \cap Y,A) = H_0(X \cap Y,A) .\]
 \end{itemize}
Our choice to define $\partial c$ to be $d c_X$ rather than $d c_Y$ is an arbitrary sign convention.

Next we need to adapt these ideas to homology with coefficients in a commutative monoid $C$.  A case of special interest is $C = \N$, since the first homology of a graph with coefficients in $\N$ controls feedback as in \cref{Cor:Feedback_around_minimal_cycles}, and we have a fairly clear picture of this first homology thanks to Theorems \ref{Thm:generated_by_minimal_elements} and \ref{Thm:simple_loops_and_minimal_cycles}. 

We continue to assume we have a union of graphs whose intersection is a discrete graph, i.e., a graph with no edges:
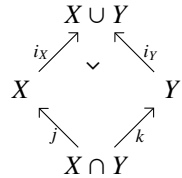
\[
\begin{tikzpicture}[scale=1]
\node (B) at (1,1) {$X$};
\node (C) at (2,0) {$X \cap Y$};
\node (D) at (3,1) {$Y$};
\node (F) at (2,2) {$X \cup Y$};
\node (push) at (2,1.4) {\rotatebox{135}{$\ulcorner$}};
\path[->,font=\scriptsize,>=angle 90]
(C) edge node[left, pos=0.3]{$j$}(B)
(C) edge node[right, pos=0.3]{$k$} (D)
(B) edge node[left, pos=0.5]{$i_X$} (F)
(D) edge node[right, pos=0.5]{$i_Y$}(F);
\end{tikzpicture}
\]
It is notationally convenient to treat $H_1(X,C) \subseteq C_1(X,C), H_1(Y,C) \subseteq  C_1(Y,C)$ and $H_1(X \cup Y,C)$ as submonoids of $C_1(X \cup Y,C)$.   Since every edge of $X \cup Y$ is either an edge of $X$ or of $Y$, but not both, we have
\[        C_1(X \cup Y,C) = C_1(X,C) \oplus C_1(Y,C)   \]
where we write an equals sign because this is an `internal direct sum': every element $c \in C_1(X \cup Y,C)$ can be uniquely written as a sum of elements $c_X \in C_1(X,C)$ and $c_Y \in C_1(Y,C)$.   We define monoid homomorphisms $p_X, p_Y \maps C_1(X \cup Y, C) \to  C_1(X \cup Y,C)$ by
\[   p_X(c) = c_X, \qquad p_Y(c) = c_Y . \]
It is also convenient to treat $C_0(X,C), C_0(Y,C)$ and $C_0(X \cap Y,C)$ as submonoids of $C_0(X \cup Y,C)$.
We abbreviate
\[        C[s], C[t] \maps C_1(X \cup Y, C) \to C_0(X \cup Y,C) \]
simply as $s$ and $t$.  We can also use these notations for the maps
\[        C[s], C[t] \maps C_1(X) \to C_0(Y) \]
and
\[        C[s], C[t] \maps C_1(Y) \to C_0(Y) \]
without confusion, since these are restrictions of the maps $s$ and $t$ defined on all of $C_1(X \cup Y,C)$.

The natural map from the coproduct $X + Y$ to the pushout $X \cup Y$ induces a map on homology
\[       \iota \maps H_1(X , C) \oplus H_1(Y,C) \to H_1(X \cup Y, C) .\] 
This map $\iota$ sends any pair $(c_X, c_Y)$ to the sum $c_X + c_Y$.  This map is a monomorphism, but perhaps not an isomorphism, since there may be emergent cycles.  The following Mayer--Vietoris-like theorem clarifies the situation.  The theorem simplifies when $C$ is \define{cancellative}, meaning that 
\[     c + e = d + e \implies c = d  .\]

\begin{thm}
\label{Thm:Mayer--Vietoris_1}
If $C$ is a commutative monoid and $X,Y$ are subgraphs of a graph $X \cup Y$ whose intersection
$X \cap Y$ is a discrete graph, then the following is an equalizer diagram in the category of commutative monoids:
\[    
\begin{tikzcd}
     H_1(X,C) \oplus H_1(Y,C) \arrow[r, "\iota"] &
     H_1(X \cup Y, C)  \arrow[r, "{(s p_X, s p_Y)}", shift left = 1.5] \arrow[r, "{(t p_X, t p_Y)}", swap, shift right = 1.5]  &
     C_0(X, C) \oplus C_0(Y, C).
\end{tikzcd}
 \]
 If $C$ is cancellative, the following diagram is also an equalizer:
\[    
\begin{tikzcd}
     H_1(X,C) \oplus H_1(Y,C) \arrow[r, "\iota"] &
     H_1(X \cup Y, C)  \arrow[r, "s p_X", shift left = 1.5] \arrow[r, "t p_X", swap, shift right = 1.5]  &
     C_0(X, C).
\end{tikzcd}
 \]
 as is the analogous diagram with $p_X$ replaced by $p_Y$.
\end{thm}

\begin{proof}
For the first equalizer, first we show that $sp_X = tp_X$ and $s p_Y = t p_Y$ on the image of $\iota$.   Any element in the image of $\iota$ is of the form $c = c_X + c_Y$ with $c_X \in H_1(X,C)$ and $c_Y \in H_1(Y,C)$.    Since $p_X c_X = c_X$ and $p_X c_Y = 0$, we have
\[    s p_X c  = s c_X = tc_X = t p_X c.  \]
Similarly we have $s p_Y c = t p_Y c$.

Next we show that any $c \in H_1(X \cup Y, C)$ with $s p_X c = t p_X c$ and $s p_Y c = t p_Y c$ is in the image of $\iota$.  These equations say that $p_X c \in H_1(X, C)$ and $p_Y c \in H_1(Y, C)$, so $c = p_X c + p_Y c = \iota(p_X c, p_Y c)$.  

We can reduce the second equalizer diagram to the first if we show  $sp_X c= tp_Xc$ implies $s p_Yc = t p_Yc$ for any $c \in H_1(X \cup Y,C)$.   Here we need $C$ to be cancellative.   Since $c$ is a cycle we have $s c = t c$ and thus
\[    s p_X c + s p_Y c = t p_X c + t p_Y c .\]
Since $C$ is cancellative so is $C_0(X \cup Y,C)$, so we can subtract the equation $s p_X c = t p_X c$
from the above equation and conclude
\[    s p_Y c = t p_Y c . \]
The diagram with $p_X$ replaced by $p_Y$ works the same way.
\end{proof}

\begin{eg}
If $C$ is not cancellative, it is possible to have a cycle $c_X \in H_1(X,C)$ and a 1-chain that is not a cycle, $c_Y \in C_1(Y,C)$, whose sum is a cycle $c_X  + c_Y \in H_1(X \cup Y, C)$.  
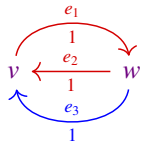
\[
\begin{tikzcd}[column sep = 3em]
\p{v} \arrow[r, bend left = 80, darkred, "e_1", "1"'] 
& 
\p{w} \arrow[l, darkred, "1", "e_2"'] 
\arrow[l, bend left = 80, blue, "1", "e_3"']
\end{tikzcd}
\]
Above is a $\lB$-labeled graph where $\lB = \{0,1\}$ with addition as in \cref{Eg:B}.   The subgraph $X$ contains only the red edges, while the subgraph $Y$ contains only the blue edges.  The vertices lie in $X \cap Y$, which is the discrete graph on these two vertices.   Let 
\[     c_X = e_1 + e_2, \qquad c_Y = e_3 .\]
Then $c_X$ is a cycle, $c_Y$ is not a cycle, and $c_X + c_Y$ is a cycle because
\[   s(c_X + c_Y) = s(e_1 + e_2 + e_3) = v + w + w = v + w, \qquad t(c_X + c_Y) = t(e_1 + e_2 + e_3) = w + v + v = v + w .\]
\end{eg}

 We can state a Mayer--Vietoris theorem that more closely resembles Equation \eqref{Eq:Mayer--Vietoris} if we use the monoid homomorphism
\[          q \maps C_0(X,C) \to C_0(X \cap Y,C)  \]
given by
\[          q \left( \sum_{v \in \{\text{vertices of } X \}} c_v \, v \right) = \sum_{v \in \{\text{vertices of } X \cap Y\}} c_v \, v  .\]
That homomorphism kills off all vertices of $X$ that are not also in $Y$.   A priori $H_0(X \cap Y,C)$ is a quotient of $C_0(X \cap Y, C)$, but $X \cap Y$ has no edges so the quotient map is an isomorphism $C_0(X \cap Y,C) \simrightarrow H_0(X \cap Y, C) $.  We shall use this isomorphism to identify these two monoids, and treat $q$ as a monoid homomorphism
\[          q \maps C_0(X,C) \to H_0(X \cap Y,C).  \]

\begin{thm}
\label{Thm:Mayer--Vietoris_2}
If $C$ is a cancellative commutative monoid and $X,Y$ are subgraphs of a graph $X \cup Y$ whose intersection
$X \cap Y$ is a discrete graph, then the following is an equalizer diagram in the category of commutative monoids:
\[    
\begin{tikzcd}
     H_1(X,C) \oplus H_1(Y,C) \arrow[r, "\iota"] &
     H_1(X \cup Y, C)  \arrow[r, "q s p_X", shift left = 1.5] \arrow[r, "q t p_X", swap, shift right = 1.5]  &
     H_0(X \cap Y, C).
\end{tikzcd}
 \]
 By the symmetry between $X$ and $Y$, the analogous diagram with $p_X$ replaced by $p_Y$ is also an equalizer.
\end{thm}

\begin{proof}
By \cref{Thm:Mayer--Vietoris_1} it suffices to show that $c \in H_1(X \cup Y,C)$ has $s p_X c = t p_X c$ if and only if $q s p_X c = q t p_X c$.  One direction of the implication is obvious, so we suppose $q s p_X c = q t p_X c$
and aim to show that $s p_X c = t p_X c$.

We write
 \[           c = \sum_{e \in \{\text{edges of } X \cup Y\}} c_e e .\]
 Since $c$ is a cycle we have
 \[   \sum_{e \in \{\text{edges of } X \cup Y\}} c_e s(e) =  \sum_{e \in \{\text{edges of } X \cup Y\}} c_e t(e)  .\]
There are three mutually exclusive choices for a vertex in $X \cup Y$: it is either 1) in $X$ but not $Y$, 2) in $X \cap Y$, or 3) $Y$ but not in $X$.    In case 1) we say the vertex is in $X - Y$ and in case 3) we say the vertex is in $Y - X$, merely by way of abbreviation.  The above equation thus implies three equations, of which the important one is the first:
  \[   \sum_{e \in \{ \text{edges of } X \cup Y \text{ whose source is in } X - Y\} } \!\!\!\!\!\!\!\!\!\!\!\!\!\!\!\!\!\!\!\!\!\!\!\!\! c_e s(e) \quad \quad \quad =  \quad \sum_{e \in \{\text{edges of } X \cup Y  \text{ whose target is in } X - Y\}} \!\!\!\!\!\!\!\!\!\!\!\!\!\!\!\!\!\!\!\!\!\!\!\!\! c_e t(e)  \]
 Since an edge of $X \cup Y$ whose source is in $X - Y$ must be an edge of $X$, this equation is equivalent to
   \[   \sum_{e \in \{ \text{edges of } X \text{ whose source is in } X - Y\} } \!\!\!\!\!\!\!\!\!\!\!\!\!\!\!\!\!\!\!\!\!\!\!\!\!  c_e s(e) \quad\quad\quad = \quad \sum_{e \in \{\text{edges of } X  \text{ whose target is in } X - Y\}} \!\!\!\!\!\!\!\!\!\!\!\!\!\!\!\!\!\!\!\!\!\!\!\!\!  c_e t(e)  .\]
Since $q s p_X c  = q t p_X c$ we also know that
 \[     \sum_{e \in \{\text{edges of } X \text{ whose source is in } X \cap Y\}} \!\!\!\!\!\!\!\!\!\!\!\!\!\!\!\!\!\!\!\!\!\!\!\!\!  \quad c_e s(e) \quad\quad\quad = \quad \sum_{e \in \{\text{edges of } X \text{ whose target is in } X \cap Y\}} \!\!\!\!\!\!\!\!\!\!\!\!\!\!\!\!\!\!\!\!\!\!\!\!\! c_e t(e) .\]
 Adding the last two equations we get
 \[     \sum_{e \in \{\text{edges of } X\}} c_e s(e) =  \sum_{e \in \{\text{edges of } X\}} c_e t(e) .\]
This says $s p_X c = t p_X c$, as desired. 
\end{proof}

\section{Conclusion}

We hope these results can be used to further the use of causal loop diagrams---that is, $\{+,-\}$-labeled graphs---and other monoid-labeled graphs in system dynamics, systems biology and other related fields.     Software already exists that explicitly uses category theory to work with causal loop diagrams and other monoid-labeled graphs.   For example, AlgebraicJulia is a software platform for doing scientific computing with categories, and it has a package called StockFlow.jl \cite{BaezLiLibkindOsgoodPatterson2023,StockFlow} that allows for the creation and composition of open causal loop diagrams \cite{BaezLiLibkindOsgoodRedekopp2023}.   Redekopp has created a program called ModelCollab, based on StockFlow.jl, whose web-based interface enables collaborative work with causal loop diagrams \cite{ModelCollab}.   Separately, Patterson and collaborators are building a software framework called CatColab, based on double categories, and they have used it to create tools for finding motifs in $M$-labeled graphs when:
 \begin{itemize}
 \item
 $M = \{+,-\}$: causal loop diagrams \cite{CatColabCausalLoop}.
 \item
 $M = \{+,-\} \times \lB$: causal loop diagrams with boolean-valued delays \cite{CatColabBoolean}.
 \item
  $M = \{+,-,0\}$: causal loop diagrams with indeterminate effects \cite{CatColabIndeterminate}.
\end{itemize}
We hope that some more ideas from this paper can be embodied in software.

On the other hand, many large databases of labeled graphs have already been compiled, some listed in Pathguide \cite{Pathguide}.  A visually charismatic example is the KEGG database of pathways for metabolism, biological information processing, and other biological processes \cite{KEGG2025}.  There is a great opportunity to apply new mathematics and software to analyze these labeled graphs.   In addition, new mathematics can help catalyze the interaction between systems biology and system dynamics \cite{Sterman2002}.

There are also many mathematical questions to study.  We briefly mention two.  First,  in \cref{Sec:Rig-labeled_graphs} we tentatively explored rig-labeled graphs.  Given a rig $R$, it might be interesting to develop a category of $R$-labeled graphs where several edges `in series' can be mapped to a single edge with their labels being multiplied:
\[                      v_0 \xrightarrow{r_1} v_1 \xrightarrow{r_2} \cdots \xrightarrow{r_k} v_k \qquad \qquad \mapsto   \qquad \qquad         v_0 \xrightarrow{r_1 \cdots r_k} v_k\]
and also several edges `in parallel' can be mapped to a single edge with their labels being added:
\[
\begin{tikzcd}
v \arrow[rr, "r_1", bend left = 60] \arrow[rr, "r_2", bend left = 30]  \arrow[rr, "r_k"', bend right]  & \vdots & w && \mapsto &&
v \arrow[rr,  "r_1 + \cdots + r_k"]   && w .
\end{tikzcd}
\]
We have seen that the former pattern shows up in the \emph{opposite} of the category of monoid-labeled graphs and Kleisli morphisms, while the latter shows up in the category of commutative-monoid labeled graphs and additive morphisms.

Second, it would be interesting to understand which commutative monoids arise as $H_1(G,\N)$ for some finite graph $G$.  As seen in \cref{Eg:Nonfreely_generated_homology}, not all of them are free.  The example there has a presentation with four minimal cycles as generators, one relation, and no syzygies (that is, relations between relations).   The following graph 
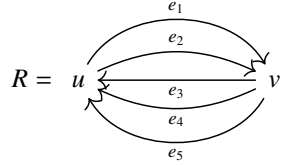
\[
R = 
\begin{tikzcd}[column sep = 3 em, row sep = 3 em]
u  \arrow[rr, bend left = 60, "e_1"] 
\arrow[rr, bend left = 25, "e_2"]  & &
v \arrow[ll, bend left = 0, "e_3"] 
\arrow[ll, bend left = 25, "e_4"]
\arrow[ll, bend left = 60, "e_5"]
\end{tikzcd}
\]
has $H_1(R,\N)$ with 6 minimal cycles as generators, 3 relations, and one syzygy.  This suggests that it might be interesting to study the homology of the commutative monoids $H_1(G,\N)$ for finite graphs $G$.    (For a simple introduction to the homology of monoids see Lafont \cite[Sec. 4]{LafontProute1991}.)  More generally, it suggests that the homology $H_1(G,C)$ with coefficients in a commutative monoid $C$ is worth studying more deeply.

\end{document}